%% file: main.tex
\documentclass[10pt, a4paper, notitlepage]{article}
\pdfoutput=1

\usepackage{biblatex}
\input{packages.tex}
\usepackage{hyperref}
\input{style_general.tex}
\input{notation_general.tex}

\newcommand{\paperone}{\cite{Paper-I}}
\newcommand{\papertwoB}{\cite{Paper-IIB}}

\title{$ A_∞ $-Deformations and their Derived Categories}
\author{Jasper van de Kreeke}
\date{15 March 2023}

\begin{document}

\maketitle

\begin{abstract}
Handling curved $ A_∞ $-deformations is challenging and defining their derived categories seems impossible. In this paper, we show how to welcome the curvature and build derived categories despite the apparent difficulties. We explicitly describe their $ A_\infty $-structure by a deformed Kadeishvili theorem.

This paper starts with a review of $ A_∞ $-deformations for non-specialists. We explain why $ A_∞ $-deformations require curvature and why curvature supposedly causes problems. We prove that infinitesimally curved $ A_∞ $-deformations are in fact non-problematic and provide the correct definitions for their twisted completion and minimal model.

The classical Kadeishvili theorem collapses for curved deformations. We show how to adapt it by iteratively optimizing the curvature and changing the homological splitting infinitesimally. The final $ A_∞ $-structure on the minimal model is constructed in terms of trees.
\end{abstract}

\newcommand{\papertwoAacknowledgements}{%
\subsection*{Acknowledgements}
This paper is part of the author's PhD thesis supervised by Raf Bocklandt. Crucial insights were gained at the 2020 Summer Camp on Derived Categories, Stability Conditions and Deformations, which the author organized in his garden in Berlin. The author thanks all participants for sharing their knowledge, in particular Severin Barmeier for advice and expertise on deformation theory through the lens of $ L_\infty $-algebras. The author's PhD project was supported by the NWO grant “Algebraic methods and structures in the theory of Frobenius manifolds and their applications” (TOP1.17.012).}

\tableofcontents

\input{filetree.tex}

\printbibliography

\end{document}

%% file: packages.tex
\usepackage{bbm}

\usepackage{etoolbox}
\usepackage{tikz}
\usetikzlibrary{calc}
\usetikzlibrary{cd}
\usetikzlibrary{decorations.markings}
\usetikzlibrary{decorations.pathreplacing}
\usetikzlibrary{decorations.pathmorphing}
\usetikzlibrary{decorations.text}
\usetikzlibrary{arrows.meta}
\usetikzlibrary{arrows}
\usetikzlibrary{positioning}

\usepackage{geometry} 
\usepackage{tocloft} 
\usepackage{fancyhdr} 
\usepackage{longtable}
\usepackage{subcaption}
\usepackage{enumitem}
\usepackage{amssymb}
\usepackage{amsmath}
\usepackage{stackrel} 
\usepackage{uniinput}
\usepackage{amsthm}
\usepackage{stmaryrd}
\usepackage{mathtools}
\usepackage{fouridx}
\usepackage{xparse}
\usepackage{bm}
\usepackage{aliascnt}
\usepackage{import}

%% file: style_general.tex
\theoremstyle{definition}
\newtheorem{theorem}{Theorem}
\let\emph\textbf

\newaliascnt{remark}{theorem}
\newaliascnt{definition}{theorem}
\newaliascnt{proposition}{theorem}
\newaliascnt{lemma}{theorem}
\newaliascnt{corollary}{theorem}
\newaliascnt{example}{theorem}
\newaliascnt{convention}{theorem}
\newaliascnt{todo}{theorem}

\newtheorem{remark}[remark]{Remark}
\newtheorem{definition}[definition]{Definition}

\newtheorem{lemma}[lemma]{Lemma}
\newtheorem{corollary}[corollary]{Corollary}
\newtheorem{example}[example]{Example}
\newtheorem{convention}[convention]{Convention}

\aliascntresetthe{remark}
\aliascntresetthe{definition}
\aliascntresetthe{proposition}
\aliascntresetthe{lemma}
\aliascntresetthe{corollary}
\aliascntresetthe{example}
\aliascntresetthe{convention}
\aliascntresetthe{todo}

\numberwithin{equation}{section}
\numberwithin{figure}{section}
\numberwithin{table}{section}
\makeatletter\let\c@table\c@figure\makeatother

\numberwithin{theorem}{section}
\numberwithin{remark}{section}
\numberwithin{definition}{section}
\numberwithin{proposition}{section}
\numberwithin{lemma}{section}
\numberwithin{corollary}{section}
\numberwithin{example}{section}
\numberwithin{convention}{section}
\numberwithin{todo}{section}

\DeclareFieldFormat*{url}{}
\DeclareFieldFormat*{doi}{}
\DeclareFieldFormat*{issn}{}
\DeclareFieldFormat[misc]{url}{\mkbibacro{URL}\addcolon\space\url{#1}}
\DeclareFieldFormat[online]{url}{\mkbibacro{URL}\addcolon\space\url{#1}}

\setlist{topsep=0.5em, itemsep=0em}
\renewcommand{\arraystretch}{1.3}
\newcommand{\morearraystretch}{\renewcommand{\arraystretch}{1.8}}
\SetLabelAlign{parleftalign}{\parbox[t]\textwidth{#1\par\mbox{}}}


%% file: notation_general.tex
\newcommand{\Id}{\operatorname{Id}}
\newcommand{\id}{\mathrm{id}}

\newcommand{\Hom}{\operatorname{Hom}}
\newcommand{\Ker}{\operatorname{Ker}}
\newcommand{\End}{\operatorname{End}}
\newcommand{\Ob}{\operatorname{Ob}}
\newcommand{\cat}[1]{\mathcal{#1}}

\newcommand{\htensor}{\widehat{\otimes}}

\newcommand{\Tw}{\operatorname{Tw}}

\newcommand{\Add}{\operatorname{Add}}
\newcommand{\MC}{\operatorname{MC}}
\newcommand{\MCb}{\overline{\operatorname{MC}}}
\newcommand{\Gtl}{\operatorname{Gtl}}

\newcommand{\D}{\operatorname{D}}

\def\H{\operatorname{H}}

\newcommand{\HH}{\operatorname{HH}}
\newcommand{\HC}{\operatorname{HC}}

\newcommand{\pmat}[1]{\begin{pmatrix} #1 \end{pmatrix}}

\newcommand{\odd}{\operatorname{odd}}
\newcommand{\even}{\operatorname{even}}
\newcommand{\isoto}{\xrightarrow{\sim}}
\newcommand{\verylongisoto}{\xrightarrow{~~~\sim~~~}}
\newcommand{\verylongmapsto}{\xmapsto{~~~\phantom{\sim}~~~}}
\newcommand{\verylongto}{\xrightarrow{~~~\phantom{\sim}~~~}}
\newcommand{\embeds}{\hookrightarrow}
\def\Im{\operatorname{Im}}

\newcommand{\restr}[1]{|_{#1}}

\newcommand{\running}{~|~}

\newcommand{\Res}{\operatorname{Res}}
\newcommand{\opp}{\mathsf{opp}}

\newcommand{\landau}{\mathcal{O}}
\newcommand{\bigmax}{\bm{\mathrm{opt}}}

\newcommand{\projects}{\twoheadrightarrow}

\NewDocumentCommand{\tensor}{t_}
 {%
  \IfBooleanTF{#1}
   {\tensorop}
   {\otimes}%
 }
\NewDocumentCommand{\tensorop}{m}
 {%
  \mathbin{\mathop{\otimes}\displaylimits_{#1}}%
 }


%% file: filetree.tex
\input{toplevel/intro.tex}

\input{ainfty/ainfty.tex}

\input{defo/intro.tex}
\input{defo/htensor.tex}
\input{defo/defo.tex}
\input{defo/twisted.tex}
\input{defo/functors.tex}

\input{dgla/intro.tex}
\input{dgla/hochschild.tex}
\input{dgla/gauge.tex}
\input{dgla/Linfty.tex}
\input{dgla/axioms.tex}

\input{constructions/intro.tex}
\input{constructions/twisted.tex}
\input{constructions/minmodel.tex}
\input{constructions/derived.tex}
\input{constructions/uncurving.tex}

\input{kadeishvili/intro.tex}
\input{kadeishvili/splitting.tex}
\input{kadeishvili/classical.tex}
\input{kadeishvili/deformed.tex}
\input{kadeishvili/optimizing.tex}
\input{kadeishvili/auxiliary.tex}
\input{kadeishvili/general.tex}
\input{kadeishvili/caseD.tex}

%% file: toplevel/intro.tex
\section{Introduction}
$ A_∞ $-categories are a tool to store homotopical information in a category. For instance, they are the standard method of homological mirror symmetry to capture the higher products on Fukaya categories and categories of coherent sheaves. From this background, mathematicians have also introduced the notion of $ A_∞ $-deformations. A prominent example is Seidel's relative Fukaya category \cite{Seidel-relative}, a deformation of the ordinary Fukaya category. As Seidel observed, the relative Fukaya category must be defined as a so-called curved $ A_∞ $-deformation.

Many constructions are available for $ A_∞ $-categories. For instance, an $ A_∞ $-category $ \cat C $ has a twisted completion $ \Tw\cat C $ and minimal model $ \H\cat C $. The derived category of $ \cat C $ is the $ A_∞ $-category $ \H\Tw\cat C $, the minimal model of the twisted completion of $ \cat C $. A crucial necessity for defining derived categories is that the differential of $ \cat C $ squares to zero.

\begin{center}
\begin{tikzpicture}
\path (0, 0) node[align=center] (A) {$ A_∞ $-category $ \cat C $} (9, 0) node[align=center] (B) {Deformed $ A_∞ $-category $ \cat C_q $};
\path (0, -1.5) node[align=center] (C) {Derived category $ \H\Tw\cat C $};
\path[draw, -, dashed] ($ (A.east) + (right:1) $) -- ($ (B.west) + (left:1) $) node[midway, above] {deformation};
\path[draw, <->] (A.south) -- (C.north) node[midway, left] {derived} node[midway, right] {equivalent};
\end{tikzpicture}
\end{center}

At first sight, it may seem impossible to fill this square and define a derived category for $ \cat C_q $ which is a deformation of $ \H\Tw\cat C $. The problem is that the deformed differential of $ \cat C_q $ need not square to zero because of the curvature. In the present paper, we solve the dilemma and show how to correctly define the twisted completion $ \Tw\cat C_q $ and minimal model $ \H\cat C_q $ of any curved $ A_∞ $-deformation. We define the derived category of $ \cat C_q $ as the minimal model of the twisted completion, effectively filling the square:

\begin{center}
\begin{tikzpicture}
\path (0, 0) node[align=center] (A) {$ A_∞ $-category $ \cat C $} (9, 0) node[align=center] (B) {Deformed $ A_∞ $-category $ \cat C_q $};
\path (0, -1.5) node[align=center] (C) {Derived category $ \H\Tw\cat C $} (9, -1.5) node (D) {\textbf{Deformed derived category $ \H\Tw\cat C_q $}};
\path[draw, dashed] ($ (A.east) + (right:1) $) -- ($ (B.west) + (left:1) $) node[midway, above] {deformation};
\path[draw, dashed] ($ (C.east) + (right:1) $) -- ($ (D.west) + (left:1) $) node[midway, above] {deformation};
\path[draw, <->] (A.south) -- (C.north) node[midway, left] {derived} node[midway, right] {equivalent};
\path[draw, <->] (B.south) -- (D.north) node[midway, left] {derived} node[midway, right] {equivalent};
\end{tikzpicture}
\end{center}

The literature includes attempts to avoid this “curvature problem” \cite{Keller-curved}, for instance Lowen and Van den Bergh \cite{Lowen-vdB} investigate how to get rid of curvature by passing to a Morita equivalent model of $ \cat C $, and Positselski \cite{Positselski} shows how to define an alternative notion of derived categories for curved $ A_∞ $-categories. However, the approaches do not lead to a notion of derived category that is a deformation of $ \H\Tw\cat C $.

\subsection*{Context}
To the benefit of the reader, we summarize here the background of $ A_∞ $-theory. We comment on $ A_∞ $-deformations and the problems one encounters when trying to build their derived categories. We also explain how the results of the present paper are used to conduct an explicit minimal model calculation in \papertwoB.

\paragraph*{$ A_∞ $-categories and their deformations}
$ A_∞ $-categories are a type of category used to store homotopical information. In comparison with the better-known dg categories, they allow for more flexibility. While dg categories come with only two operations, the differential and the composition, an $ A_∞ $-category allows room for higher products $ μ^3, μ^4, … $. The flexibility entails that many different $ A_∞ $-categories are equivalent to a given one, varying in size and amount of nonzero higher products.

Deformation theory is a means of capturing all ways to algebraically or geometrically deform objects. The language of $ A_∞ $-categories makes it possible to define immediately an algebraic deformation theory over a given local ring $ B $. The idea is to look for all ways to deform the $ A_∞ $-products $ μ^1, μ^2, … $ of $ \cat C $ and allow additional curvature $ μ^0 $ in such a way that the products satisfy the “curved $ A_∞ $-relations” and reduce to the $ A_∞ $-structure of $ \cat C $ once the maximal ideal of $ B $ is divided out.

\paragraph*{Minimal models and derived categories}
Given an $ A_∞ $-category $ \cat C $ and an equivalent $ A_∞ $-category $ \cat D $, one also says that $ \cat D $ is a different model for $ \cat C $. Switching to a different model of a given $ A_∞ $-category is often tedious because it requires computing the higher products on the model. The Kadeishvili construction is an explicit method to construct these higher products. The idea is originally due to Kadeishvili, but has been worked out explicitly in \cite[Chapter 6, 3.3.2]{Kontsevich-Soibelman}.

Minimal models are a way to decrease the size of the hom spaces of an $ A_∞ $-category. An $ A_∞ $-category is called minimal if its differential $ μ^1 $ vanishes. A minimal model for $ \cat C $ is by definition any $ A_∞ $-category $ \H\cat C $ which is quasi-isomorphic to $ \cat C $ and whose differential $ μ^1_{\H\cat C} $ vanishes. Thanks to the Kadeishvili construction, every $ A_∞ $-category clearly has a minimal model: We take the cohomology of all the hom spaces and equip them with the $ A_∞ $-structure produced by the Kadeishvili construction.

Twisted completion is a way to form an $ A_∞ $-category of formal chain complexes of objects. The idea is to include arbitrary formal sums of shifted objects together with an upper-triangular matrix which represents the differential. Derived categories of $ A_∞ $-categories are defined by combining the twisted completion and minimal model construction. More precisely, the derived category of $ \cat C $ is defined as the $ A_∞ $-category $ \H\Tw\cat C $.

\paragraph*{The curvature problem and derived categories}
In formal $ A_∞ $-deformation theory, one regards infinitesimal deformations of a given $ A_∞ $-structure such that the $ A_∞ $-relations are preserved. An issue is that only deforming the $ A_∞ $-structure does not give a notion of deformations that is invariant under quasi-equivalence of $ A_∞ $-categories. In order to obtain a notion invariant under quasi-equivalence, one needs to permit the deformation to have curvature. More precisely, the curvature must be infinitesimal in the sense that it lies in a multiple of the maximal ideal of the local ring. Infinitesimal curvature is inevitable for a good notion of $ A_∞ $-deformations.

Dealing with curvature is however regarded as tedious, because the curvature prevents the differential from squaring to zero. The presence of curvature is often referred to as the “curvature problem”. A main question is how to gauge away the curvature or otherwise how to deal with the remaining curvature. An instance of the uncurving problem has been studied by Lowen and Van den Bergh \cite{Lowen-vdB}.

Unfortunately the deformation makes it difficult to speak about twisted completions, minimal models and derived categories: Twisted complexes are classically assumed to satisfy the Maurer-Cartan equation, a condition which ensures that the category $ \Tw\cat C $ becomes curvature-free. Minimal models require that the differential squares to zero in order to be able to take cohomology. Derived categories require both the twisted complex and minimal model construction to be well-defined.

\paragraph*{Application of the Kadeishvili construction} In \papertwoB, we work out an explicit example where the deformed $ A_∞ $-category is a deformed gentle algebra $ \cat C_q = \Tw\Gtl_q Q $. We provide an explicit matching of “result components” coming from the minimal model computation for $ \H\cat C_q $ and relative Fukaya disks of the surface $ |Q| $. This entails working out the meaning of all complicated Kadeishvili trees. This way, we prove that $ \H\Tw\Gtl_q Q $ agrees with the relative Fukaya category of $ |Q| $ with respect to the set of punctures $ Q_0 $.

\begin{center}
\input{toplevel/fig_id.tex}
\end{center}

\subsection*{Results}
The aim of the present paper is to show that deformations of $ A_∞ $-categories have derived categories. The first step is to define the twisted completion and minimal model of deformed $ A_∞ $-categories in an abstract way. This abstract method entails first forming the twisted completion or minimal model of the $ A_∞ $-category without deformation and afterwards pushing the deformation onto it. The second step is to describe twisted completion and minimal model explicitly. For the twisted completion, the construction is tricky but straight-forward. For the minimal model, the construction is a fine-tuned variant of the Kadeishvili construction.

\begin{center}
\begin{tabular}{@{}ccc@{}}
\morearraystretch
& \textbf{Original} & \textbf{Deformed} \\\hline
Category & $ \cat C $ & $ \cat C_q $ \\
Twisted completion & $ \Tw\cat C $ & $ \Tw\cat C_q $ \\
Minimal model & $ \H\cat C $ & $ \H\cat C_q $ \\
Derived category & $ \H\Tw\cat C $ & $ \H\Tw\cat C_q $
\end{tabular}
\end{center}

\paragraph*{Abstract definition for twisted completion}
The most conceptual way to define $ \Tw\cat C_q $ is to take the twisted completion $ \Tw\cat C $ and induce the deformation afterwards. More precisely, the first step is to form the twisted completion $ \Tw\cat C $ and regard the inclusion $ i: \cat C → \Tw\cat C $. The second step is to observe the existence of a noncanonical $ L_∞ $-isomorphism of Hochschild DGLAs $ i_*: \HC(\cat C) → \HC(\Tw\cat C) $ since $ i $ is a derived equivalence. The deformation $ \cat C_q $ is a Maurer-Cartan element of $ \HC(\cat C) $ over the given local ring. The third step is to push the Maurer-Cartan element via $ i_* $ and obtain a Maurer-Cartan element of $ \HC(\Tw\cat C) $. This is the desired deformation $ \Tw\cat C_q $.

\paragraph*{Explicit construction for twisted completion}
The explicit way to define $ \Tw\cat C_q $ is by gathering the same objects as $ \Tw\cat C $, namely those twisted complexes satisfying the non-deformed Maurer-Cartan equation $ μ^1 (δ) + μ^2 (δ, δ) + … = 0 $. This is not the same as gathering twisted complexes with twisted differential $ δ $ satisfying the deformed Maurer-Cartan equation $ μ^0_q + μ^1_q (δ) + … = 0 $. In contrast to $ \Tw\cat C $, the category $ \Tw\cat C_q $ inherits infinitesimal curvature, stemming precisely from the infinitesimal failure of the twisted differentials to satisfy the deformed Maurer-Cartan equation.

\paragraph*{Abstract definition for the minimal model}
The most conceptual way to define $ \H\cat C_q $ consists of taking any minimal model $ \H\cat C $ and inducing the deformation $ \cat C_q $ onto $ \H\cat C $ via any quasi-isomorphism $ π: \cat C → \H\cat C $. This abstract approach means that the “minimal model” $ \H\cat C_q $ may have infinitesimal curvature as well as a residual infinitesimal differential. This is not the same as taking cohomology of the hom complexes $ (\Hom_{\cat C_q} (X, Y), μ^1_q) $. In fact, the deformed differential of an $ A_∞ $-deformation need not even square to zero because of the curvature.

\paragraph*{Explicit construction for the minimal model}
Minimal models of $ A_∞ $-categories can classically be computed by means of homological splittings and Kadeishvili trees. In our deformed Kadeishvili theorem, we show that this method carries over to the deformed case. The starting point is an $ A_∞ $-category $ \cat C $ together with an deformation $ \cat C_q $. The difficulties encountered in constructing the minimal model are the presence of curvature $ μ^0_q $, the fact that the deformed differential $ μ^1_q $ does not square to zero and the fact that $ μ^1_q $ is not compatible with the homological splitting chosen for $ \cat C $. In \autoref{sec:kadeishvili}, we show how to adapt the Kadeishvili construction to these special circumstances.

Our construction can be summarized as follows:
\begin{enumerate}
\item Choose a homological splitting for $ \cat C $.
\item Perform an infinitesimal change on the homological splitting in order to adapt it to $ μ^1_q $.
\item Gauge away part of the deformation's curvature.
\item Repeat steps 2 and 3 indefinitely. Take the limit of this process.
\item Calculate the deformed codifferential $ h_q $ and projection $ π_q $.
\item Define the structure of $ \H\cat C_q $ by sums over deformed Kadeishvili trees.
\end{enumerate}

\subsection*{Relation to the literature}
The theory of $ A_∞ $-deformations presented in \autoref{sec:prelim-ainfty} and \ref{sec:lens} is known to experts. In the literature the topic is however under-represented, to the extent that it is hard to pinpoint a single location where $ A_∞ $-deformations are defined. In what follows, we summarize three articles which gave rise to the understanding of the theory presented in this paper.

\paragraph*{Barmeier and Wang}
In \cite{Barmeier-Wang}, Barmeier and Wang prove the power of $ L_∞ $-morphisms in deformation theory of ordinary algebras. They depart from a quiver algebra with relations $ A = ℂQ / I $, where the ideal $ I $ is supposed to come from a reduction system. In order to classify all deformations of $ A $, they replace the Hochschild DGLA $ \HC(A) $ by a quasi-isomorphic $ L_∞ $-algebra $ L(A) $.

To define the Hochschild DGLA, we typically use the bar resolution. Barmeier and Wang demonstrate that this resolution can be substituted by any other one, in particular the very simple bimodule resolution $ P_{\bullet} $ of Chouhy-Solotar, see \cite[Section 4.2]{Barmeier-Wang}. The cochain map between resolutions immediately gives a quasi-isomorphism of complexes
\begin{equation}
\label{eq:literature-barmeier-qi}
\Hom_{A ¤ A^{\opp}} \left(\bigoplus\limits_{i ∈ ℕ} A ¤ A^{¤_i} ¤ A, A\right) \isoto \Hom_{A ¤ A^{\opp}} (P_{\bullet}, A).
\end{equation}
The left-hand side already being a DGLA, the $ L_∞ $-structure transfer theorem induces an $ L_∞ $-structure on the right-hand side such that \eqref{eq:literature-barmeier-qi} becomes a quasi-isomorphism of $ L_∞ $-algebras. Barmeier and Wang then compute part of the $ L_∞ $-structure on the right-hand side, just enough to classify all of its Maurer-Cartan elements. All deformations of $ A $ can then be written down explicitly.

Barmeier and Wang's realization that $ L_∞ $-quasi-isomorphisms transport Maurer-Cartan elements has greatly helped us shape the curved $ A_∞ $-deformation theory of \autoref{sec:lens}. There, we make constant use of the fact that quasi-equivalences of $ A_∞ $-categories induces $ L_∞ $-quasi-isomorphisms of their Hochschild DGLAs:
\begin{equation*}
F: \cat C \isoto \cat D \quad \rightsquigarrow \quad F_*: \HC(\cat C) \isoto \HC(\cat D).
\end{equation*}
Correspondingly, their sets of Maurer-Cartan elements over any deformation base $ B $ match:
\begin{equation*}
\MCb(\HC(\cat C), B) \isoto \MCb(\HC(\cat D), B).
\end{equation*}
This is the main principle that lets us push deformations to and fro between different categories. For example, it lets us seamlessly reduce an $ A_∞ $-category to a skeleton of non-isomorphic objects without changing its deformation theory. It enable us to prove that if an object $ X ∈ \cat C_q $ is uncurvable, then all objects $ Y ∈ \cat C_q $ quasi-isomorphic to $ X $ are uncurvable as well.

\paragraph*{Keller}
In \cite{Keller}, Keller proves the $ L_∞ $-invariance of the Hochschild DGLA under Morita equivalence of dg-categories. Keller actually regards the Hochschild DGLA as a $ B_∞ $-algebra which is even stronger than the notion of $ L_∞ $-algebra.

Let $ A $ and $ B $ be Morita equivalent dg algebras, with Morita equivalence provided by the $ A $-$ B $-bimodule $ M $. Then Keller's core argument for invariance is as follows: Embed $ A $ and $ B $ into the triangular dg algebra
\begin{equation*}
D ≔ \begin{pmatrix} A & M \\ 0 & B \end{pmatrix}.
\end{equation*}
Both embeddings $ A, B ⊂ D $ turn out to be Morita equivalences. Keller exploits the natural restriction maps $ \HC(D) → \HC(A) $ and $ \HC(D) → \HC(B) $, which automatically respect the $ B_∞ $-structure. Departing from the knowledge that Hochschild cohomology of dg algebras is invariant under Morita equivalences as graded vector space, Keller concludes that both $ \HC(D) → \HC(A) $ and $ \HC(D) → \HC(B) $ are $ B_∞ $-quasi-isomorphisms. Correspondingly, $ \HC(A) $ and $ \HC(B) $ are $ B_∞ $-quasi-isomorphic.

Unfortunately, Keller's result is currently restricted to the case dg algebras. To the present paper, this means that we simply assume as axioms that the theory extends to all $ A_∞ $-categories.
%
The heavy use of $ A_∞ $-deformations in our paper shows how imperative it is to rigorously extend Keller's paper to the $ A_∞ $-context. According to private communication with Keller, these results can likely be obtained by the same triangular construction together with appropriate $ A_∞ $-bimodule theory.

\paragraph*{Lowen and Van den Bergh}
In \cite{Lowen-vdB}, Lowen and Van den Bergh explain how to remove curvature from $ A_∞ $-deformations of dg categories. They depart from a dg algebra $ A $ together with an infinitesimally curved $ A_∞ $-deformation $ A_q $ over $ ℂ⟦q⟧ $. Lowen and Van den Bergh observe that a category $ \Tw(A_q) $ of twisted complexes over $ A_q $ can be formed even with infinitesimal entries below the diagonal, just as in our \autoref{th:twisted-deforming-delta}. Interpret $ A_q $ as an $ A_∞ $-deformation with a single object. Then the core observation of Lowen and Van den Bergh is that the following twisted complex has vanishing curvature:
\begin{equation}
\label{eq:literature-LvdB-trick}
X ≔ \left(A ⊕ A[1], \begin{pmatrix} 0 & μ^0_q / q \\ q \id_A & 0 \end{pmatrix}\right) ∈ \Tw(A_q).
\end{equation}
This means that $ C_q ≔ \End_{\Tw(A_q)} (X, X) $ is a curvature-free deformed $ A_∞ $-algebra. It is interesting to note that its special fiber $ C $ at $ q = 0 $ is actually a dg algebra. Indeed, the higher products $ μ^{≥3} $ on $ C_q $ are given by embracing $ μ_{A_q} $ with the matrix entries $ μ^0_q /q $ and $ q \id_A $:
\begin{equation*}
μ^{k≥3}_{C_q} (a_k, …, a_1) = \sum μ^{≥3}_{\Add A_q} (δ, …, δ, a_k, …, δ, …, δ, a_1, δ, …, δ).
\end{equation*}
Restricting this sum to $ q = 0 $ yields only higher products $ μ^{≥3} $ of $ A $. Since $ A $ is a dg algebra, we deduce $ μ^{k≥3}_C = 0 $. In conclusion, $ C $ is a dg algebra as well and $ C_q $ is a curvature-free $ A_∞ $-deformation of $ C $.

Lowen and Van den Bergh prove that $ A $ and $ C $ are in fact related by Morita equivalence. This costs substantial effort and uses the assumption that the curvature $ μ^0_{A_q} $ is nilpotent in the cohomology of $ A $. The result is however that $ A $ and $ C $ are Morita equivalent, and moreover that $ C_q $ is the deformation of $ C $ corresponding to the deformation $ A_q $ of $ A $ along this Morita equivalence:
\begin{center}
\begin{tikzpicture}
\path (0, 0) node {dg algebra A} (3, 0) node {\Large $ \rightsquigarrow $} (7, 0) node {dg algebra $ C $};
\path (0, -0.5) node {curved $ A_∞ $-deformation $ A_q $} (3, -0.5) node {\Large $ \rightsquigarrow $} (7, -0.5) node {uncurved $ A_∞ $-deformation $ C_q $};
\end{tikzpicture}
\end{center}

The work of Lowen and Van den Bergh illustrates that curvature is essential in the notion of $ A_∞ $-deformations, but not an invariant on its own. Curvature can sometimes be removed by changing to a derived equivalent $ A_∞ $-category. Our uncurving construction for band objects in \papertwoB\ provides an example where a mere gauge transformation suffices to remove curvature.



Uncurving is a crucial contribution to our deformed Kadeishvili construction. The presence of curvature namely allows the differential of an $ A_∞ $-deformation to be behave in an arbitrary way. In contrast to the classical Kadeishvili construction, the differential fails to structure the hom spaces into cohomology, image and remainder term. By iteratively gauging the deformation, we optimizing the curvature value until in the limit the differential structures the hom spaces sufficiently. Only after this uncurving procedure can the deformed minimal model be read off from Kadeishvili trees.

As aftermath of our paper, we conclude that there is theoretically no hindrance to forming a derived category of a curved $ A_∞ $-deformation: By \autoref{sec:constructions-derived}, a derived category $ \H\Tw\cat C_q $ exists even for infinitesimally curved deformations. The statement of Lowen and Van den Bergh that a curved $ A_∞ $-deformation has no classical derived category remains true, but our paper contends that the study of deformations profits greatly from permitting also these “non-classical derived categories” $ \H\Tw\cat C_q $.


\subsection*{Structure of the paper}
In \autoref{sec:ainfty}, we recall $ A_∞ $-categories. In \autoref{sec:prelim-ainfty}, we recall $ A_∞ $-deformations, their twisted completion and functors. In \autoref{sec:lens}, we interpret $ A_∞ $-deformations via deformation theory and DGLAs. In \autoref{sec:constructions}, we explain the abstract definitions of twisted completion and minimal models and comment on uncurving theory. In \autoref{sec:kadeishvili}, we construct our deformed Kadeishvili theorem.

For the purposes of \autoref{sec:lens}, \ref{sec:constructions-minmodel} and \ref{sec:uncurving-theory}, we have to assume axioms on the invariance of the Hochschild DGLA under $ A_∞ $-quasi-equivalences. We summarize these axioms in \autoref{th:curved-axioms-convention}.

\papertwoAacknowledgements

%% file: toplevel/fig_id.tex
\begingroup
\newcommand{\hexagondraw}{\path[draw, gray] (0, 0) -- ++(up:1) coordinate[midway] (A) coordinate[pos=0.7] (Ai) -- ++(330:1) coordinate[midway] (B) coordinate[pos=0.3] (Bi) ++(60:0.2) -- ++(150:1) coordinate[midway] (C) coordinate[pos=0.7] (Ci) -- ++(30:1) coordinate[midway] (D) coordinate[pos=0.3] (Di) ++(120:0.2) -- ++(210:1) coordinate[midway] (E) coordinate[pos=0.7] (Ei) -- ++(up:1) coordinate[midway] (F) coordinate[pos=0.3] (Fi) ++(left:0.2) -- ++(down:1) coordinate[midway] (G) coordinate[pos=0.7] (Gi) -- ++(150:1) coordinate[midway] (H) coordinate[pos=0.3] (Hi) ++(240:0.2) -- ++(330:1) coordinate[midway] (I) coordinate[pos=0.7] (Ii) -- ++(210:1) coordinate[midway] (J) coordinate[pos=0.3] (Ji) ++(300:0.2) -- ++(30:1) coordinate[midway] (K) coordinate[pos=0.7] (Ki) -- ++(down:1) coordinate[midway] (L) coordinate[pos=0.3] (Li);
\path ($ (F)!0.5!(G) $) coordinate (1);
\path ($ (D)!0.5!(E) $) coordinate (2);
\path ($ (B)!0.5!(C) $) coordinate (3);
\path ($ (L)!0.5!(A) $) coordinate (4);
\path ($ (J)!0.5!(K) $) coordinate (5);
\path ($ (H)!0.5!(I) $) coordinate (6);}

\begin{tikzpicture}[scale=1.5]
\hexagondraw
\path (1) -- (6) coordinate[pos=0.4] (out);
\path (out) node[above left] {\small out};
\foreach \i in {1, 2, 3, 4, 5, 6, out} \path[fill] (\i) circle[radius=0.03];
\path (1) node[above] {$ h_1 $};
\path (2) node[below right] {$ h_2 $};
\path (3) node[below] {$ h_3 $};
\path (4) node[below] {$ h_4 $};
\path (5) node[below] {$ h_5 $};
\path (6) node[below left] {$ h_6 $};
\path[draw, semithick, -{To[scale=1.5]}] (1) -- (2);
\path[draw, semithick] (2) -- (3);
\path[draw, semithick] (3) -- (4);
\path[draw, semithick] (4) -- (5);
\path[draw, semithick] (5) -- (6);
\path[draw, semithick, {To[scale=1.5]}-] (6) -- (out) to[bend right] (1);
\path (1) -- (2); 
\path (2) -- (3); 
\path (3) -- (4); 
\path (4) -- (5); 
\path (5) -- (6); 
\path (6) -- (out); 
\path[draw, <->] (1.5, 1.2) -- (3.5, 1.2) node[midway, above] {\papertwoB};
\begin{scope}[shift={(4, 2.2)}]
\path node (A) {$ h_6 $} node[right of=A] (B) {$ h_5 $} node[right of=B] (C) {$ h_4 $} 
node[right of=C] (D) {$ h_3 $} node[right of=D] (E) {$ h_2 $} node[right of=E] (F) {$ h_1 $}
node[below right of=B] (G) {$ h_q μ^2_q $} edge (B) edge (C)
node[below right of=G] (H) {$ h_q μ^2_q $} edge (G) edge (D)
node[below right of=H] (I) {$ h_q μ^2_q $} edge (H) edge (E)
node[below left of=I] (J) {$ h_q μ^2_q $} edge (I) edge (A)
node[below right of=J] (K) {$ π_q μ^2_q $} edge (J) edge (F);
\end{scope}
\end{tikzpicture}
\endgroup

%% file: ainfty/ainfty.tex
\section{Preliminaries on $ A_∞ $-categories}
\label{sec:ainfty}
The notion of $ A_∞ $-category is now widely used as a homological-algebraic tool to study symplectic and algebraic geometry. Its relation to dg categories, triangulated categories and $ ∞ $-categories can be described as follows:
\begin{itemize}
\item DG categories are more rigid than $ A_∞ $-categories. DG quasi-isomorphisms cannot be quasi-inverted on the dg level, while they always have a quasi-inverse if one interprets the dg structures as $ A_∞ $-categories. In particular, zigzags of dg quasi-isomorphisms can be resolved into a single $ A_∞ $-quasi-isomorphism. DG structures are very easy to work with, for instance homotopy colimits are easy to calculate due to the well-known Tabuada model structure \cite{Tabuada, Keller-onDG}.
\item Triangulated categories are weaker than $ A_∞ $-categories. They do not remember any of the “higher structure” that $ A_∞ $-categories contain. The higher structure on an $ A_∞ $-category makes the entire category reconstructible from a subcategory of generators. There is an $ A_∞ $ notion of derived category $ \D\cat C ≔ \H\Tw\cat C $, whose “lower structure” $ (\D \cat C)^0 $ is a triangulated category \cite{Kontsevich-ICM}.
\item Stable $ ∞ $-categories also remember the homotopy information, similar to $ A_∞ $-categories. A priori they lack the linearity over a base field, which can be added afterwards. In fact, there is a direct correspondence between $ k $-linear stable $ ∞ $-categories and $ A_∞ $-categories over $ k $, given in one direction by taking the derived category \cite{Cohn, Faonte}.
\end{itemize}

We work over an algebraically closed field of characteristic zero and always write $ ℂ $. Let us now recall the definition of $ A_∞ $-categories:

\begin{definition}
\label{def:ainfty-def}
A ($ ℤ $- or $ ℤ/2ℤ $-graded, strictly unital) \emph{$ A_∞ $-category} $ \cat C $ consists of a collection of objects together with $ ℤ $- or $ ℤ/2ℤ $-graded hom spaces $ \Hom(X, Y) $, distinguished identity morphisms $ \id_X ∈ \Hom^0 (X, X) $ for all $ X ∈ \cat C $, together with multilinear higher products
\begin{equation*}
μ^k: \Hom(X_k, X_{k+1}) × … × \Hom(X_1, X_2) → \Hom(X_1, X_{k+1}), \quad k ≥ 1
\end{equation*}
of degree $ 2-k $ such that the $ A_∞ $-relations and strict unitality axioms hold: For all compatible morphisms $ a_1, …, a_k $ we have
\begin{align*}
& \sum_{0 ≤ n < m ≤ k} (-1)^{‖a_n‖ + … + ‖a_1‖} μ(a_k, …, μ(a_m, …, a_{n+1}), a_n, …, a_1) = 0, \\
& μ^2 (a, \id_X) = a, ~ μ^2 (\id_Y, a) = (-1)^{|a|} a, ~ μ^{≥3} (…, \id_X, …) = 0.
\end{align*}
\end{definition}

\begin{remark}
Let us write down the first few $ A_∞ $-relations. In particular, we note that if $ μ^1 $ or $ μ^3 $ vanishes then $ μ^2 $ is graded associative:
\begin{align*}
μ^1 (μ^1 (a)) &= 0, \\
(-1)^{‖b‖} μ^2 (μ^1 (a), b) + μ^2 (a, μ^1 (b)) + μ^1 (μ^2 (a, b)) &= 0, \\
(-1)^{‖c‖} μ^2 (μ^2 (a, b), c) + μ^2 (a, μ^2 (b, c)) + μ^1 (μ^3 (a, b, c)) \phantom{XYZ} & \\
+ (-1)^{‖b‖ + ‖c‖} μ^3 (μ^1 (a), b, c) + (-1)^{‖c‖} μ^3 (a, μ^1 (b), c) + μ^3 (a, b, μ^1 (c)) &= 0.
\end{align*}
\end{remark}

Recall that a functor $ F: \cat C → \cat D $ of $ A_∞ $-categories is a map which intertwines the products of $ \cat C $ and $ \cat D $:

\begin{definition}
\label{def:prelim-ainfty-functor}
Let $ \cat C $ and $ \cat D $ be $ A_∞ $-categories. A \emph{functor} $ F: \cat C → \cat D $ consists of a map $ F: \Ob(\cat C) → \Ob(\cat D) $ together with for every $ k ≥ 1 $ a degree $ 1-k $ multilinear map
\begin{equation*}
F^k: \Hom_{\cat C} (X_k, X_{k+1}) ¤ … ¤ \Hom_{\cat C} (X_1, X_2) → \Hom_{\cat C} (FX_1, FX_{k+1})
\end{equation*}
such that the $ A_∞ $-functor relations hold:
\begin{multline*}
\sum_{0 ≤ j < i ≤ k} (-1)^{‖a_j‖ + … + ‖a_1‖} F(a_k, …, a_{i+1}, μ(a_i, …, a_{j+1}), a_j, …, a_1) \\
= \sum_{\substack{l ≥ 0 \\ 1 = j_1 < … < j_l ≤ k}} μ(F(a_k, …, a_{j_l}), …, F(…, a_{j_2}), F(…, a_{j_1})).
\end{multline*}
When $ F: \cat C → \cat D $ and $ G: \cat D → \cat E $ are $ A_∞ $-functors, then their composition $ GF $ is given on objects by $ G∘F: \Ob(\cat C) → \Ob(\cat E) $ and on morphisms by
\begin{equation*}
(GF)(a_k, …, a_1) = \sum G(F(a_k, …), …, F(…, a_1)).
\end{equation*}
\end{definition}

There are several notions for functors between $ A_∞ $-categories to be equivalences, which we recall as follows:

\begin{definition}
\label{def:prelim-ainfty-functorquasi}
Let $ \cat C $ and $ \cat D $ be two $ A_∞ $-categories and $ F: \cat C → \cat D $ be a functor. Then $ F $ is
\begin{itemize}
\item an \emph{isomorphism} if it is an isomorphism on object level and $ F^1: \Hom(X, Y) → \Hom(FX, FY) $ is an isomorphism for every $ X, Y ∈ \cat C $,
\item a \emph{quasi-isomorphism} if the induced functor $ \H F: \H\cat C → \H\cat D $ is an isomorphism,
\item a \emph{quasi-equivalence} if $ (\H F)^1: \Hom_{\H\cat C} (X, Y) → \Hom_{\H\cat D} (FX, FY) $ is an isomorphism for every $ X, Y ∈ \cat C $ and if $ \H F $ reaches every object in $ \cat D $ up to quasi-isomorphism,
\item a \emph{derived equivalence} if the induced functor $ \H\Tw F: \H\Tw \cat C → \H\Tw \cat D $ is an equivalence.
\end{itemize}
\end{definition}

There are several notions for $ A_∞ $-categories to be equivalent, without explicit reference to a functor:

\begin{definition}
\label{def:prelim-ainfty-catquasi}
Let $ \cat C $ and $ \cat D $ be two $ A_∞ $-categories. Then $ \cat C $ and $ \cat D $ are
\begin{itemize}
\item \emph{isomorphic} if there exists an isomorphism $ F: \cat C → \cat D $,
\item \emph{quasi-isomorphic} if there exists a quasi-isomorphism $ F: \cat C → \cat D $,
\item \emph{quasi-equivalent} if there exists a quasi-equivalence $ F: \cat C → \cat D $,
\item \emph{derived equivalent} if $ \Tw\cat C $ and $ \Tw\cat D $ are quasi-equivalent.
\end{itemize}
\end{definition}

\begin{remark}
No functor in either direction is required for two $ A_∞ $-categories to be derived equivalent. The contrast with quasi-isomorphisms is due to the lack of functors $ \Tw\cat C → \cat C $, in contrast to the existence of natural functors $ \H\cat C → \cat C $ and $ \cat C → \H\cat C $. There are many more equivalent and equally esthetic ways to define every of the above notions, so we have only presented a selection.
\end{remark}

Let us recall minimal models of $ A_∞ $-categories. These models are called minimal because the differential on all hom spaces vanish. In \autoref{sec:kadeishvili}, we comment extensively on minimal models. Minimal models of $ A_∞ $-category are not uniquely defined. Instead we fix the following parlance:

\begin{definition}
An $ A_∞ $-category is \emph{minimal} if its differential vanishes. Let $ \cat C $ be an $ A_∞ $-category. Then a \emph{minimal model} for $ \cat C $ is any minimal $ A_∞ $-category $ \H\cat C $ such that $ \cat C $ and $ \H\cat C $ are quasi-isomorphic.
\end{definition}

If $ \cat C $ is a minimal $ A_∞ $-category, let us regard its degree-zero part $ \cat C^0 $. It has the same objects as $ \cat C $ but only including the degree-zero part of the hom spaces. Then the product composition $ a ∘ b ≔ (-1)^{|b|} μ^2 (a, b) $ is associative on $ \cat C^0 $, rendering $ \cat C^0 $ an ordinary $ ℂ $-linear category. Two objects $ X, Y ∈ \cat C $ are \emph{quasi-isomorphic} if they are isomorphic in $ (\H\cat C)^0 $.

As an auxiliary construction, we recall the additive completion of $ A_∞ $-categories. The idea is to form a category whose objects are formal sums of objects. The hom spaces and $ A_∞ $-products are extended accordingly. We fix the definition and notation as follows:

\begin{definition}
\label{def:2Ainfty-ainfty-Add}
Let $ \cat C $ be an $ A_∞ $ category with product $ μ_{\cat C} $. The \emph{additive completion} $ \Add \cat C $ of $ \cat C $ is the category of formal sums of shifted objects of $ \cat C $:
\begin{equation*}
A_1 [k_1] ⊕ … ⊕ A_n [k_n].
\end{equation*}
The hom space between two such objects $ X = \bigoplus A_i [k_i] $ and $ Y = \bigoplus B_i [m_i] $ is
\begin{equation*}
\Hom_{\Add\cat C} (X, Y) = \bigoplus_{i, j} \Hom_{\cat C} (A_i, B_j) [m_j - k_i].
\end{equation*}
Here $ [-] $ denotes the right-shift. The products on $ \Add\cat C $ are given by multilinear extensions of
\begin{equation*}
μ_{\Add \cat C}^k (a_k, …, a_1) = (-1)^{\sum_{j < i} \Vert a_i \Vert l_j} μ_{\cat C}^k (a_k, …, a_1).
\end{equation*}
Here each $ a_i $ lies in some $ \Hom(X_i[k_i], X_{i+1} [k_{i+1}]) $. The integer $ l_i $ denotes the difference $ k_{i+1} - k_i $ between the shifts and the degree $ ‖a_i‖ $ is the degree of $ a_i $ as element of $ \Hom_{\cat C} (X_i, X_{i+1}) $.
\end{definition}

Next we recall the twisted completion $ \Tw\cat C $ of an $ A_∞ $-category $ \cat C $. The objects of this category are virtual chain complexes of objects of $ \cat C $. The usage of the twisted completion construction is entirely analogous to the usage of the category of complexes over an abelian category. We recall the definition and notation as follows:

\begin{definition}
\label{def:2Ainfty-ainfty-Tw}
A \emph{twisted complex} in $ \cat C $ is an object $ X ∈ \Add\cat C $ together with a morphism $ δ ∈ \Hom^1_{\Add\cat C} (X, X) $ of degree $ 1 $ such that $ δ $ is strictly upper triangular and satisfies the Maurer-Cartan equation:
\begin{equation*}
\MC(δ) ≔ μ^1 (δ) + μ^2 (δ, δ) + … = 0.
\end{equation*}
We may refer to the morphism $ δ $ as the \emph{twisted differential}. Note that the upper triangularity ensures that this sum is well-defined. The \emph{twisted completion} of $ \cat C $ is the $ A_∞ $-category $ \Tw\cat C $ whose objects are twisted complexes. Its hom spaces are the same as for the additive completion:
\begin{equation*}
\Hom_{\Tw\cat C} (X, Y) = \Hom_{\Add\cat C} (X, Y).
\end{equation*}
The products on $ \Tw \cat C $ of $ \cat C $ are given by embracing with $ δ $'s:
\begin{equation*}
μ_{\Tw \cat C}^k (a_k, …, a_1) = \sum_{n_0, …, n_k ≥ 0} μ_{\Add \cat C} (\underbrace{δ, …, δ}_{n_k}, a_k, …, a_1, \underbrace{δ, …, δ}_{n_0}).
\end{equation*}
\end{definition}

The notion of a \emph{curved $ A_∞ $-category} is a generalization of ordinary $ A_∞ $-category. In contrast to an ordinary $ A_∞ $-category, a curved $ A_∞ $-category also has an element $ μ^0_X $ of degree $ 2 $ associated with every object $ X ∈ \cat C $. The required curved $ A_∞ $-relations are the same as those of an $ A_∞ $-category, in particular allowing $ μ^0_X $ to appear as inner $ μ $. The first two curved $ A_∞ $-relations read
\begin{equation*}
μ^1 (μ^0) = 0, \quad μ^1 (μ^1 (a)) + (-1)^{‖a‖} μ^2 (μ^0, a) + μ^2 (a, μ^0) = 0.
\end{equation*}
Curved $ A_∞ $-categories are however ill-behaved: They do not have a notion of derived category. One may form the twisted completion, but if one enforces the Maurer-Cartan equation $ μ^0 + μ^1 (δ) + … = 0 $, very few objects remain and the construction is not functorial. If one instead discards the Maurer-Cartan requirement for twisted complexes, one obtains a category $ \Tw \cat C $ with curvature, the curvature given by the Maurer-Cartan formula $ μ^0 + μ^1 (δ) + … $. This category however has no minimal model $ \H\Tw\cat C $, because curvature prevents us from bringing $ μ^1 $ to zero. In short, curved $ A_∞ $-categories are not useful, except for the purpose of matrix factorizations.

We would like to stress that infinitesimal curvature however does not hurt. In \autoref{sec:prelim-ainfty-defo}, we recall what infinitesimal curvature entails and we explain in \autoref{sec:constructions} that infinitesimally curved $ A_∞ $-categories do have derived categories.

%% file: defo/intro.tex
\section{Deformations of $ A_∞ $-categories}
\label{sec:prelim-ainfty}
In this section, we recollect material on $ A_∞ $-deformations. Theory on $ A_∞ $-deformations is scarce and implicit in the literature, so we provide a rigorous definition here that includes the infinitesimally curved case. We explain why curvature is inevitable if one wants to obtain a deformation theory invariant under quasi-equivalences.

In \autoref{sec:prelim-htensor}, we recall the completed tensor products of the form $ B \htensor X $, where $ B $ is a local algebra. In \autoref{sec:prelim-ainfty-defo}, we recall curved deformations of $ A_∞ $-categories. In \autoref{sec:prelim-ainfty-twisted}, we define their twisted completion. In \autoref{sec:curved-functors}, we recall functors between $ A_∞ $-categories and between deformations of $ A_∞ $-categories.

The material in the entire section is not original. For instance, Seidel already mentioned curved $ A_∞ $-deformations in \cite{Seidel-relative} and Lowen and Van den Bergh considered the “curvature problem” for formal deformations in \cite{Lowen-vdB}.

In this section, we will be comparing different deformations with another. We should set some terminology right before we get started. A “deformation base” $ B $ will be a complete local Noetherian unital $ ℂ $-algebra with residue field $ B / \mathfrak{m} = ℂ $. Every deformation $ \cat C_q $ over $ B $ can be interpreted as object of two different universes: either as deformation of $ \cat C $, or as a deformation of any $ A_∞ $-category. Let us depict this pictorially:
\begin{center}
\begin{tikzpicture}
\path[draw] (0, 0) circle[x radius=2.5, y radius=0.6] ++(up:0.1) node {deformations of $ \cat C $};
\path[draw] (7, 0) circle[x radius=3.5, y radius=0.8] ++(up:0.2) node {deformations of any category};
\path[fill] (0.5, -0.3) circle[radius=0.05] node[left] {$ \cat C_q' $};
\path[fill] (-0.5, -0.3) circle[radius=0.05] node[left] {$ \cat C_q $};
\path[fill] (6, -0.4) circle[radius=0.05] node[left] {$ \cat C_q $};
\path[fill] (7, -0.4) circle[radius=0.05] node[left] {$ \cat C_q' $};
\path[fill] (8, -0.4) circle[radius=0.05] node[left] {$ \cat D_q $};
\end{tikzpicture}
\end{center}
This perspective makes a difference: The correct notion for two deformations $ \cat C_q, \cat C_q' $ to be similar on the left side is to be gauge equivalent, while the correct notion for two deformations $ \cat C_q $ and $ \cat D_q $ to be similar on the right side is to be quasi-equivalent. In the former case, the leading term of the functor connecting the two is supposed to be the identity, in the latter case the leading term can be any quasi-equivalence of $ A_∞ $-categories.

To distinguish the two perspectives, we will often use the terminology \emph{$ A_∞ $-deformation} to refer to the left side, and \emph{deformed $ A_∞ $-category} to refer to the right side.

%% file: defo/htensor.tex
\subsection{The completed tensor product}
\label{sec:prelim-htensor}
In this section, we recall the completed tensor products of the form $ B \htensor X $, where $ B $ is a local algebra and $ X $ is a vector space. This serves as a preparation for \autoref{sec:prelim-ainfty-defo} where we define $ A_∞ $-deformations.

Throughout this paper, we deform over local rings like $ ℂ⟦q⟧ $. There are a few more conditions we put on the local ring: In order to work with the $ A_∞ $-formalism we need the ring to be a $ ℂ $-algebra. In order to speak of a special fiber, or algebraically of an $ \mathfrak{m} $-adic leading term, we need to require that its residue field is $ ℂ $ itself. As is customary in deformation theory, we shall also require that the ring be complete and Noetherian. We have decided to give the type of local rings a name in this paper:

\begin{definition}
\label{def:prelim-defo-base}
A \emph{deformation base} is a complete local Noetherian unital $ ℂ $-algebra $ B $ with residue field $ B/\mathfrak{m} = ℂ $. The maximal ideal is always denoted $ \mathfrak{m} $.
\end{definition}

\begin{remark}
By the Cohen structure theorem, every deformation base is of the form $ ℂ⟦x_1, …, x_n⟧ / I $ with $ I $ denoting some ideal.
\end{remark}

The idea behind formal $ A_∞ $-deformations is to tensor the hom spaces with a deformation base. One then looks at $ A_∞ $-structures on the enlarged collection of hom spaces which reduce to the original $ A_∞ $-structure once the maximal ideal is divided out. The construction of these tensored hom spaces makes use of the completed tensor product, which we now recall.

\begin{definition}
\label{def:prelim-htensor-def}
Let $ B $ be a deformation base and $ X $ a vector space. Then the \emph{completed tensor product} $ B \htensor X $ is the $ B $-module limit
\begin{equation*}
B \htensor X = \lim(… → B / \mathfrak{m}^1 ¤ X → B / \mathfrak{m}^0 ¤ X),
\end{equation*}
For simplicity, we write $ \mathfrak{m}^k X $ to denote the infinitesimal part $ \mathfrak{m}^k X = \mathfrak{m}^k \htensor X ⊂ B \htensor X $.
\end{definition}

\begin{remark}
In case $ B = ℂ⟦q⟧ $, the completed tensor product $ B \htensor X $ equals the even more well-known space $ X⟦q⟧ $. The space $ X⟦q⟧ $ consists of formal $ X $-valued power series in one variable. This way $ X⟦q⟧ $ becomes naturally a $ ℂ⟦q⟧ $-module. The space is different from $ ℂ⟦q⟧ ¤ X $. In fact, elements of $ ℂ⟦q⟧ ¤ X $ are only those power series which can be written as a finite sum of pure tensors $ a ¤ x $. Simply speaking, in $ ℂ⟦q⟧ ¤ X $, the power series entries are divided into finitely many partitions in which all coefficients are interrelated. Meanwhile in $ X⟦q⟧ $ any entries can be chosen at random. However if $ X $ is finite-dimensional, then $ X⟦q⟧ = ℂ⟦q⟧ ¤ X $.
\end{remark}

There are two popular ways of defining formal deformations of an associative algebra $ A $. The first definition asks for a product $ μ_q: A ¤ A → B \htensor A $ and the second definition asks for a product $ μ_q: (B \htensor A) ¤ (B \htensor A) → B \htensor A $. The first definition immediately gives rise to a Maurer-Cartan element, while in the second definition associativity is formulated more naturally. In what follows, we shall explain briefly why both are equivalent.

\begin{definition}
Let $ B $ be a deformation base. Then the $ \mathfrak{m} $-adic topology on $ B $ is the topology on $ B $ generated by open neighborhoods $ x + \mathfrak{m}^k ⊂ B $ for $ x ∈ B $ and $ k ∈ ℕ $. The $ \mathfrak{m} $-adic topology on $ B \htensor X $ is generated by the open neighborhoods $ x + \mathfrak{m}^k \htensor X ⊂ B \htensor X $ for every $ x ∈ B \htensor X $ and $ k ∈ ℕ $. A map $ φ: B \htensor X → B \htensor Y $ is \emph{continuous} if it is continuous with respect to the $ \mathfrak{m} $-adic topologies. A map $ φ: (B \htensor X_k) ¤ … ¤ (B \htensor X_1) → B \htensor Y $ is \emph{continuous} if for every $ 1 ≤ i ≤ k $ and every sequence of elements $ x_1, …, \hat x_i, …, x_k $ the map
\begin{equation*}
μ(x_k, …, -, …, x_1): B \htensor X_i → B \htensor Y
\end{equation*}
is continuous.
\end{definition}

\begin{remark}
It is well-known that the $ \mathfrak{m} $-adic topology turns $ B \htensor X $ into a sequential Hausdorff space. It is also well-known that $ B \htensor X $ can simultaneously be interpreted as limit and completion. More precisely, $ B \htensor X $ is the completion of $ B ¤ X $ with respect to the so-called $ \mathfrak{m} $-adic metric on $ B ¤ X $. For convenience, we may from time to time use expressions like $ x = \landau(\mathfrak{m}^k) $ to indicate $ x ∈ \mathfrak{m}^k X $.
\end{remark}

\begin{remark}
Every element in $ B \htensor X $ can be written as a series $ \sum_{i = 0}^∞ m_i x_i $. Here $ m_i $ is a sequence of elements $ m_i ∈ \mathfrak{m}^{→∞} $ and $ x_i $ is a sequence of elements $ x_i ∈ X $. We have used the notation $ m_i ∈ \mathfrak{m}^{→∞} $ to indicate that $ m_i ∈ \mathfrak{m}^{k_i} $ for some sequence $ (k_i) ⊂ ℕ $ with $ k_i → ∞ $.
\end{remark}

\begin{remark}
Let $ φ: B \htensor X → B \htensor Y $ be a $ B $-linear map. If $ φ(\mathfrak{m}^k X) ⊂ \mathfrak{m}^k Y $, then $ φ $ is continuous.
\end{remark}

Whenever $ X → B \htensor Y $ is a linear map, we can extend it uniquely to a $ B $-linear continuous map $ B \htensor X → B \htensor Y $. Conversely, we can restrict any $ B $-linear map $ B \htensor X → B \htensor Y $ to a linear map $ X → B \htensor Y $. Restriction and extension are in fact inverse to each other, providing a one-to-one correspondence:

\begin{lemma}
\label{th:prelim-htensor-autocontinuous}
Let $ B $ be a deformation base and $ X, Y $ be vector spaces. Then a $ B $-linear map $ B \htensor X → B \htensor Y $ is automatically continuous. There is a one-to-one correspondence between:
\begin{itemize}
\item linear maps $ X → B \htensor Y $,
\item $ B $-linear maps $ B \htensor X → B \htensor Y $,
\item $ B $-linear continuous maps $ B \htensor X → B \htensor Y $.
\end{itemize}
\end{lemma}

\begin{proof}
The proof consists of three parts: First, we show that every $ B $-linear map $ B \htensor X → B \htensor Y $ is continuous. Second, we recall how to extend a map $ X → B \htensor Y $ to $ B \htensor X → B \htensor Y $. Third, we comment on the one-to-one aspect of the claim.

For the first step, we show that any $ B $-linear map $ φ: B \htensor X → B \htensor Y $ is continuous. Let $ k ∈ ℕ $. We claim that $ φ(\mathfrak{m}^k X) ⊂ \mathfrak{m}^k Y $. The idea is to exploit the Cohen structure theorem. Write $ B = ℂ⟦q_1, …, q_n⟧ / I $, and regard the maximal ideal $ \mathfrak{m} = (q_1, …, q_n) $. With this in mind, we can write any element $ x ∈ \mathfrak{m}^k X $ as a series
\begin{equation*}
x = \sum_{i = 0}^∞ m_i \tilde m_i y_i.
\end{equation*}
Here $ m_i $ is a monomial of degree $ k $ in the variables $ q_1, …, q_n $, the letter $ \tilde m_i $ denotes a sequence $ \tilde m_i ∈ \mathfrak{m}^{→∞} $, and $ y_i ∈ Y $. We conclude
\begin{equation*}
x = \sum_{\substack{\text{monomials } M \\ \text{ of degree } k}} M \sum_{\substack{i ≥ 0 \\ m_i = M}} \tilde m_i y_i.
\end{equation*}
The outer sum is finite. For every monomial $ M $ of degree $ k $, the inner sum is an element $ x_M ∈ B \htensor X $. We get that
\begin{equation*}
φ(x) = \sum_{\substack{\text{monomials } M \\ \text{ of degree } k}} M φ(x_M).
\end{equation*}
We conclude that $ φ(x) ∈ \mathfrak{m}^k Y $. This shows $ φ(\mathfrak{m}^k X) ⊂ \mathfrak{m}^k Y $. In particular, $ φ $ is continuous.

For the second step, denote by $ π_i $ the projection maps $ π_i: B \htensor X → B / \mathfrak{m}^i ¤ X $ or $ B \htensor Y → B/ \mathfrak{m}^i ¤ Y $ and by $ π_{ij} $ the projection maps $ π_{ij}: B/\mathfrak{m}^i ¤ Y → B/\mathfrak{m}^j ¤ Y $ for $ i > j $. Note that $ (π_{ij} ¤ \Id_Y) ∘ π_i = π_j $ for $ i > j $.

Let now $ F: X → B \htensor Y $ be a linear map. In order to define a map $ \hat F: B \htensor X → B \htensor Y $, we shall build maps $ \hat F_i: B \htensor X → B / \mathfrak{m}^i ¤ Y $ and then use the universal property to combine them into $ \hat F $.

Let us now construct auxiliary maps $ F_i $ and the maps $ \hat F_i $. Define $ F_i: B / \mathfrak{m}^i ¤ X → B / \mathfrak{m}^i ¤ Y $ as the $ B/\mathfrak{m}^i $-linear extension of $ π_i F: X → B/\mathfrak{m}^i ¤ Y $. Then let $ \hat F_i: B \htensor X → B/\mathfrak{m}^i ¤ Y $ be the composition $ F_i π_i $. More directly, we could write $ \hat F_i (z) = F(π_i (z)) $, where on the right-hand side $ F $ is interpreted $ B/\mathfrak{m}^i $-linearly.

We claim that $ \hat F_j = (π_{ij} ¤ \Id_Y) ∘ \hat F_i $ for $ i > j $. Indeed,
\begin{equation*}
π_{ij} (\hat F_i (z)) = (π_{ij} ¤ \Id_Y) (F(π_i (z)) = F((π_{ij} ¤ \Id_X) (π_i (z))) = F(π_j (z)) = \hat F_j (z).
\end{equation*}
This proves that the family of maps $ \{\hat F_i\}_{i ∈ ℕ} $ is compatible with the projections $ π_{ij} ¤ \Id_Y $. By the universal property of $ B \htensor Y $, this family of maps factors through the limit $ B \htensor Y $, yielding a $ B $-linear map
\begin{equation*}
\hat F: B \htensor X → B \htensor Y, \text{ with } π_i \hat F = \hat F_i.
\end{equation*}
The map $ \hat F $ can also be written explicitly as follows: Let $ x = \sum_{i = 0}^∞ m_i x_i $ be an element of $ B \htensor X $, with $ m_i ∈ \mathfrak{m}^{→∞} $ and $ x_i ∈ X $. Then $ \hat F(\sum m_i x_i) = \sum m_i F(x_i) $.

As third step of the proof, we comment on the one-to-one aspect in the claim. The only remaining statement to explain is that for a given map $ F: X → B \htensor Y $, there is only one single $ B $-linear continuous extension to a map $ B \htensor X → B \htensor Y $. But this is obvious, since $ B $-linearity already determines the value on $ B ¤ X $ and continuity then determines the value on all of $ B \htensor X $. This finishes the proof.
\end{proof}

We shall provide a few more standard utilities. Recall that a map $ φ: X → Y $ of topological spaces is an \emph{embedding} if it is a homeomorphism onto its image, the image being equipped with subspace topology. The \emph{leading term} of a $ B $-linear map $ φ: B \htensor X → B \htensor Y $ is the map $ φ_0: X → Y $ given by the composition $ φ_0 = π φ \restr{X} $, where $ π: B \htensor Y → Y $ denotes the standard projection.

\begin{lemma}
\label{th:prelim-htensor-isocriterion}
Let $ X, Y $ be vector spaces and $ φ: B \htensor X → B \htensor Y $ be $ B $-linear with injective leading term. Then $ φ $ is an embedding and its image is closed. If the leading term is surjective, then $ φ $ is an isomorphism.
\end{lemma}

\begin{proof}
Before we dive into the proof, we start with an observation regarding the leading term. Let $ φ_0: X → Y $ be the leading term of $ φ $ and define for every $ i ∈ ℕ $ the map
\begin{equation}
\label{eq:prelim-htensor-tangentmap}
φ_i: \mathfrak{m}^i / \mathfrak{m}^{i+1} \tensor X → \mathfrak{m}^i / \mathfrak{m}^{i+1} \tensor Y
\end{equation}
induced from $ φ $. Our observation is that $ φ_i $ is in fact equal to $ \id_{\mathfrak{m}^i / \mathfrak{m}^{i+1}} \tensor φ_0 $.

We are now ready to start the proof. First we show that $ φ $ is injective. Second we show that $ φ^{-1}: \Im(φ) → B \htensor X $ is continuous. Third we show that $ \Im(φ) $ is closed. Fourth we prove $ φ $ surjective if $ φ_0 $ is surjective.

For the first part of the proof, we show that $ φ $ injective. Let $ x ∈ B \htensor X $ with $ φ(x) = 0 $. By induction, we show that $ x ∈ \mathfrak{m}^i X $ for every $ i ∈ ℕ $. For $ i = 0 $ this is clear. As induction hypothesis, assume that $ x ∈ \mathfrak{m}^i X $. Then $ φ_i ([x]) = [φ(x)] = 0 $. Since $ φ_i $ is injective, we obtain $ x ∈ \mathfrak{m}^{i+1} X $. This finishes the induction. In consequence we have $ x ∈ \mathfrak{m}^i X $ for every $ i ∈ ℕ $ and therefore $ x = 0 $. This proves $ φ $ injective.

For the second part of the proof, we show that $ φ^{-1} $ preserves the filtration by $ \mathfrak{m} $. In other words, we show that $ φ(x) ∈ \mathfrak{m}^i Y $ implies $ x ∈ \mathfrak{m}^i X $. We do this by induction over $ i ∈ ℕ $. For $ i = 0 $ this is clear. Assume this statements holds for a certain $ i ∈ ℕ $. Now let $ φ(x) ∈ \mathfrak{m}^{i+1} Y $. Then $ x ∈ \mathfrak{m}^i X $ by induction hypothesis. We get $ φ_i ([x]) = 0 $, since $ φ(x) ∈ \mathfrak{m}^{i+1} Y $. Since $ φ_i $ is injective, we get $ x ∈ \mathfrak{m}^{i+1} X $. This finishes the induction. Finally, we have shown $ φ^{-1} (\mathfrak{m}^i Y) ⊂ \mathfrak{m}^i X $. This renders $ φ^{-1}: φ(B \htensor X) → B \htensor X $ continuous.

For the third part of the proof, we check that $ \Im(φ) $ is closed. Since $ B \htensor Y $ is a sequential space, it suffices to check that any sequence $ φ(x_n) ⊂ \Im(φ) $ converging in $ B \htensor Y $ converges in $ \Im(φ) $. Pick a sequence $ (x_n) ⊂ B \htensor X $ and assume $ φ(x_n) → y ∈ B \htensor Y $. This makes $ φ(x_n) - y ∈ \mathfrak{m}^{→ ∞} Y $. We have $ φ(x_n) - φ(x_m) ∈ \mathfrak{m}^{→ ∞} Y $ and by continuity of $ φ^{-1} $ we get $ x_n - x_m ∈ \mathfrak{m}^{→ ∞} $. In particuar, $ x_n $ is a Cauchy sequence and converges to some $ x ∈ B \htensor X $. We get $ φ(x) = \lim φ(x_n) = y $ and therefore $ y ∈ \Im(φ) $. This proves $ \Im(φ) $ closed.

For the fourth part of the proof, consider $ y ∈ B \htensor Y $. We construct inductively a sequence $ (x_n) $ with $ x_n ∈ \mathfrak{m}^n \htensor X $ such that $ φ(\sum_{i = 0}^n x_i) = π_n (y) $. For $ i = 0 $, let $ x_0 ∈ X $ be defined as $ x_0 = φ_0^{-1} (π_0 (y)) $. Now assume the sequence has been constructed for indices until $ n ∈ ℕ $. Put $ z ≔ y - φ(\sum^n x_i) ∈ \mathfrak{m}^n Y $. Since $ φ_0 $ is an isomorphism, the map $ φ_n $ from \eqref{eq:prelim-htensor-tangentmap} is an isomorphism as well. Put
\begin{equation*}
x_{n+1}' ≔ φ_{n+1}^{-1} (z) ∈ \mathfrak{m}^{n+1} / \mathfrak{m}^{n+2} X.
\end{equation*}
Let $ x_{n+1} ∈ \mathfrak{m}^{n+1} X $ be any lift of $ x_{n+1}' $. Then we have
\begin{equation*}
φ \left(\sum_{i = 0}^{n+1} x_i\right) = y - z + φ(x_{n+1}) = y + \landau(\mathfrak{m}^{n+2}).
\end{equation*}
This finishes the inductive construction of the sequence $ (x_i) $. Finally, the series $ x = \sum x_i $ converges and $ φ(x) = y $. This proves $ φ $ surjective.
\end{proof}

The pathway to defining $ A_∞ $-deformations is now clear: When $ \cat C $ is an $ A_∞ $-category, a deformation of $ \cat C $ will always be modeled on the collection of enlarged hom spaces $ \{B \htensor \Hom_{\cat C} (X, Y)\}_{X, Y ∈ \cat C} $. Any $ B $-multilinear product on these hom spaces is automatically continuous. Similarly, functors of $ A_∞ $-deformations will be defined as maps between tensor products of the enlarged hom spaces and will be automatically continuous as well.

%% file: defo/defo.tex
\subsection{Curved deformations}
\label{sec:prelim-ainfty-defo}
In this section, we give a quick definition of curved deformations of $ A_∞ $-categories. We have already used this notion in \paperone. The notion is not surprising and known to experts. Despite being uncomfortable to work with, curvature comes naturally via the Hochschild complex and is necessary if one aims at deformation theory invariant under quasi-equivalences. The curvature in $ A_∞ $-deformations will always be infinitesimal, which we consider harmless in contrast with the different notion of curved $ A_∞ $-categories \cite{Caldararu-Tu}.

Let us recall the setup of Hochschild cohomology in the classical case. Let $ A $ be an associative algebra. As the reader probably knows, Hochschild cohomology $ \HH^2 (A) $ captures the associative deformations of the algebra over $ ℂ[ε] / (ε^2) $. More precisely, a 2-cochain $ ν ∈ \HC^2 (A) $ is a cocycle if and only if $ μ + εν $ is an associative product on $ A ⊕ Aε $ (note $ ε^2 = 0 $). Classical Hochschild cohomology in other degrees than 2 helps characterize the deformation problem, but provides no actual deformations.

Since Hochschild's original definition in 1946, Hochschild cohomology has been generalized to $ A_∞ $-categories. The trick is to use the same formula for the differential, with two adaptations: include also higher products $ μ^k $ instead of only $ μ^2 $, and use grading induced from the shift $ \cat C[1] $. We will make this precise in \autoref{sec:curved-hochschild}.

What should an $ A_∞ $-deformation be then? The naive answer would be to allow deformations of all products $ μ^k $. This is however too shortsighted: We contend that for a notion of $ A_∞ $-deformations whose infinitesimal deformations are classified by a cochain complex (or better a DGLA), the $ A_∞ $-Hochschild complex is the natural choice. Its Maurer-Cartan elements consist however not only of deformations to the products $ μ^k $ with $ k ≥ 1 $, but also introduce curvature $ μ^0 $.

We are now ready to define $ A_∞ $-deformations precisely.

\begin{definition}
\label{def:prelim-defo-def}
Let $ \cat C $ be an $ A_∞ $ category with products $ μ $ and $ B $ a deformation base. An (infinitesimally curved) \emph{deformation} $ \cat C_q $ of $ \cat C $ consists of
\begin{itemize}
\item The same objects as $ \cat C $,
\item Hom spaces $ \Hom_{\cat C_q} (X, Y) = B \htensor \Hom_{\cat C} (X, Y) $ for $ X, Y ∈ \cat C $,
\item $ B $-multilinear products of degree $ 2 - k $
\begin{equation*}
μ_q^k: \Hom_{\cat C_q} (X_k, X_{k+1}) \tensor … \tensor \Hom_{\cat C_q} (X_1, X_2) → \Hom_{\cat C_q} (X_1, X_{k+1}), ~ k ≥ 1
\end{equation*}
\item Curvature of degree $ 2 $ for every object $ X ∈ \cat C $
\begin{equation*}
μ_{q, X}^0 ∈ \mathfrak{m} \Hom_{\cat C_q} (X, X),
\end{equation*}
\end{itemize}
such that $ μ_q $ reduces to $ μ $ once the maximal ideal $ \mathfrak{m} ⊂ R $ is divided out, and $ μ_q $ satisfies the curved $ A_∞ $ ($ cA_∞ $) relations
\begin{equation*}
\sum_{k ≥ l ≥ m ≥ 0} (-1)^{‖a_m‖ + … + ‖a_1‖} μ_q (a_k, …, μ_q (a_l, …), a_m, …, a_1) = 0.
\end{equation*}
The deformation is \emph{unital} if the deformed higher products still satisfy the unitality axioms
\begin{equation*}
μ^2_q (a, \id_X) = a, ~ μ^2_q (\id_Y, a) = (-1)^{|a|} a, ~ μ^{≥3}_q (…, \id_X, …) = 0.
\end{equation*}
\end{definition}

We use the terms $ A_∞ $-deformation and deformation of an $ A_∞ $-category interchangeably. Whenever we speak of deformations of $ A_∞ $-categories in this paper, they are allowed to be (infinitesimally) curved.

\begin{remark}
As explained in \autoref{sec:prelim-htensor}, the product $ μ^k_q $ is automatically $ \mathfrak{m} $-adically continuous. Spelling out this continuity requirement, the datum of the map $ μ^k $ is equivalent to the datum of merely multilinear maps
\begin{equation*}
μ_q: \Hom_{\cat C} (X_k, X_{k+1}) \tensor … \tensor \Hom_{\cat C} (X_1, X_2) ~→~ B \htensor \Hom_{\cat C} (X_1, X_{k+1}).
\end{equation*}
\end{remark}

There is a notion for two objects to be quasi-isomorphic in an $ A_∞ $-deformation. We provide here an ad-hoc definition which seems odd at first, but we will encounter evidence in \autoref{sec:constructions} and \ref{sec:kadeishvili} which supports correctness of the definition.

\begin{definition}
\label{def:prelim-defo-qi}
Let $ \cat C $ be an $ A_∞ $-category and $ \cat C_q $ a deformation. Let $ X, Y ∈ \cat C_q $ be two objects. Then $ X $ and $ Y $ are \emph{quasi-isomorphic} if they are quasi-isomorphic in $ \cat C $.
\end{definition}

It sometimes comes handy to work with deformations that include more objects than $ \cat C $ does. We fix terminology as follows:

\begin{definition}
Let $ \cat C $ be an $ A_∞ $-category. Let $ O $ be an arbitrary set of objects and $ F: O → \Ob(\cat C) $ a map. Then the \emph{object-cloned version} $ F^* \cat C $ of $ \cat C $ is the $ A_∞ $-category given by object set $ O $, hom spaces
\begin{equation*}
\Hom_{F^* \cat C} (X, Y) = \Hom_{\cat C} (F(X), F(Y)), \quad X, Y ∈ O,
\end{equation*}
and products simply given by the same composition as in $ \cat C $.
\end{definition}

An $ A_∞ $-deformation of $ \cat C $ always gives an induced deformation of $ F^* \cat C $. This provides a map from deformations of $ \cat C $ to deformations of $ F^* \cat C $. In case $ F $ is surjective, the categories $ \cat C $ and $ F^* \cat C $ are equivalent and as we elaborate in \autoref{sec:lens} actually have the same deformation theory. For this reason, we fix the following terminology:

\begin{definition}
\label{def:prelim-ainfty-objectcloning}
Let $ \cat C $ be an $ A_∞ $-category and $ B $ a deformation base. Let $ O $ be an arbitrary set and $ F: O → \Ob\cat C $ a map. An \emph{object-cloning deformation} is a deformation $ \cat D_q $ of $ \cat D = F^* \cat C $. The object-cloning deformation is \emph{essentially surjective} if $ F: \Ob\cat D → \Ob\cat C $ reaches all objects of $ \cat C $ up to isomorphism.
\end{definition}

%% file: defo/twisted.tex
\subsection{Twisted completion}
\label{sec:prelim-ainfty-twisted}
In this section, we define the twisted completion of a deformed $ A_∞ $-category. The starting point is an $ A_∞ $-category $ \cat C $ together with a deformation $ \cat C_q $. The aim is to define a deformation $ \Tw\cat C_q $ of $ \Tw\cat C $ which incorporates the deformation given by $ \cat C_q $. More formally, the aim is to push the deformation $ \cat C_q $ onto $ \Tw\cat C $ along the inclusion $ \cat C → \Tw\cat C $ which is a derived equivalence. Since we only explain the pushforward of deformations in \autoref{sec:lens}, we give here a direct and explicit definition of $ \Tw\cat C_q $.

We start by defining additive completions for deformed $ A_∞ $-categories. Let $ \cat C $ be an $ A_∞ $-category and $ \cat C_q $ a deformation. The aim is to define an additive completion $ \Add\cat C_q $ in such a way that it is a deformation of $ \Add\cat C $. This is straight-forward:

\begin{definition}
Let $ \cat C $ be an $ A_∞ $-category and $ \cat C_q $ a deformation. Then the \emph{additive completion} $ \Add\cat C_q $ is the deformation of $ \Add\cat C $ given by the following deformed product:
\begin{equation*}
μ_{\Add\cat C_q}^k (α_k, …, α_1) = \sum (-1)^{\sum_{j < i} \Vert α_i \Vert l_j} μ_{\cat C_q}^k (α_k, …, α_1),
\end{equation*}
with the same sign convention as in the non-deformed case.
\end{definition}

\begin{remark}
The only difference between $ μ_{\Add \cat C} $ and $ μ_{\Add \cat C_q} $ lies in using the non-deformed product $ μ_{\cat C} $ for the former and the deformed product $ μ_{\cat C_q} $ for the latter.
\end{remark}

We now extend the twisted completion construction to the case of $ A_∞ $-deformations. Let $ \cat C $ be an $ A_∞ $-category and $ \cat C_q $ a deformation. The aim is to define a twisted completion $ \Tw\cat C_q $ in such a way that it is a deformation of $ \Tw\cat C $. In particular, the category $ \Tw\cat C_q $ should have the same objects as $ \Tw\cat C $. With this in mind, we define:

\begin{definition}
\label{def:curved-twisted}
Let $ \cat C $ be an $ A_∞ $ category with products $ μ_{\cat C} $ and $ \cat C_q $ a deformation with products $ μ_{\cat C_q} $. Then the \emph{twisted completion} $ \Tw \cat C_q $ is the (possibly curved) deformation of $ \Tw \cat C $ given by the deformed products
\begin{equation*}
μ_{\Tw \cat C_q}^k (α_k, …, α_1) = \sum_{n_0, …, n_k ≥ 0} μ_{\Add \cat C_q} (\underbrace{δ, …, δ}_{n_k}, α_k, …, α_1, \underbrace{δ, …, δ}_{n_0}).
\end{equation*}
\end{definition}

\begin{remark}
As one may have expected, the product $ μ_{\Tw \cat C_q} $ now simply uses $ μ_{\cat C_q} $ instead of $ μ_{\cat C} $. The set of objects of $ \Tw\cat C_q $ may however be surprising: The objects of $ \Tw\cat C_q $ are not formed with twisted differentials $ δ ∈ \Hom^1_{\cat C_q} (X, X) $. Instead, the objects of $ \Tw\cat C_q $ are twisted complexes $ (X, δ) ∈ \Tw\cat C $, in other words, their $ δ $-differential must lie in $ \Hom^1_{\cat C} (X, X) $. It is easily checked that $ μ_{\Tw\cat C_q} $ satisfies the (curved) $ A_∞ $-axioms, rendering $ \Tw\cat C_q $ a genuine $ A_∞ $-deformation of $ \Tw\cat C $.
\end{remark}

\begin{remark}
Denoting the twisted completion of $ \cat C_q $ by $ \Tw\cat C_q $ constitutes a slight abuse of notation: “$ \Tw\cat C_q $” suggests that twisted complexes are to be taken with the $ δ $-differential formed from values in $ \cat C_q $, which is not the case. A more proper notation for the twisted completion of $ \cat C_q $ would have been $ (\Tw\cat C)_q $, which we however found too complicated.
\end{remark}

\begin{remark}
Typical objects in $ \Tw\cat C_q $ have curvature. Indeed, according to \autoref{def:curved-twisted}, an object $ (X, δ) ∈ \Tw\cat C_q $ has curvature
\begin{equation*}
μ^0_{(X, δ)} = μ^0_{X, \Add\cat C_q} + μ^1_{\Add\cat C_q} (δ) + μ^2_{\Add\cat C_q} (δ, δ) + ….
\end{equation*}
The curvature $ μ^0_{X, \Add\cat C_q} $ can be spelled out more concretely as the sum of the curvatures of all constituents $ A_i $ of $ X $. The differential $ μ^1_{\Add\cat C_q} (δ) $ is concretely the sum of the deformed differentials $ μ^1_{\cat C_q} $ applied to each entry of $ δ $ as a matrix. Whatever the value of $ μ^0_{(X, δ)} $ adds up to, the reader can see that it does typically not vanish because the twisted differential $ δ ∈ \Hom^1 (X, X) $ only satisfies the Maurer-Cartan equation with respect to the non-deformed differential $ μ^1_{\cat C} $.
\end{remark}

\begin{remark}
A popular way to form twisted completions of curved categories is to pick only curvature-free twisted complexes. This might be the way to luck in case of curved $ A_∞ $-categories, because a curved twisted completion cannot be passed to the minimal model. In contrast, for $ A_∞ $-deformations, curvature on $ \Tw\cat C_q $ is not problematic at all. Picking all twisted complexes of $ \cat C $ is in fact necessary in order to render $ \Tw\cat C_q $ a deformation of $ \Tw\cat C $.
\end{remark}

%% file: defo/functors.tex
\subsection{Functors}
\label{sec:curved-functors}
In this section, we define the notion of functors between deformed $ A_∞ $-categories. These functors serve as a framework for gauge equivalences, quasi-equivalences and pushforwards of $ A_∞ $-deformations. This class of functors is presumably known to experts. Our sign conventions are those of \cite{HKK} and \cite[Section 3.1.4/symplectic]{Bocklandt-book}.

Based on the classical definition of $ A_∞ $-functors which we recalled in \autoref{def:prelim-ainfty-functor}, we are now ready to explain the natural extension to the deformed case. We fix the definition as follows:

\begin{definition}
\label{def:curved-functors-def}
Let $ \cat C, \cat D $ be two $ A_∞ $-categories and $ \cat C_q, \cat D_q $ deformations. A \emph{functor of deformed $ A_∞ $-categories} consists of a map $ F_q: \Ob(\cat C) → \Ob(\cat D) $ together with for every $ k ≥ 1 $ a $ B $-multilinear degree $ 1-k $ map
\begin{equation*}
F^k_q: \Hom_{\cat C_q} (X_k, X_{k+1}) ¤ … ¤ \Hom_{\cat C_q} (X_1, X_2) → \Hom_{\cat D_q} (F_q X_1, F_q X_{k+1})
\end{equation*}
and infinitesimal curvature $ F^0_{q, X} ∈ \mathfrak{m} \Hom^1_{\cat D} (F_q X, F_q X) $ for every $ X ∈ \cat C $, such that the curved $ A_∞ $-functor relations hold:
\begin{multline*}
\sum_{0 ≤ j ≤ i ≤ k} (-1)^{‖a_j‖ + … + ‖a_1‖} F_q (a_k, …, a_{i+1}, μ_q (a_i, …, a_{j+1}), a_j, …, a_1) \\
= \sum_{\substack{l ≥ 0 \\ 1 = j_1 ≤ … ≤ j_l ≤ k}} μ_q (F_q(a_k, …, a_{j_l}), …, F_q(…, a_{j_2}), F_q(…, a_{j_1})).
\end{multline*}
\end{definition}

\begin{remark}
A functor $ F_q $ of deformed $ A_∞ $-categories consists of maps between hom spaces which are allowed to have deformed (nonconstant) terms themselves. As we have seen in \autoref{sec:prelim-htensor}, the maps $ F_q^k $ are automatically continuous. Apart from the components $ F_q^{≥1} $, the functor is allowed to have infinitesimal curvature $ F^0_q $. This curvature is a feature of the deformed world where infinitesimal curvature is not only welcome, but is necessary. The first two curved $ A_∞ $-functor relations read
\begin{align*}
F^0_q + F^1_q (μ^0_{\cat C_q, X}) &= μ^1_{\cat D_q} (F^0_{q, X}), \\
F^1_q (μ^1_{\cat C_q} (a)) + (-1)^{‖a‖} F^2_q (μ^0_{\cat C_q, Y}, a) + F^2_q (a, μ^0_{\cat C_q, X}) &= μ^1_{\cat D_q} (F^1_q (a)) + μ^2_{\cat D_q} (F^0_{q, Y}, F^1_q (a)) \\
& \quad + μ^2_{\cat D_q} (F^1_q (a), F^0_{q, X}), \quad ∀ ~ a: X → Y.
\end{align*}
\end{remark}

We can interpret a functor $ F_q: \cat C_q → \cat D_q $ as an extension of an ordinary $ A_∞ $-functor $ \cat C → \cat D $. Indeed, when forgetting the terms of $ F = \{F_q^k\}_{k ≥ 0} $ in higher $ \mathfrak{m} $-adic order, we get a collection of maps $ F = \{F^k\}_{k ≥ 1} $. Since $ F_q $ satisfies the curved $ A_∞ $-functor relations, its reduction $ F $ satisfies the ordinary $ A_∞ $-functor relations. We fix this terminology as follows:

\begin{definition}
Let $ F_q: \cat C_q → \cat D_q $ be a functor of deformed $ A_∞ $-categories. Then its \emph{leading term} is the functor $ F: \cat C → \cat D $ obtained by dividing out the maximal ideal $ \mathfrak{m} $.
\end{definition}

Let us now define various notions of equivalences between deformed $ A_∞ $-categories. The aim is to transfer the notions of equivalences between $ A_∞ $-categories recalled in \autoref{def:prelim-ainfty-functorquasi} and \ref{def:prelim-ainfty-catquasi} to the world of deformed $ A_∞ $-categories. The main idea is to declare a functor $ F_q: \cat C_q → \cat D_q $ an isomorphism if its leading term $ F: \cat C → \cat D $ is an isomorphism. Note that this definition is not vacuous: It still requires the functor $ F_q $ to satisfy the (curved) $ A_∞ $-relations, but displays the isomorphism property as a side issue which is dealt with on the leading term part.

\begin{definition}
\label{def:curved-functors-quasiequivalence}
Let $ \cat C, \cat D $ be $ A_∞ $-categories and $ \cat C_q, \cat D_q $ deformations. Let $ F_q: \cat C_q → \cat D_q $ be a functor of deformed $ A_∞ $-categories and denote by $ F: \cat C → \cat D $ its leading term. Then $ F_q $ is
\begin{itemize}
\item an \emph{isomorphism} if $ F $ is an isomorphism,
\item a \emph{quasi-isomorphism} if $ F $ is a quasi-isomorphism,
\item a \emph{quasi-equivalence} if $ F $ is a quasi-equivalence,
\item a \emph{derived equivalence} if $ F $ is a derived equivalence.
\end{itemize}
Two deformed $ A_∞ $-categories $ \cat C_q $ and $ \cat D_q $ are
\begin{itemize}
\item \emph{isomorphic} if there is an isomorphism $ F_q: \cat C_q → \cat D_q $,
\item \emph{quasi-isomorphic} if there is a quasi-isomorphism $ F_q: \cat C_q → \cat D_q $,
\item \emph{quasi-equivalent} if there is a quasi-equivalence $ F_q: \cat C_q → \cat D_q $,
\item \emph{derived equivalent} if $ \Tw\cat C_q $ and $ \Tw\cat D_q $ are quasi-equivalent.
\end{itemize}
\end{definition}

We explain in \autoref{th:curved-axioms-equivrelation} why relations such as quasi-equivalence and derived equivalence are equivalence relations among deformed $ A_∞ $-categories. Among the $ A_∞ $-deformations of one single category $ \cat C $, there is one further notion of equivalence, known as gauge equivalence:

\begin{definition}
\label{def:curved-gauge-functordef}
Let $ \cat C $ be an $ A_∞ $-category and $ \cat C_q, \cat C_q' $ be deformations. Then a \emph{gauge equivalence} between $ \cat C_q $ and $ \cat C_q' $ is a functor $ F_q: \cat C_q → \cat C_q' $ of deformed $ A_∞ $-categories whose leading term $ F: \cat C → \cat C $ is the identity.
\end{definition}

We elaborate gauge equivalence further in the context of the Hochschild DGLA in \autoref{sec:curved-gauge}.

%% file: dgla/intro.tex
\section{The lens of DGLAs}
\label{sec:lens}
In this section, we show how to abstractly transfer $ A_∞ $-deformations from one category onto another.

In \autoref{sec:curved-hochschild}, we recall how to view $ A_∞ $-deformations as Maurer-Cartan elements via the so-called Hochschild DGLA. In \autoref{sec:curved-gauge}, we recall the notion of gauge equivalence of Maurer-Cartan elements and explain its application in the specific case of $ A_∞ $-deformations. In \autoref{sec:curved-linfty}, we recall the generalization of DGLAs known as $ L_∞ $-algebras. In \autoref{sec:curved-axioms}, we explain how to push forward deformations between quasi-equivalent categories.

For the interpretation of gauge equivalences and pushforwards of deformations in terms of $ L_∞ $-theory we will not provide any proof. Rather, we state in \autoref{sec:curved-axioms} a collection of axioms on the Hochschild DGLA which we will just assume in this paper.

The material in this section is not original. The Hochschild DGLA is classical at least in the case of ordinary algebras. Keller proved its invariance under quasi-equivalence in the dg case in \cite{Keller}.

Instead, this section serves to support a very specific viewpoint on $ A_∞ $-deformations. While there is widespread belief that curvature is a hindrance, we want to demonstrate here that it is at least possible to live peacefully with infinitesimal curvature. Let us present which section supports which claim: From \autoref{sec:prelim-ainfty-twisted} we learn that the twisted completion of an $ A_∞ $-deformation can be formed even under the presence of infinitesimal curvature. This construction is entirely natural in the sense that the deformed twisted completion is a deformation of the twisted completion of the non-deformed $ A_∞ $-category. From \autoref{sec:curved-hochschild} we learn that infinitesimal curvature is inevitable for a good notion of $ A_∞ $-deformations. The core summary is that curvature does not hurt for constructions with $ A_∞ $-categories because it enters the picture automatically when a deformation is induced from one category to another.

%% file: dgla/hochschild.tex
\subsection{The Hochschild DGLA}
\label{sec:curved-hochschild}
In this section, we recall how to view $ A_∞ $-deformations from the DGLA point of view. The material is all known to experts and nicely shows how curvature enters the picture. Useful references include \cite[Chapter V]{Manetti} and \cite{Yekutieli}.

Before we recall the concept of DGLA and Maurer-Cartan elements, let us summarize the philosophy: The aim is to capture every deformation problem by a DGLA. The solutions of the deformation problem should then correspond to Maurer-Cartan elements of the DGLA. This empowers the mathematician to use the force of DGLA theory. Standard questions in DGLAs include: to find a quasi-isomorphism $ F: L → L' $ between two DGLAs, or to classify all Maurer-Cartan elements of $ L $ up to gauge equivalence.
\begin{center}
\begin{tikzpicture}
\path (0, 0) node (A) {deformation problem} (5, 0) node (DGLA) {DGLA};
\path[draw, ->] ($ (A.east)!0.2!(DGLA.west) $) to node[midway, above] {1. reformulation} ($ (A.east)!0.8!(DGLA.west) $);
\path[draw, ->, bend right=100, looseness=5] (DGLA.south east) to node[pos=0.7, above] {2. algebraic power} (DGLA.north east);
\path[draw, ->, bend left] (5, -0.5) to node[midway, above] {3. harvest} (1, -0.5);
\end{tikzpicture}
\end{center}

We are now ready to recall the notion of DGLAs and Maurer-Cartan elements.

\begin{definition}
\label{def:curved-hochschild-DGLA}
A \emph{DG Lie algebra} (DGLA) is a $ ℤ $- or $ ℤ/2ℤ $-graded vector space $ L $ together with
\begin{itemize}
\item a differential $ d: L^i → L^{i+1} $,
\item a bracket $ [-, -]: L × L → L $ of degree zero,
\end{itemize}
satisfying skew-symmetry, Leibniz rule and Jacobi identity:
\begin{align*}
& [a, b] = (-1)^{|a||b| + 1} [b, a], \\
& d([a, b]) = [da, b] + (-1)^{|a|} [a, db], \\
& (-1)^{|a||c|} [a, [b, c]] + (-1)^{|b||a|} [b, [c, a]] + (-1)^{|c||b|} [c, [a, b]] = 0.
\end{align*}
\end{definition}

For example, the bracket $ [-, -] $ is commutative on odd elements. With this consideration, we recall the definition of Maurer-Cartan elements:

\begin{definition}
\label{def:curved-hochschild-MC}
Let $ L $ be a DGLA and $ B $ a deformation base. Regard the tensored DGLA $ B \htensor L $ with differential $ d $ and bracket $ [-, -] $ simply extended continuously. A \emph{Maurer-Cartan element} of $ L $ over $ B $ is an element $ ν ∈ B \htensor L^1 $ which satisfies the Maurer-Cartan equation
\begin{equation*}
dν + \frac{1}{2} [ν, ν] = 0.
\end{equation*}
The set of Maurer-Cartan elements of $ L $ over $ B $ is denoted $ \MC(L, B) $. In case the DGLA $ L $ is $ ℤ/2ℤ $-graded, a Maurer-Cartan element is supposed to lie in $ B \htensor L^{\odd} $.
\end{definition}

In the rest of this section, we will make sense of these definitions in the case of the so-called Hochschild DGLA:

\begin{definition}
\label{def:curved-hochschild-HC}
Let $ \cat C $ be a $ ℤ $- or $ ℤ/2ℤ $-graded $ A_∞ $-category. Then its Hochschild complex $ \HC(\cat C) $ is the graded vector space
\begin{equation*}
\HC(\cat C) = \prod_{\substack{X_1, …, X_{k+1} ∈ \cat C \\ k ≥ 0}} \Hom\big(\Hom_{\cat C} (X_k, X_{k+1})[1] \tensor … \tensor \Hom_{\cat C} (X_1, X_2)[1], ~ \Hom_{\cat C} (X_1, X_{k+1})[1]\big).
\end{equation*}
Here $ [1] $ denotes the left-shift and $ ‖a‖ = |a| - 1 $ denotes the reduced degree of a morphism $ a ∈ \Hom_{\cat C} (X, Y) $. The grading $ ‖·‖ $ on $ \HC(\cat C) $ is the one induced from the shifted degrees of the hom spaces of $ \cat C $. In other words, we have
\begin{equation*}
‖η(a_k, …, a_1)‖ = ‖η‖ + ‖a_k‖ + … + ‖a_1‖, \quad η ∈ \HC(\cat C).
\end{equation*}
For $ η, ω ∈ \HC(\cat C) $, the Gerstenhaber product $ μ · ω ∈ \HC(\cat C) $ is defined as
\begin{equation*}
(η · ω) (a_k, …, a_1) = \sum (-1)^{(∥a_l∥ + … + \Vert a_1 \Vert)∥ω∥} η(a_k, …, ω(…), a_l, …, a_1).
\end{equation*}
The \emph{Hochschild DGLA} is the following $ ℤ $- or $ ℤ/2ℤ $-graded DGLA structure on $ \HC(\cat C) $: The bracket on $ \HC(\cat C) $ is the Gerstenhaber bracket
\begin{equation*}
[η, ω] = η · ω - (-1)^{‖ω‖ ‖η‖} ω · η.
\end{equation*}
The differential on $ \HC(\cat C) $ consists of commuting with the product $ μ_{\cat C} ∈ \HC^1 (\cat C) $:
\begin{equation*}
dν = [μ_{\cat C}, ν].
\end{equation*}
\end{definition}

\begin{remark}
It is not hard to check that $ \HC(\cat C) $ is indeed a DGLA. The reader who wishes to perform the computation is advised to write all double brackets in terms of Gerstenhaber products and use the associator relation
\begin{align*}
& (a·b)·c - a·(b·c) \\
&= \sum (-1)^{‖b‖ (‖a_1‖ + … + ‖a_i‖) + ‖c‖ (‖a_1‖ + … + ‖a_j‖) + ‖b‖‖c‖} a(…, b(…), a_i, …, c(…), a_j, …) \\
& + \sum (-1)^{‖b‖(‖a_1‖ + … + ‖a_j‖) + ‖c‖(‖a_1‖ + … + ‖a_i‖)} a(…, c(…), a_i, …, b(…), a_j, …)).
\end{align*}
The sum in these formulas runs over all ways to insert $ b $ and $ c $ into $ a $. Despite the way the formulas are written, it is not necessary that any elements $ a_i $ or $ a_j $ actually lie in between or behind $ b $ and $ c $. This is merely an artifact needed to define the sign right: the sign shall be the total reduced degree of all elements coming after $ b $ or after $ c $, respectively.
\end{remark}

\begin{remark}
In the terminology of the Hochschild DGLA $ \HC(\cat C) $, we can interpret $ A_∞ $-deformations of $ \cat C $ precisely as Maurer-Cartan elements of $ \HC(\cat C) $. We start by observing that the $ A_∞ $-product $ μ_{\cat C} $ amounts to an element $ μ_{\cat C} ∈ \HC^1 (\cat C) $ and the $ A_∞ $-relations translate to $ μ_{\cat C} · μ_{\cat C} = 0 $.

Now let $ \cat C_q $ be a deformation of $ \cat C $. Then the curved $ A_∞ $-relations for $ μ_{\cat C_q} $ translate to $ μ_{\cat C_q} · μ_{\cat C_q} = 0 $. Decompose $ μ_{\cat C_q} = μ_{\cat C} + ν $ as non-deformed part $ μ_{\cat C} $ plus deformation $ ν ∈ \mathfrak{m} \HC^1 (\cat C) $. Given that $ μ_{\cat C} $ already satisfies the $ A_∞ $-relation $ μ_{\cat C} · μ_{\cat C} = 0 $, the element $ ν $ itself satisfies the Maurer-Cartan equation
\begin{equation*}
0 = (μ_{\cat C} + ν) · (μ_{\cat C} + ν) = dν + [ν, ν]/2.
\end{equation*}
Conversely, pick a Maurer-Cartan element $ ν ∈ \MC(\HC(\cat C), B) $. According to \autoref{th:prelim-htensor-autocontinuous}, the element $ μ_{\cat C} + ν $ extends in a multilinear and $ \mathfrak{m} $-adically continuous way to a collection of mappings
\begin{equation*}
μ_{\cat C_q}^{k≥0}: (B \htensor \Hom_{\cat C} (X_k, X_{k+1})) ¤ … ¤ (B \htensor \Hom_{\cat C} (X_1, X_2)) → B \htensor \Hom_{\cat C} (X_1, X_{k+1}).
\end{equation*}
The Maurer-Cartan identity for $ ν $ makes that $ μ_{\cat C_q} $ satisfies the curved $ A_∞ $-relations.
\end{remark}

\begin{remark}
In case $ L $ is a $ ℤ/2ℤ $-graded DGLA, then Maurer-Cartan elements of $ L $ are by definition odd elements $ ν ∈ \mathfrak{m} L^{\odd} $ with $ dν + [ν, ν]/2 = 0 $. For example, our deformation $ \Gtl_q Q $ is only a $ ℤ/2ℤ $-graded deformation of $ \Gtl Q $. In the context of deformations, we have to view both $ \Gtl Q $ and its Hochschild DGLA $ \HC(\Gtl Q) $ as $ ℤ/2ℤ $-graded.
\end{remark}

\begin{remark}
For ordinary algebras, which are concentrated in degree zero and have vanishing higher products, the Hochschild cohomology is typically defined without the shifts. This results in a grading difference of $ 1 $ from what we present here. For example, the center of the algebra is the classical zeroth Hochschild cohomology. In our $ A_∞ $-setting, this cohomology will rather be found in degree $ -1 $.
\end{remark}

%% file: dgla/gauge.tex
\subsection{Gauge equivalence}
\label{sec:curved-gauge}
In this section, we recall the notion of gauge equivalence. By virtue of algebraic deformation theory, we have two ways of defining this equivalence: via an explicit definition and via the Hochschild DGLA. Both notions are defined here in parallel. Useful references are \cite{Manetti} and \cite{Yekutieli}.

\begin{center}
\begin{tikzpicture}
\path (0, 0) node[align=center] (A) {\textbf{Gauge equivalence} \\ by functor $ F: \cat C_q → \cat C_q' $} (8, 0) node[align=center] (B) {\textbf{Gauge equivalence} \\ as MC elements $ μ_{\cat C_q} \sim μ_{\cat C_q'} $};
\path[draw, <->] ($ (A.east)!0.2!(B.west) $) -- ($ (A.east)!0.8!(B.west) $);
\end{tikzpicture}
\end{center}

Recall from \autoref{def:curved-gauge-functordef} that a gauge equivalence between two deformations $ \cat C_q $, $ \cat C_q' $ of an $ A_∞ $-category $ \cat C $ consists of a functor $ F: \cat C_q → \cat C_q' $ whose leading term is the identity.

\begin{remark}
The idea behind \autoref{def:curved-gauge-functordef} is that the set of automorphisms $ \cat C_q → \cat C_q' $ as deformed $ A_∞ $-categories is rather large. The leading term of an automorphism can be any automorphism of $ \cat C $, which is not interesting from the perspective of deformation theory. Therefore one restricts to those functors whose leading term is the identity on $ \cat C $.
\end{remark}

After the success of the notion of gauge equivalence throughout mathematics and physics, a definition has also been captured in the abstract DGLA approach. The idea here is that the gauging functor $ F_q $ can be viewed as an element of a “gauge group”. The infinitesimal action of this gauge group can be described purely in terms of the DGLA:

\begin{definition}
\label{def:curved-gauge-def}
Let $ L $ be a DGLA and $ B $ a deformation base. Then there is a group action by $ \exp(\mathfrak{m} L^0) $ on $ B \htensor L^1 $ with infinitesimal generator
\begin{equation*}
φ.ν = dφ + [ν, φ] ∈ B \htensor L^1, \quad ν ∈ \mathfrak{m} L^0.
\end{equation*}
The group action preserves the set of Maurer-Cartan elements $ \MC(L; B) $. Two Maurer-Cartan elements $ ν, ν' ∈ \MC(L; B) $ are \emph{gauge-equivalent} if they lie in the same orbit. The set of Maurer-Cartan elements up to gauge-equivalence is denoted $ \MCb(L; B) $.
\end{definition}

\begin{remark}
In case $ L $ is a $ ℤ/2ℤ $-graded DGLA, then the gauge group is $ \exp(\mathfrak{m} L^{\even}) $. This group acts on Maurer-Cartan elements of $ L $, which are by definition odd elements $ ν ∈ \mathfrak{m} L^{\odd} $ with $ dν + [ν, ν]/2 = 0 $.
\end{remark}

Let us compare the infinitesimal generators of the action, at the non-deformed product $ μ_{\cat C} $: An “infinitesimal functor” $ \Id + εF $ pushes the non-deformed product $ μ_{\cat C} $ to some $ μ' $ such that $ (\Id + εF) · μ_{\cat C} = μ' ∘ (\Id + εF) $. Here $ ∘ $ denotes functor composition, in contrast to the Gerstenhaber product “$ · $”. Setting $ ε^2 = 0 $, we read off
\begin{equation*}
μ' = μ_{\cat C} + ε [μ_{\cat C}, F].
\end{equation*}
Meanwhile, the trivial deformation $ μ_{\cat C} $ corresponds to the zero Maurer-Cartan element in $ \MC(\HC(\cat C)) $. Gauging it by $ εF $ under $ ε^2 = 0 $ gives the element
\begin{equation*}
(εF).0 = 0 + d(εF) + [0, εF] = ε [μ_{\cat C}, F].
\end{equation*}
This element corresponds to the deformation $ μ_{\cat C} + ε [μ_{\cat C}, F] $. We see that gauging a deformation by a gauge functor $ \Id + εF $ is the same as gauging its corresponding Maurer-Cartan element in the Hochschild DGLA:
\begin{center}
\begin{tikzpicture}
\path (0, 0) node[align=center] (A) {\textbf{Infinitesimal functor} $ \Id + εF $ \\ $ μ' = μ_{\cat C} + ε [μ_{\cat C}, F] $} (9, 0) node[align=center] (B) {\textbf{Infinitesimal DGLA gauge} $ εF $ \\ $ μ' = μ_{\cat C} + (d(εF) + [0, εF]) $};
\path[draw, <->] ($ (A.east) + (1, 0) $) -- ($ (B.west) + (-1, 0) $);
\end{tikzpicture}
\end{center}
More precisely, two deformations are gauge equivalent in the sense of \autoref{def:curved-gauge-functordef} if and only if their corresponding Maurer-Cartan elements are gauge equivalent in the sense of \autoref{def:curved-gauge-def}.

\begin{remark}
$ A_∞ $-Hochschild cohomology of $ \cat C $ is defined as the cohomology of $ \HC(\cat C) $, merely regarded as a cochain complex instead of DGLA. This way, Hochschild cohomology is the linear approximation of Maurer-Cartan elements up to gauge equivalence.
\end{remark}

%% file: dgla/Linfty.tex
\subsection{$ L_∞ $-algebras}
\label{sec:curved-linfty}
In this section we recall the formalism of $ L_∞ $-algebras. In our context of $ A_∞ $-categories, we namely want to transport deformations from one category to another, so that one needs morphisms between their Hochschild DGLAs. The world of DGLAs and their morphisms is quite restrictive, but there is a more flexible version of DGLAs known as $ L_∞ $-algebras. We finish this section with a definition and brief discussion of the $ L_∞ $-theory. We follow version 3 of \cite{Barmeier-Wang}. The reader be cautioned that version 4 of that paper has changed signs.

\begin{definition}
An \emph{$ L_∞ $-algebra} is a graded vector space $ L $ together with multilinear maps
\begin{equation*}
l_k: \underbrace{L × … × L}_{k \text{ times}} → L
\end{equation*}
of degree $ 2-k $ which are graded skew-symmetric and satisfy the higher Jacobi identities:
\begin{align*}
&l_k (x_{s(1)}, …, x_{s(k)}) = χ(s) l_k (x_1, …, x_k), \\
&\sum_{\substack{i+j = k+1 \\ i, j ≥ 1}} \sum_{s ∈ S_{i, k-i}} (-1)^{i(n-i)} χ(s) l_j (l_i (x_{s(1)}, …, x_{s(i)}), x_{s(i+1)}, …, x_{s(k)}) = 0.
\end{align*}
Here $ S_{i, k-i} $ denotes the set of shuffles, i.e.~$ s ∈ S_k $ with $ s(1) < … < s(i) $ and $ s(i+1) < … < s(k) $. The sign $ χ(s) $ is the product of the signum of $ s $ and the Koszul sign of $ s $. The Koszul sign of a transposition $ (i ~ j) $ is $ (-1)^{|x_i| |x_j|} $, and this rule is extended multiplicatively to arbitrary permutations.
\end{definition}

With this sign convention, a DGLA as in \autoref{def:curved-hochschild-DGLA} is simply an $ L_∞ $-algebra without higher products. In particular, the Hochschild DGLA can automatically be regarded as an $ L_∞ $-algebra.

Morphisms between $ L_∞ $-algebras are indeed more flexible than between DGLAs: $ L_∞ $-morphisms are allowed to have higher components, just like $ A_∞ $-functors allow for higher components.

\begin{definition}
A \emph{morphism of $ L_∞ $-algebras} $ φ: L → L' $ is given by multilinear maps
\begin{equation*}
φ^k: L × … × L → L
\end{equation*}
of degree $ 1-k $ for all $ k ≥ 1 $ such that $ φ(x_{s(1)}, …, x_{s(k)}) = χ(s) φ(x_1, …, x_k) $ for any $ s ∈ S_k $ and
\begin{align*}
&\sum_{\substack{i+j = k+1 \\ i, j ≥ 1}} \sum_{s ∈ S_{i, n-i}} (-1)^{i(k-i)} χ(s) φ(l(x_{s(1)}, …, x_{s(i)}), x_{s(i+1)}, …, x_{s(k)}) \\
&= \sum_{\substack{1 ≤ r ≤ k \\ i_1 + … + i_r = k}} \sum_t (-1)^u χ(t) l'_r (φ(x_{t(1)}, …, x_{t(i_1)}), …, φ(x_{t(i_1 + … + i_{r-1} + 1)}, …, x_{t(k)})),
\end{align*}
where $ t $ runs over all $ (i_1, …, i_r) $-shuffles for which
\begin{equation*}
t(i_1 + … + i_{l-1} + 1) < t(i_1 + … + i_l + 1).
\end{equation*}
and $ u = (r-1) (i_1 - 1) + … + 2 (i_{r-2} - 1) + (i_{r-1} - 1) $. A morphism $ φ: L → L' $ of $ L_∞ $-algebras is a \emph{quasi-isomorphism} if $ φ^1 $ is a quasi-isomorphism of complexes.
\end{definition}

\begin{definition}
Let $ L $ be an $ L_∞ $-algebra and $ B $ a deformation base. Then a \emph{Maurer-Cartan element} is an element $ x ∈ \mathfrak{m} \htensor L^1 $ satisfying the Maurer-Cartan equation
\begin{equation*}
\sum_{k ≥ 1} \frac{l_k (x, …, x)}{k!} = 0.
\end{equation*}
We write $ \MC(L, B) $ for the set of Maurer-Cartan elements of $ L $ over $ B $.
\end{definition}

In contrast to the DGLA case, there is no gauge group acting on the Maurer-Cartan elements. Instead, one should regard an equivalence relation of “homotopy”. All we should assume here is that the notion of homotopy exists and gives rise to a quotient set $ \MCb(L, B) $, just as in the DGLA case.

%% file: dgla/axioms.tex
\subsection{Axioms on $ A_∞ $-deformations}
\label{sec:curved-axioms}
There is a slight gap in our treatment of $ A_∞ $-deformations: We cannot prove invariance of the Hochschild DGLA under quasi-equivalences. While derived invariance is known in the dg case due to \cite{Keller}, according to private communication with Keller there is no literature available for the $ A_∞ $-case. We do not fill the gap here. Instead, we state a small set of axioms in this section which we will simply assume for the purpose of \autoref{sec:uncurving-theory} and \ref{sec:constructions-minmodel}.

The motivation for our axioms is that we need to be able to push deformations from one category to another. Assume two $ A_∞ $-categories $ \cat C $ and $ \cat D $ are quasi-equivalent by means of a quasi-equivalence $ F: \cat C → \cat D $. Intuition says that a deformation $ \cat C_q $ of $ \cat C $ can be “pushed” via $ F $ to a deformation $ \cat D_q = F_* (\cat C_q) $ of $ \cat D $ such that $ \cat C_q $ and $ \cat D_q $ still quasi-equivalent to each other by a deformation of the functor $ F $. We formalize this notion as follows:

\begin{definition}
Let $ \cat C $ be an $ A_∞ $-category and $ \cat C_q $ be a deformation. Let $ F: \cat C → \cat D $ be a quasi-equivalence. Then we call any deformation $ \cat D_q $ of $ \cat D $ a \emph{naive pushforward} of the deformation $ \cat C_q $ along $ F $ if there exists a functor $ F_q: \cat C_q → \cat D_q $ with leading term $ F $.
\end{definition}

\begin{remark}
If forming Hochschild DGLAs were functorial, pushforwards would be easy. Indeed, a quasi-equivalence $ \cat C → \cat D $ would ideally induce a quasi-isomorphism of $ L_∞ $-algebras $ \HC(\cat C) → \HC(\cat D) $ and therefore a bijection $ \MCb(\cat C, B) → \MCb(\cat D, B) $ of Maurer-Cartan elements. The naive pushforward of $ \cat C_q $ would then simply be obtained as the image under this map of the Maurer-Cartan elemement defined by $ \cat C_q $. It is an awkward fact of algebra that however neither the Hochschild DGLA nor Hochschild cohomology is functorial. For instance, even the center $ \HH^0 (A) = Z(A) $ of an algebra $ A $ is not functorial a property. Two quasi-equivalent $ A_∞ $-categories however have quasi-isomorphic Hochschild DGLAs. At least, this is a folklore statement, with actual proof only available by Keller in the dg case \cite{Keller}.
\end{remark}

This definition of naive pushforwards raises many detail questions. For instance, let $ G: \cat D → \cat E $ be yet another quasi-equivalence. Then we are interested in the question whether the double pushforward of $ \cat C_q $ along $ F $ and $ G $ is gauge-equivalent to the single pushforward of $ \cat C_q $ along $ GF $. We shall provide answer to this question in the form of axioms:

\begin{convention}
\label{th:curved-axioms-convention}
We assume the following axioms regarding the Hochschild DGLA: Let $ B $ be a fixed deformation base. Then:
\begin{itemize}
\item[(A1)] Two $ A_∞ $-deformations $ \cat C_q, \cat C_q' $ over $ B $ are gauge equivalent if and only if they are gauge equivalent as Maurer-Cartan elements of $ \HC(\cat C) $.
\item[(A2)] Let $ F: \cat C → \cat D $ be a quasi-equivalence of $ A_∞ $-categories. Then there exists a quasi-isomorphism of $ L_∞ $-algebras $ F_*: \HC(\cat C) → \HC(\cat D) $.
\item[(A3)] We call $ F_* $ the pushforward of $ F $. The pushforward is noncanonical. However its induced map $ (F_*)^{\MC}: \MCb(\HC(\cat C), B) → \MCb(\HC(\cat D), B) $ has the following property: $ \cat D_q $ is a naive pushforward of $ \cat C_q $ along $ F $ if and only if $ μ_{\cat D_q} = F_*^{\MC} (μ_{\cat C_q}) $.
\item[(A4)] Let $ \cat C ⊂ \cat D $ be a subcategory such that the inclusion $ i: \cat C → \cat D $ is a quasi-equivalence. Then $ i_*^{\MC} (μ_{\cat D_q} \restr{\cat C}) = μ_{\cat D_q} $.
\item[(A5)] Push-forward is functorial on Maurer-Cartan elements: $ (GF)_*^{\MC} (μ_{\cat C_q}) = G_*^{\MC} (F_*^{\MC} (μ_{\cat C_q}) $.
\end{itemize}
\end{convention}

\begin{remark}
There is a slight abuse of the notation in \autoref{th:curved-axioms-convention}. Where we have written $ F_*^{\MC} (μ_{\cat C_q}) $, we actually mean the Maurer-Cartan element $ μ_{\cat C_q} - μ_{\cat C} $ instead of $ μ_{\cat C_q} $. In fact, the element $ μ_{\cat C_q} ∈ B \htensor \HC^1 (\cat C) $ is not a Maurer-Cartan element itself. In similar abuse of notation, we may occasionally write $ F_*^{\MC} (\cat C_q) $ instead of $ F_*^{\MC} (μ_{\cat C_q}) $.
\end{remark}

As an application of \autoref{th:curved-axioms-convention}, we show here that quasi-equivalence of deformed $ A_∞ $-categories is an equivalence relation. By quasi-equivalence of two deformed $ A_∞ $-categories, we refer to the relation defined in \autoref{def:curved-functors-quasiequivalence}. A proof without assuming the axioms would either need to deal with complicated explicit constructions, or make use of an ∞-categorical level. In other words, assuming the axioms seems to be a healthy midway for the scope of the paper.

\begin{lemma}
\label{th:curved-axioms-equivrelation}
Quasi-equivalence of deformed $ A_∞ $-categories is an equivalence relation. Even stronger, for every quasi-equivalence $ F: \cat C → \cat D $ there exists a quasi-equivalence $ G: \cat D → \cat C $ such that $ G_* F_* = \id $. Derived equivalence is an equivalence relation as well.
\end{lemma}

\begin{proof}
We need to prove reflexivity, transitivity and symmetry of the quasi-equivalence relation. The first two properties are easy: The identity functor obviously provides a quasi-equivalence from any deformation $ \cat C_q $ to itself. And if $ F_q: \cat C_q → \cat D_q $ and $ G_q: \cat D_q → \cat E_q $ are quasi-equivalences, then the composition $ G_q F_g: \cat C_q → \cat E_q $ is a quasi-equivalence as well. We conclude that only symmetry remains to be proven.

The proof of symmetry consists of five steps: First, we define a set of “good” quasi-equivalences $ F $ for which there exists a quasi-equivalence $ G $ with $ G_* F_* = \id $. The second, third and fourth step establish basic properties of this “good” set. In the fifth step, we show that those properties already make every quasi-equivalence lie in $ \mathcal{F} $.

As step 1, let us recall our context. We are interested in the set of quasi-equivalences $ F: \cat C → \cat D $ such that there exists a quasi-equivalence $ G: \cat D → \cat C $ with $ G_* F_* = \id $. Denote this set by $ \mathcal{F} $:
\begin{equation*}
\mathcal{F} ≔ \{F: \cat C → \cat D \text{ q.e.} \running ∃ G: \cat D → \cat C \text{ q.e.}: G_* F_* = \id\}.
\end{equation*}
Our aim is to show that any quasi-equivalence lies in $ \mathcal{F} $. Steps 2, 3, 4 are devoted to proving several properties of $ \mathcal{F} $.

As step 2, we prove the property
\begin{equation*}
F ∈ \mathcal{F} \text{ with } G \text{ q.e. such that } G_* F_* = \id \quad \Longrightarrow \quad F_* G_* = \id \text{ and } G ∈ \mathcal{F}.
\end{equation*}
Indeed, pick $ F $ as on the left-hand side. Since both $ F $ and $ G $ are quasi-equivalences, both pushforwards $ F_* $ and $ G_* $ are bijections. Together with $ G_* F_* = \id $, we conclude that $ G_* $ and $ F_* $ are simply inverses to each other. In other words, $ F_* G_* = \id $ holds as well. We conclude that $ G ∈ \mathcal{F} $.

As step 3, we prove for composable quasi-equivalences $ F, G $ the property
\begin{equation*}
F, G ∈ \mathcal{F} \quad \Longleftrightarrow \quad GF ∈ \mathcal{F}.
\end{equation*}
One should think of this as a strong version of the two-out-of-three property. To prove it, pick $ F, G ∈ \mathcal{F} $ with $ F'_* F_* = \id $ and $ G'_* G_* = \id $. We get $ (F'G')_* (GF)_* = F'_* F_* = \id $, which renders $ GF ∈ \mathcal{F} $. Conversely assume $ GF ∈ \mathcal{F} $ with $ H_* (GF)_* = \id $. Then $ (HG)_* F_* = \id $, hence $ F ∈ \mathcal{F} $. Step 2 implies $ F_* (HG)_* = \id $. In other words $ (FH)_* G_* = \id $, hence $ G ∈ \mathcal{F} $. We conclude that both $ F $ and $ G $ lie in $ \mathcal{F} $, as desired.

As step 4, we show that the following lie in $ \mathcal{F} $:
\begin{align*}
&\text{quasi-equivalences with a one-sided inverse}, \\
&\text{inclusions of skeletal subcategories}, \\
&\text{inclusion } i: \H\cat C → \cat C \text{ and projection } π: \cat C → \H\cat C.
\end{align*}
To this end, assume $ F $ and $ G $ are quasi-equivalences with $ FG = \Id $. Combining $ \Id ∈ \mathcal{F} $ with step 3, we deduce $ F, G ∈ \mathcal{F} $. In particular, isomorphisms fall under this regime. Inclusions of skeletal subcategories also fall under this regime, since a skeletal subcategory $ S ⊂ \cat C $ produces a quasi-equivalence $ \cat C → S $ which reduces to the identity on $ S $. Now regard a category $ \cat C $ and its minimal model $ \H\cat C $. The minimal model is not unique, but every minimal model comes with quasi-isomorphisms $ i: \H\cat C → \cat C $ and $ π: \cat C → \H\cat C $. Regard the map $ πi: \H\cat C → \H\cat C $. Since both $ i $ and $ π $ are quasi-isomorphisms, $ πi $ is a quasi-isomorphism as well. Moreover, $ \H\cat C $ is already a minimal category, hence $ πi $ is an isomorphism. We conclude that $ πi ∈ \mathcal{F} $. By step 3, we deduce that both $ π $ and $ i $ lie in $ \mathcal{F} $. This finishes step 4.

As final step 5, we prove that any quasi-equivalence lies in $ \mathcal{F} $. To this end, let $ F: \cat C → \cat D $ be any quasi-equivalence. Our strategy is to build a diagram to whose arrows we can apply step 3 and 4 to deduce that $ F $ also lies in $ \mathcal{F} $. In order to write down the diagram, pick minimal models $ \H\cat C $ and $ \H\cat D $, together with inclusion map $ i_{\cat C}: \H\cat C → \cat C $ and projection $ π_{\cat D}: \cat D → \H\cat D $. Define $ F' ≔ π_{\cat D} F i_{\cat C} $. This gives a diagram, commutative by definition,
\begin{center}
\begin{tikzpicture}
\path (0, 0) node (C) {$ \cat C $} (1.5, 0) node (D) {$ \cat D $};
\path (0, -1.5) node (HC) {$ \H\cat C $} (1.5, -1.5) node (HD) {$ \H\cat D $};
\path[draw, ->] (C) -- (D) node[midway, above] {$ F $};
\path[draw, ->] (HC) -- (C) node[midway, left] {$ i_{\cat C} $};
\path[draw, ->] (D) -- (HD) node[midway, right] {$ π_{\cat D} $};
\path[draw, ->] (HC) -- (HD) node[midway, above] {$ F' $};
\end{tikzpicture}
\end{center}
Choose a skeletal subcategory $ S_{\cat C} ⊂ \H\cat C $ and set $ S_{\cat D} ≔ F'(S_{\cat C}) $. Then $ S_{\cat D} ⊂ \H\cat D $ is a skeletal subcategory as well: Any object $ X ∈ \H\cat D $ is isomorphic to some $ F'(Y) $, and $ Y $ in turn is isomorphic to some $ Z ∈ S_{\cat C} $, hence $ X ≅ F'(Z) ∈ S_{\cat D} $. Moreover if $ Y, Z ∈ S_{\cat C} $ and $ F'(Y) ≅ F'(Z) $, then $ Y ≅ Z $. This implies $ Y = Z $ because $ S_{\cat C} $ is a skeleton. In total, we conclude that $ S_{\cat D} ⊂ \H\cat D $ is a skeletal subcategory, and we obtain a restricted quasi-equivalence $ F'': S_{\cat C} → S_{\cat D} $. Putting everything together, we have the commutative diagram
\begin{center}
\begin{tikzpicture}
\path (0, 0) node (C) {$ \cat C $} (1.5, 0) node (D) {$ \cat D $};
\path (0, -1.5) node (HC) {$ \H\cat C $} (1.5, -1.5) node (HD) {$ \H\cat D $};
\path (0, -3) node (S) {$ S $} (1.5, -3) node (S') {$ S' $};
\path[draw, ->] (C) -- (D) node[midway, above] {$ F $};
\path[draw, ->] (HC) -- (C) node[midway, left] {$ i_{\cat C} $};
\path[draw, ->] (D) -- (HD) node[midway, right] {$ π_{\cat D} $};
\path[draw, ->] (HC) -- (HD) node[midway, above] {$ F' $};
\path[draw, {right hook}->] (S) -- (HC) node[midway, left] {$ I $};
\path[draw, {right hook}->] (S') -- (HD) node[midway, right] {$ J $};
\path[draw, ->] (S) -- (S') node[midway, above] {$ F'' $};
\end{tikzpicture}
\end{center}
Here we denoted by $ I $ and $ J $ the inclusion of the full subcategories $ S $ and $ S' $ into $ \H\cat C $ resp.~$ \H\cat D $. The top and bottom square are strictly commutative by definition of $ F' $ and $ F'' $.

We now count everything together: $ F'' $ is an isomorphism and $ J $ is an inclusion of a skeletal subcategory. By step 4, we get $ J, F'' ∈ \mathcal{F} $. By step 3, we get $ JF'' ∈ \mathcal{F} $. By commutativity of the bottom of the diagram we have $ F' I = JF'' $. By step 3, this implies $ F' ∈ \mathcal{F} $. By commutativity of the top of the diagram we have $ π_{\cat D} F i_{\cat C} = F' ∈ \mathcal{F} $. A double application of step 3 renders $ F ∈ \mathcal{F} $. Since $ F $ was arbitrary, this shows that all quasi-equivalences lie in $ \mathcal{F} $. In other words, for every $ F: \cat C → \cat D $ there exists a $ G: \cat D → \cat C $ such that $ G_* F_* = \id $. In particular, this shows quasi-equivalence is an equivalence relation.

Let us now explain why derived equivalence is an equivalence relation as well. Indeed, $ \cat C_q $ and $ \cat D_q $ derived equivalent according to \autoref{def:curved-functors-quasiequivalence} if $ \Tw\cat C_q $ and $ \Tw\cat D_q $ are quasi-equivalent. Since we have just shown that quasi-equivalence is an equivalence relation, we conclude that derived equivalence is an equivalence relation. This finishes the proof.
\end{proof}

\begin{remark}
Pushing deformations from one category to another is not only possible via quasi-equivalences. We can also push forward deformations from one category to a derived equivalent category. Namely, let $ F: \Tw\cat C → \Tw\cat D $ be a quasi-equivalence and let $ \cat C_q $ be a deformation of $ \cat C $. Then the twisted completion $ \Tw\cat C_q $ from \autoref{def:curved-twisted} is canonically a deformation of $ \Tw\cat C $. The pushforward $ F_*^{\MC} $ on Maurer-Cartan elements now transports this deformation $ \Tw\cat C_q $ to a deformation $ F_*^{\MC} (μ_{\Tw\cat C_q}) $ of $ \Tw\cat D $. Restricting this deformation to $ \cat D ⊂ \Tw\cat D $ gives a deformation of $ \cat D $. Via the route of twisted completion, pushforward and restriction, the deformation $ \cat D_q $ corresponds to the deformation $ \cat C_q $. We may equally call $ \cat D_q $ the pushforward of $ \cat C_q $ along $ F $.
\end{remark}

%% file: constructions/intro.tex
\section{Derived categories of $ A_∞ $-deformations}
\label{sec:constructions}
In this section, we define and prove existence of derived categories of $ A_∞ $-deformations. In \autoref{sec:constructions-twisted}, we revisit the notion of twisted completion from \autoref{sec:prelim-ainfty-twisted}. In \autoref{sec:constructions-minmodel}, we define minimal models for $ A_∞ $-deformations, in such a way that the minimal model $ \H\cat C_q $ of $ \cat C_q $ is a deformation of $ \H\cat C $. In \autoref{sec:constructions-derived}, we combine both steps and define the derived category of $ \cat C_q $ to be $ \H\Tw\cat C_q $, the minimal model of the twisted completion. In \autoref{sec:uncurving-theory} we comment on uncurving of $ A_∞ $-deformations and on uncurvable objects.

%% file: constructions/twisted.tex
\subsection{Twisted completion}
\label{sec:constructions-twisted}
In this section, we revisit the twisted completion construction for deformed $ A_∞ $-categories from an abstract perspective. While in \autoref{sec:prelim-ainfty-twisted} we already defined $ \Tw\cat C_q $ in an ad-hoc way, in the present section we have the additional abstract operations of \autoref{sec:lens} available. We explain that $ \Tw\cat C_q $ is the pushforward of the deformation $ \cat C_q $ onto $ \Tw\cat C $ along the derived equivalence $ \cat C → \Tw\cat C $. This section is meant to illustrate that the story of $ \Tw\cat C_q $ is well-rounded and $ \Tw\cat C_q $ is correctly defined by \autoref{def:curved-twisted}.

To start with, let us recall the additive completion and twisted completion from \autoref{def:curved-twisted}. The additive completion $ \Add\cat C_q $ is the deformation of $ \Add\cat C $ whose deformed products are simply induced from those of $ \cat C_q $. The twisted completion $ \Tw\cat C_q $ is the deformation of $ \Tw\cat C_q $ whose deformed products are defined by embracing the deformed products of $ \cat C_q $ with the $ δ $-matrix. Note that the objects of $ \Tw\cat C_q $ are by definition the same as those of $ \Tw\cat C $, namely twisted complexes $ (A_1 [k_1] ⊕ … ⊕ A_n [k_n], δ) $ where $ δ $ satisfies the Maurer-Cartan equation with respect to $ μ_{\cat C} $. The category $ \Tw\cat C_q $ is a curved deformation and its curvature is given by the failure of $ δ $ to satisfy the Maurer-Cartan equation with respect to $ μ_{\cat C_q} $.

With the help of \autoref{def:curved-functors-quasiequivalence}, we can put the definition of $ \Tw\cat C_q $ into a better framing. The easiest way to make this precise is to note that the inclusion $ \cat C → \Tw\cat C $ is trivially a derived equivalence. The twisted completion $ \Tw\cat C_q $ satisfies precisely the deformed analog of this classical statement:

\begin{lemma}
The inclusion $ \cat C_q ⊂ \Tw\cat C_q $ is a derived equivalence of deformed $ A_∞ $-categories.
\end{lemma}

\begin{proof}
Due to our definition of quasi-equivalences of deformed $ A_∞ $-categories, this statement is trivial. Indeed, regard the inclusion functor $ F_q: \cat C_q → \Tw\cat C_q $. This is by definition a functor of deformed $ A_∞ $-categories. Its leading term is the standard inclusion $ F: \cat C → \Tw\cat C $. Ultimately, the functor $ F $ is a derived equivalence and we conclude that $ F_q $ is a derived equivalence according to \autoref{def:curved-functors-quasiequivalence}.
\end{proof}

In this section, let us go beyond the definition of $ \Tw\cat C_q $ and define a few more utilities. It is possible to define a variant $ \Tw'\cat C_q $ of the twisted completion of $ A_∞ $-deformations. It is not a deformation of $ \Tw\cat C $, but rather an object-cloning deformation in the sense of \autoref{def:prelim-ainfty-objectcloning}.

\begin{definition}
\label{th:twisted-deforming-delta}
Let $ \cat C $ be an $ A_∞ $-category and $ \cat C_q $ a deformation. The objects of $ \Tw'\cat C_q $ are be pairs
\begin{equation*}
(X, δ = δ_0 + δ'), \quad X ∈ \Add\cat C, \quad δ_0 ∈ \Hom_{\cat C}^1 (X, X), \quad δ' ∈ \mathfrak{m} \Hom_{\cat C}^1 (X, X).
\end{equation*}
Only the leading part $ δ_0 $ is required to be upper triangular and satisfy the Maurer-Cartan equation with respect to $ μ_{\cat C} $. The hom spaces and higher products of $ \Tw'\cat C_q $ shall be defined by embracing with $ δ $ as in \autoref{def:curved-twisted}.
\end{definition}

\begin{remark}
In an object $ (X, δ_0 + δ') ∈ \Tw'\cat C_q $, the infinitesimal part $ δ' $ can also lie below the diagonal. The category $ \Tw'\cat C_q $ is clearly a deformed $ A_∞ $-category. When dividing out the maximal ideal we do not exactly reach $ \Tw\cat C $ though, but a version of $ \Tw\cat C $ with lots of isomorphic objects, one copy for every choice of infinitesimal terms being added in the $ δ $ matrix. This makes $ \Tw'\cat C_q $ an essentially surjective object-cloning deformation of $ \Tw\cat C $ along the cloning map $ O: \Ob(\Tw'\cat C_q) → \Ob(\Tw\cat C) $ given by $ (X, δ_0 + δ') ↦ (X, δ_0) $.
\end{remark}

\begin{remark}
Let us check that products of this deformed $ A_∞ $-category $ \Tw'\cat C_q $ are still well-defined: Regard a sequence of $ k $ compatible entries of $ δ $. Then at least $ k - d $ of them are infinitesimal, where $ d $ is the dimension of the $ δ $-matrix. We conclude that a term
\begin{equation*}
μ_q (\underbrace{δ, …, δ}_{k_{n+1} \text{ times}}, a_n, \underbrace{δ, …, δ}_{k_n \text{ times}}, …, a_1, \underbrace{δ, …, δ}_{k_1 \text{ times}})
\end{equation*}
lies in the $ K ≔ (k_1 + … + k_{n+1} - (n+1)d) $-th power of the maximal ideal $ \mathfrak{m} $. When the number $ k_1 + … + k_{n+1} $ of total $ δ $ insertions goes to infinity, this exponent $ K $ tends towards infinity as well. This renders all products in $ \Tw'\cat C_q $ well-defined.
\end{remark}

The category $ \Tw'\cat C_q $ allows us to define twisted complexes more liberally than in $ \Tw\cat C_q $. Nevertheless both categories are essentially equal:

\begin{lemma}
\label{th:twisted-TwTw'}
Let $ \cat C $ be an $ A_∞ $-category and $ \cat C_q $ a deformation. Then $ \Tw\cat C_q $ and $ \Tw'\cat C_q $ are quasi-equivalent as deformed $ A_∞ $-categories. Moreover, let $ S ≔ \{(X_i, δ_i)\}_{i = 1, …, n} $ be a collection of objects in $ \Tw\cat C_q $ and $ δ_i' ∈ \mathfrak{m}\End^1 (X_i) $ be infinitesimal terms. Then the category
\begin{equation*}
S' = \{(X_i, δ_i + δ_i')\}_{i = 1, …, n} ⊂ \Tw'\cat C_q
\end{equation*}
is gauge equivalent to $ S $.
\end{lemma}

\begin{proof}
For the first part, regard the inclusion $ \Tw\cat C_q ⊂ \Tw'\cat C_q $. Upon dividing out the maximal ideal $ \mathfrak{m} $, this inclusion reduces to the inclusion of $ \Tw\cat C $ into a version of $ \Tw\cat C $ with lots of copies of objects. We conclude that the inclusion $ \Tw\cat C_q ⊂ \Tw'\cat C_q $ is a quasi-equivalence of deformed $ A_∞ $-categories.

For the second part, build the functor
\begin{equation*}
F_q: S' → S, \quad F^0_{q, (X_i, δ_i + δ_i')} = δ_i', \quad F_q^1 = \Id, \quad F_q^{≥2} = 0.
\end{equation*}
This functor is the identity on objects and indeed satisfies the curved $ A_∞ $-functor relations:
\begin{align*}
\sum μ_{S} (\underbrace{F_q^0, …, F_q^0}_{≥0 \text{ times}}, F_q^1 (a_k), & \underbrace{F_q^0, …, F_q^0}_{≥0 \text{ times}}, …, F^1 (a_1), \underbrace{F_q^0, …, F_q^0}_{≥0 \text{ times}}) \\
& = \sum μ_{\Add\cat C_q} (\underbrace{δ_i, δ_i', …}_{\text{any mix}}, a_k, …, a_1, \underbrace{δ_i, δ_i', …}_{\text{any mix}}) \\
&= \sum μ_{\Add\cat C_q} (\underbrace{δ, …, δ}_{≥0 \text{ times}}, a_k, \underbrace{δ, …, δ}_{≥0 \text{ times}}, …, a_1, \underbrace{δ, …, δ}_{≥0 \text{ times}}) \\
&= F_q^1 (μ_{S'} (a_k, …, a_1).
\end{align*}
This shows that $ F_q $ is a gauge-equivalence $ S' → S $.
\end{proof}

%% file: constructions/minmodel.tex
\subsection{Minimal models}
\label{sec:constructions-minmodel}
In this section we define a notion of minimal model for deformed $ A_∞ $-categories. The definition might differ from the reader's expectation. Indeed, a minimal model according to our definition need not have vanishing differential. We finish the section by explaining why every $ A_∞ $-deformation has a minimal model.

In \autoref{def:kadeishvili-existence-def}, we define the notion of minimal model for $ A_∞ $-deformations. The aim is to have a definition which is compatible with the classical notion and make minimal models exist for any deformation.

\begin{definition}
\label{def:kadeishvili-existence-def}
Let $ \cat C $ and $ \cat D $ be $ A_∞ $-categories and $ \cat C_q $ and $ \cat D_q $ deformations. Then $ \cat D_q $ is a \emph{minimal model} for $ \cat C_q $ if $ \cat D $ is a minimal category and there is a functor of deformed $ A_∞ $-categories
\begin{equation*}
F_q: \cat C_q → \cat D_q
\end{equation*}
whose leading term is a quasi-isomorphism $ F: \cat C → \cat D $. We denote any minimal model of $ \cat C_q $ by $ \H\cat C_q $.
\end{definition}

Let us discuss \autoref{def:kadeishvili-existence-def}. Classical minimal models of $ A_∞ $-categories have vanishing curvature $ μ^0 $ and differential $ μ^1 $ by definition. This is not the case anymore for minimal models of deformations. Instead, we require that $ \cat D $ itself is minimal, while $ \cat D_q $ is allowed to have both curvature and nonvanishing differential. Of course, the curvature and differential of $ \cat D_q $ are both infinitesimal, since $ \cat D_q $ is a deformation of the minimal $ A_∞ $-category $ \cat D $:
\begin{align*}
μ^0_{\cat D_q, X} &∈ \mathfrak{m} \Hom_{\cat D} (X, X), \quad ∀X ∈ \cat D, \\
μ^1_{\cat D_q} (x) &∈ \mathfrak{m} \Hom_{\cat D} (X, Y), \quad ∀X, Y ∈ \cat D, ~ x ∈ \Hom_{\cat D} (X, Y).
\end{align*}
At this point, writing $ \H\cat C_q $ for a minimal model constitutes abuse of notation, since we have not explained their existence. This is done in \autoref{th:kadeishvili-existence-th}.

\begin{remark}
An alternative definition is obtained by requiring that the map $ \cat C → \cat D $ induced by $ F $ be only a quasi-equivalence. The difference in the two notions is merely cosmetic. While in the version of \autoref{def:kadeishvili-existence-def} the objects of $ \cat C $ and $ \cat D $ are required to match, a definition requiring quasi-equivalence allows for additional bloat: One may add any amount of quasi-isomorphic objects to both $ \cat C $ and $ \cat D $.
\end{remark}

There are several equivalent ways of characterizing minimal models. Recall the notion of quasi-isomorphism from \autoref{def:curved-functors-quasiequivalence}.

\begin{lemma}
\label{th:kadeishvili-existence-equiv}
Let $ \cat C, \cat D $ be $ A_∞ $-categories and $ \cat C_q, \cat D_q $ deformations. The following statements are equivalent:
\begin{enumerate}
\item $ \cat D_q $ is a minimal model for $ \cat C_q $.
\item $ \cat C_q $ and $ \cat D_q $ are quasi-isomorphic, and $ \cat D $ is minimal.
\item $ \cat C_q $ and $ \cat D_q $ are quasi-isomorphic, and $ \cat D $ is a minimal model for $ \cat C $.
\item \label{it:kadeishvili-existence-std} There is a quasi-isomorphism $ F: \cat C → \cat D $ such that $ F_*^{\MC} (μ_{\cat C_q}) = μ_{\cat D_q} $ and $ \cat D $ is minimal.
\end{enumerate}
Here $ F_*^{\MC}: \MCb(\HC(\cat C), B) → \MCb(\HC(\cat D), B) $ denotes the push-forward map of Maurer-Cartan elements along $ F $.
\end{lemma}

\begin{proof}
This is a simple consequence of the axioms stated in \autoref{th:curved-axioms-convention}.
\end{proof}

A consequence of \autoref{th:kadeishvili-existence-equiv} is that minimal models always exist:

\begin{corollary}
\label{th:kadeishvili-existence-th}
Let $ \cat C $ be an $ A_∞ $-category. Then any deformation $ \cat C_q $ has a minimal model.
\end{corollary}

\begin{proof}
Pick any (classical) minimal model $ \cat D $ for $ \cat C $. Then $ \cat D $ is a minimal category and there is an $ A_∞ $-quasi-isomorphism $ F: \cat C → \cat D $. Now pick the push-forward
\begin{equation*}
μ_{\cat D_q} ≔ F_* (μ_{\cat C_q}) ∈ \MCb(\HC(\cat D), B).
\end{equation*}
This Maurer-Cartan element defines a deformation $ \cat D_q $ of $ \cat D $. This satisfies statement \ref{it:kadeishvili-existence-std} of \autoref{th:kadeishvili-existence-equiv}.
\end{proof}

%% file: constructions/derived.tex
\subsection{Derived categories}
\label{sec:constructions-derived}
In this section, we define derived categories of $ A_∞ $-deformations. The starting point is an $ A_∞ $-category $ \cat C $ and a deformation $ \cat C_q $. In \autoref{sec:prelim-ainfty-twisted} we have already defined the twisted completion $ \Tw\cat C_q $. In \autoref{sec:constructions-minmodel} we have introduced the notion of minimal models for $ A_∞ $-deformations. By applying it to $ \Tw\cat C_q $, we obtain the derived category $ \D\cat C_q ≔ \H\Tw\cat C_q $.

\begin{definition}
Let $ \cat C $ be an $ A_∞ $-category and $ \cat C_q $ a deformation. Then its \emph{derived category} is the $ A_∞ $-deformation $ \D\cat C_q ≔ \H\Tw\cat C_q $ of $ \H\Tw\cat C $.
\end{definition}

Derived categories are not unique. For instance, any deformation of $ \H\Tw\cat C $ which is gauge equivalent to a given model $ \H\Tw\cat C_q $ is a derived category for $ \cat C $ as well.

\begin{lemma}
Let $ \cat C $ be an $ A_∞ $-category and $ \cat C_q $ a deformation. Then the derived category $ \H\Tw\cat C_q $ is defined uniquely up to quasi-isomorphism.
\end{lemma}

\begin{proof}
According to our definition, the twisted completion $ \Tw\cat C_q $ is defined uniquely. However, the minimal model $ \H\Tw\cat C_q $ is not unique. Instead, minimal models are only defined up to quasi-isomorphism. This finishes the proof.
\end{proof}

\begin{lemma}
Let $ \cat C, \cat D $ be $ A_∞ $-categories and $ \cat C_q, \cat D_q $ deformations. If $ \cat C_q $ and $ \cat D_q $ are derived equivalent, then $ \H\Tw\cat C_q $ and $ \H\Tw\cat D_q $ are quasi-equivalent.
\end{lemma}

\begin{proof}
By \autoref{def:curved-functors-quasiequivalence}, the deformed $ A_∞ $-categories $ \Tw\cat C_q $ and $ \Tw\cat D_q $ are quasi-equivalent. By definition $ \H\Tw\cat C_q $ and $ \Tw\cat C_q $ are quasi-isomorphic. Similarly $ \H\Tw\cat D_q $ and $ \Tw\cat D_q $ are quasi-isomorphic. By \autoref{th:curved-axioms-equivrelation}, quasi-equivalence is an equivalence relation and hence $ \H\Tw\cat C_q $ and $ \H\Tw\cat D_q $ are quasi-equivalent.
\end{proof}

The proof above is inconvenient in that it seems to give no clue about an actual quasi-equivalence $ \H\Tw\cat C_q → \H\Tw\cat D_q $. However we can construct such a functor by properly inverting and chaining the given information. Indeed, the minimal models by definition come with quasi-isomorphisms $ \Tw\cat C_q → \H\Tw\cat C_q $ and $ \Tw\cat D_q → \H\Tw\cat D_q $. As we see in the proof of \autoref{th:curved-axioms-equivrelation}, the quasi-isomorphism $ \Tw\cat C_q → \H\Tw\cat C_q $ has a kind of quasi-inverse. By chaining up all functors, we now obtain an explicit quasi-equivalence
\begin{equation*}
\H\Tw\cat C_q → \Tw\cat C_q → \Tw\cat D_q → \H\Tw\cat D_q.
\end{equation*}

%% file: constructions/uncurving.tex
\subsection{The theory of uncurving}
\label{sec:uncurving-theory}
In this section, we recollect uncurving theory for $ A_∞ $-deformations. The starting point is an arbitrary deformed $ A_∞ $-category with curvature. We are interested in the question which objects become curvature-free once we apply a gauge functor to the category. In the present section, we explain this question and define terminology. We prove elementary properties. Uncurving has been studied in the literature, for instance under the name of the “curvature problem” in \cite{Lowen-vdB}.

It is our aim to explain in how far a gauge equivalence can change the curvature of a deformation. As a starting point, let $ \cat C $ be an $ A_∞ $ category. Let $ \cat C_q $ and $ \cat C_q' $ be two deformations of $ \cat C $ connected by a gauge equivalence $ F: \cat C_q → \cat C_q' $. The curvature of $ \cat C_q $ and $ \cat C_q' $ are related by the zeroth (curved) $ A_∞ $-functor relation:
\begin{equation*}
F^1 (μ^0_{\cat C_q}) = μ^0_{\cat C_q'} + μ^1 (F^0) + μ^2 (F^0, F^0) + …~.
\end{equation*}
The map $ F^1 $ has leading term the identity and is then a linear isomorphism by \autoref{th:prelim-htensor-isocriterion}. This means that the curvature $ μ^0_{\cat C_q} $ depends only on $ F^0 $ and $ F^1 $, and not on the higher components $ F^{≥2} $. If we approximate $ F^1 $ as the identity, we conclude that uncurving essentially depends only on the choice of $ F^0 $. This gives rise to the following definition:

\begin{definition}
Let $ \cat C $ be an $ A_∞ $-category and $ \cat C_q $ a deformation. Let $ X ∈ \cat C $. Then $ X $ is \emph{uncurvable} if there exists an $ S ∈ \mathfrak{m} \End^1(X) $ such that
\begin{equation*}
μ^0_X + μ^1_{\cat C_q} (S) + μ^2_{\cat C_q} (S, S) + … = 0.
\end{equation*}
\end{definition}

We will now prove several basic properties regarding uncurvable objects.

\begin{lemma}
\label{th:uncurving-functor-image}
Let $ F: \cat C_q → \cat D_q $ be a functor of deformed $ A_∞ $-categories. If $ X ∈ \cat C_q $ is uncurvable, then so is $ F(X) $.
\end{lemma}

\begin{proof}
Let $ S ∈ \mathfrak{m}\End^1 (X) $ be the uncurving morphism, that is
\begin{equation*}
μ^0_X + μ^1 (S) + μ^2 (S, S) + … = 0.
\end{equation*}
Now set
\begin{equation*}
T := F^0_X + F^1 (S) + F^2 (S, S) + … ∈ \mathfrak{m}\End^1 (F(X)).
\end{equation*}
We claim that $ T $ is an uncurving morphism for $ F(X) $. Indeed,
\begin{align*}
μ^0_{F(X)} + μ^1 (T) + μ^2 (T, T) + … & = \big(μ^0_{F(X)} + μ^1 (F^0_X) + μ^2 (F^0_X, F^0_X) + …\big) \\
& + \big(μ^1 (F^1 (S)) + μ^2 (F^1 (S), F^0_X) + μ^2 (F^0_X, F^1 (S)) + …\big) \\
& + \big(μ^1 (F^2 (S, S)) + μ^2 (F^1 (S), F^1 (S)) + μ^2 (F^2 (S, S), F^0_X) + … \big) + …
\end{align*}
We apply the curved $ A_∞ $ rule to these terms and continue
\begin{align*}
&= F^1 \big(μ^0_X + μ^1 (S) + μ^2 (S, S) + …\big) + F^2 \big(μ^0_X + μ^1 (S) + μ^2 (S, S) + …, S\big) \\
& \quad + F^2 \big(S, μ^0_X + μ^1 (S) + μ^2 (S, S) + …\big) + … \\
&= F^1 (0) + F^2 (0, S) + F^2 (S, 0) + F^3 (0, S, S) + … = 0.
\end{align*}
This shows that $ T $ is an uncurving morphism for $ F(X) $.
\end{proof}

\begin{lemma}
\label{th:uncurving-uncurvable-isomorphic}
Let $ \cat C $ be an $ A_∞ $-category and $ \cat C_q $ a deformation. Assume $ \cat C $ is minimal, $ X ≅ Y $ in $ \cat C $ and $ X $ is uncurvable. Then so is $ Y $.
\end{lemma}

\begin{proof}
Regard the embedding
\begin{equation*}
\{Y\}_q → \{X, Y\}_q,
\end{equation*}
where both sides are defined as full subcategories of $ \cat C_q $. The embedding is an equivalence, since it reduces to the inclusion $ \{Y\} → \{X, Y\} $ which is essentially surjective in cohomology. By \autoref{th:curved-axioms-equivrelation}, there is a quasi-equivalence in opposite direction
\begin{equation*}
\{X, Y\}_q → \{Y\}_q.
\end{equation*}
In particular, it maps $ X $ to $ Y $. Since $ X $ is uncurvable, an application of \autoref{th:uncurving-functor-image} shows that $ Y $ is uncurvable.
\end{proof}

The next lemma concerns uncurving of objects in minimal models. As a preparation let us recall that a minimal model $ \H\cat C_q $ comes with a choice of quasi-isomorphism $ F_q: \cat C_q → \H\cat C_q $. Correspondingly, the objects of $ \cat C_q $ and $ \H\cat C_q $ are identified via $ F_q $.

\begin{lemma}
\label{th:uncurving-uncurvable-minmodel}
Let $ \cat C $ be an $ A_∞ $-category and $ \cat C_q $ a deformation. Then an object $ X $ is uncurvable in $ \cat C_q $ if and only if it is uncurvable in $ \H\cat C_q $.
\end{lemma}

\begin{proof}
Let $ F_q: \cat C_q → \H\cat C_q $ be a quasi-isomorphism and let $ X ∈ \cat C_q $. We need to show that $ X $ is uncurvable if and only if $ F_q (X) $ is uncurvable. By \autoref{th:uncurving-functor-image}, $ F_q (X) $ is clearly uncurvable if $ X $ is uncurvable. For the other direction, regard the restriction $ F_q \restr{\{X\}}: \{X\}_q → \{F_q(X)\}_q $. It is a quasi-isomorphism and by \autoref{th:curved-axioms-equivrelation} there exists a quasi-isomorphism $ \{F_q (X)\}_q → \{X\}_q $ in opposite direction. Now if $ F_q (X) $ is uncurvable, then by \autoref{th:uncurving-functor-image} also $ X $ is uncurvable. This finishes the proof.
\end{proof}

\begin{corollary}
\label{th:uncurving-quasi-isomorphic}
Let $ \cat C $ be an $ A_∞ $-category and $ \cat C_q $ a deformation. If $ X $ and $ Y $ are quasi-isomorphic in $ \cat C $, then $ X $ is uncurvable if and only if $ Y $ is uncurvable.
\end{corollary}

\begin{proof}
Pick a minimal model $ \H\cat C_q $. The objects $ X $ and $ Y $ are isomorphic in $ \H\cat C $. Combining \autoref{th:uncurving-uncurvable-isomorphic}
 and \ref{th:uncurving-uncurvable-minmodel}, we conclude
\begin{equation*}
X ∈ \cat C_q \text{ uncurvable } ⇔ X ∈ \H\cat C_q \text{ uncurvable } ⇔ Y ∈ \H\cat C_q \text{ uncurvable } ⇔ Y ∈ \cat C_q \text{ uncurvable}.
\end{equation*}
This chain of equivalences proves the claim.
\end{proof}

\autoref{th:uncurving-quasi-isomorphic} might be slightly surprising. It is entirely irrelevant for uncurvability how $ X $ and $ Y $ get deformed themselves, the only relevant measure is whether they are quasi-equivalent in $ \cat C $. This supports \autoref{def:prelim-defo-qi} where we defined two objects $ X, Y ∈ \cat C_q $ to be quasi-isomorphic already if they are quasi-isomorphic in $ \cat C $.

%% file: kadeishvili/intro.tex
\section{A deformed Kadeishvili theorem}
\label{sec:kadeishvili}
The aim of this section is to prove a deformed Kadeishvili theorem. The classical Kadeishvili theorem states that every $ A_∞ $-category has a minimal model, and the minimal model can be computed by a construction with trees. The starting point for the present section is an arbitrary deformed $ A_∞ $-category. In particular, it may contain curvature and its differential need not square to zero. As we have already explained in \autoref{sec:constructions-minmodel}, every deformed $ A_∞ $-category has a minimal model.

In the present section we construct a minimal model $ \H\cat C_q $ explicitly for any deformed $ A_∞ $-category. Our approach is to construct the minimal model via trees. The bottleneck in comparison with the classical case is the curvature and the failure of the differential to square to zero. We are therefore forced to analyze the shape of the differential in detail and build methods that are robust enough to work with less premises than the classical Kadeishvili theorem.

\begin{center}
\begin{tikzpicture}
\begin{scope}
\path (-4, 0) node {\textbf{Classical Kadeishvili:}};
\path (0, 0) node (A) {$ A_∞ $-category $ \cat C $} (6, 0) node[align=center] (B) {Minimal model $ \H\cat C $};
\path[draw, decorate, decoration={snake, amplitude=0.2em, post length=0.5em}, ->] ($ (A.east)!0.2!(B.west) $) to ($ (A.east)!0.8!(B.west) $);
\end{scope}
\begin{scope}[shift={(down:1)}]
\path (-4, 0) node {\textbf{Deformed Kadeishvili:}};
\path (0, 0) node (A) {$ A_∞ $-deformation $ \cat C_q $} (6, 0) node[align=center] (B) {Minimal model $ \H\cat C_q $};
\path[draw, decorate, decoration={snake, amplitude=0.2em, post length=0.5em}, ->] ($ (A.east)!0.2!(B.west) $) to ($ (A.east)!0.8!(B.west) $);
\end{scope}
\end{tikzpicture}
\end{center}

In \autoref{sec:kadeishvili-splitting}, we recall homological splittings, a classical basic notion in the construction of minimal models. In \autoref{sec:kadeishvili-classical}, we review the classical Kadeishvili theorem and the description of the higher products by trees. In \autoref{sec:kadeishvili-deformed}, we analyze differentials of $ A_∞ $-deformations in detail. In \autoref{sec:kadeishvili-optimizing}, we provide a procedure to optimize the curvature of $ A_∞ $-deformations. In \autoref{sec:kadeishvili-auxiliary}, we provide an auxiliary minimal model construction for deformed $ A_∞ $-categories which already have optimal curvature. In \autoref{sec:kadeishvili-general}, we compile all the constructions into a single theorem. Our deformed Kadeishvili theorem \autoref{th:kadeishvili-general-th} states that a minimal model for every deformed $ A_∞ $-category can be described explicitly, by means of applying the optimization procedure followed by a construction with trees. In \autoref{sec:kadeishvili-caseD} we study a special case of the deformed Kadeishvili theorem and relate it back to the classical Kadeishvili theorem.

%% file: kadeishvili/splitting.tex
\subsection{Homological splittings}
\label{sec:kadeishvili-splitting}
In this section we recall the notion of homological splittings. The idea is to split a cochain complex into three direct summands in terms of which the differential becomes easy to describe. This is a classical notion, often just referred to as a “split” \cite{Bocklandt}. Instead of defining the notion for any cochain complex, we will directly set off in the context of an $ A_∞ $-category.

\begin{definition}
Let $ \cat C $ be an $ A_∞ $-category. Then a \emph{homological splitting} of $ \cat C $ consists of a direct sum decomposition
\begin{equation*}
\Hom_{\cat C} (X, Y) = H(X, Y) ⊕ I(X, Y) ⊕ R(X, Y), \quad ∀X, Y ∈ \cat C
\end{equation*}
for all its hom spaces, such that
\begin{equation*}
I(X, Y) = \Im(μ^1), \quad \Ker(μ^1) = H(X, Y) ⊕ I(X, Y), \quad ∀X, Y ∈ \cat C.
\end{equation*}
We frequently denote a homological splitting of $ \cat C $ simply by the letters $ H ⊕ I ⊕ R $, the dependence on $ X, Y ∈ \cat C $ understood implicitly.
\end{definition}

Given a category $ \cat C $, one obtains a homological splitting by choosing $ H $ as a space of cocycles that represents the cohomology of the hom complexes. One then chooses $ R $ as a complement to $ H $ in $ \Ker(μ^1) $. The notation $ I $ is simply a shorthand for the image of the differential.

\begin{remark}
Almost everywhere in this \autoref{sec:kadeishvili}, we leave out the letters $ X $ and $ Y $. All definitions, equations and expressions referring to elements of $ H $, $ I $ and $ R $ are to be interpreted as being quantified over $ X, Y ∈ \cat C $. The quantification is understood implicitly. We will even write for instance $ \Hom_{\cat C} = H ⊕ I ⊕ R $, meaning $ \Hom_{\cat C} (X, Y)  = H(X, Y) ⊕ I(X, Y) ⊕ R(X, Y) $ for every $ X, Y ∈ \cat C $.
\end{remark}

In the remainder of this section, we record a few consequences of the choice of homological splitting. To start with, in terms of the direct sum decomposition $ \Hom_{\cat C} = H ⊕ I ⊕ R $, the differential reads
\begin{equation}
\label{eq:kadeishvili-splitting-shape}
μ^1 = \begin{pmatrix} 0 & 0 & 0 \\ 0 & 0 & μ^1 \\ 0 & 0 & 0 \end{pmatrix}.
\end{equation}
Indeed, the spaces $ H $ and $ I $ are mapped entirely to zero through $ μ^1 $. The only space not being sent to zero is $ R $. This gives the claimed matrix shape \eqref{eq:kadeishvili-splitting-shape}. The action of $ μ^1 $ on the three summands $ H ⊕ I ⊕ R $ is depicted visually in \autoref{fig:kadeishvili-splitting-differential}.

A second observation is that $ μ^1 \restr{R}: R → I $ is a linear isomorphism and provides an identification between $ R $ and $ I $. In fact, the map is injective because the kernel of $ μ^1 $ equals $ H ⊕ I $ which has vanishing intersection with $ R $. Moreover, the map is surjective because $ μ^1 $ already reaches its entire image on $ R $. As a consequence, we can identify $ R $ and $ I $ by means of $ μ^1 $. Note that both differ by a shift of $ 1 $. Upon this identification, the remaining $ μ^1 $ entry in the matrix presentation \eqref{eq:kadeishvili-splitting-shape} becomes the identity $ \Id_R $.

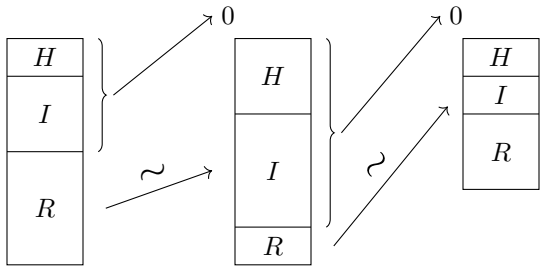
\begin{figure}
\centering
\begin{tikzpicture}
\path[draw] (0, 0) -- ++(down:0.5) coordinate (A) coordinate[midway] (H1) -- ++(down:1) coordinate (B) coordinate[midway] (I1) -- ++(down:1.5) coordinate[midway] (R1) -- ++(right:1) -- ++(up:3) -- cycle;
\path[draw] (3, 0) -- ++(down:1) coordinate (C) coordinate[midway] (H2) -- ++(down:1.5) coordinate (D) coordinate[midway] (I2) -- ++(down:0.5) coordinate[midway] (R2) -- ++(right:1) -- ++(up:3) -- cycle;
\path[draw] (6, 0) -- ++(down:0.5) coordinate (E) coordinate[midway] (H3) -- ++(down:0.5) coordinate (F) coordinate[midway] (I3) -- ++(down:1) coordinate[midway] (R3) -- ++(right:1) -- ++(up:2) -- cycle;
\foreach \i in {A, B, C, D, E, F} \path[draw] (\i) -- ++(right:1);
\path[draw, decorate, decoration={brace, amplitude=3pt}] (1.2, 0) -- (1.2, -1.5) coordinate[midway] (m1);
\path[draw, decorate, decoration={brace, amplitude=3pt}] (4.2, 0) -- (4.2, -2.5) coordinate[midway] (m2);
\path[draw, ->] (1.3, -2.25) -- (2.7, -1.75) node[midway, above, sloped] {\Large $ \sim $};
\path[draw, ->] (4.3, -2.75) -- (5.8, -0.9) node[midway, above, sloped] {\Large $ \sim $};
\path[draw, ->] (m1) ++(right:0.2) -- (2.7, 0.3) node[right] {$ 0 $};
\path[draw, ->] (m2) ++(right:0.2) -- (5.7, 0.3) node[right] {$ 0 $};
\foreach \i in {1, 2, 3} \path (H\i) ++(right:0.5) node {$ H $};
\foreach \i in {1, 2, 3} \path (I\i) ++(right:0.5) node {$ I $};
\foreach \i in {1, 2, 3} \path (R\i) ++(right:0.5) node {$ R $};
\end{tikzpicture}
\caption{The differential $ μ^1 $ in terms of a homological splitting}
\label{fig:kadeishvili-splitting-differential}
\end{figure}

\begin{remark}
In the context of homological splittings, we may use tuple notation to indicate an element of the direct sum. Moreover, we may write elements of $ \Im(μ^1) $ explicitly as $ μ^1 (r') $ where $ r' ∈ R $. In total, we may write an element of $ \Hom_{\cat C}  $ as
\begin{equation*}
x = (h, μ^1 (r'), r), \quad \text{with } h ∈ H, ~ r' ∈ R, ~ r ∈ R.
\end{equation*}
\end{remark}

Let us set up two more pieces of terminology. The first is the codifferential $ h $, which is a zero extension of the inverse of the bijection $ μ^1 \restr{R}: R → I $. The second is the projection to $ H $.

\begin{definition}
Let $ H ⊕ I ⊕ R $ be a homological splitting for $ \cat C $. Then the \emph{codifferential} is the map
\begin{align*}
h: \Hom_{\cat C}  &→ R, \\
(h, μ^1 (r'), r) &↦ r', \quad h ∈ H, ~ r' ∈ R, ~ r ∈ R.
\end{align*}
The \emph{projection to cohomology} is the map
\begin{equation*}
π: \Hom_{\cat C} = H ⊕ I ⊕ R \twoheadrightarrow H.
\end{equation*}
\end{definition}

A small abuse of notation consists in the fact that we typically denote elements of the space $ H $ by the letter $ h $. Typically, there seems to be little chance of confusion.

%% file: kadeishvili/classical.tex
\subsection{The classical Kadeishvili theorem}
\label{sec:kadeishvili-classical}
In this section we recall the classical Kadeishvili theorem. This serves as a preparation for our deformed Kadeishvili theorem and fixes pieces of notation. We follow the construction by means of trees, as given in \cite[Chapter 6, 3.3.2]{Kontsevich-Soibelman}. A good reference is also \cite[Section 3.2]{Bocklandt-book}.

To start with, we recall the standard notion of minimal $ A_∞ $-categories and minimal models:

\begin{definition}
An $ A_∞ $-category is \emph{minimal} if its differential vanishes. Let $ \cat C $ and $ \cat D $ be an $ A_∞ $-categories. Then $ \cat D $ is a \emph{minimal model} for $ \cat C $ if $ \cat C $ and $ \cat D $ are quasi-isomorphic and $ \cat D $ is minimal.
\end{definition}

The intention of the Kadeishvili construction is to construct minimal models explicitly. The starting point for the construction is a homological splitting $ H ⊕ I ⊕ R $. The result of the construction is an $ A_∞ $-structure on $ H = \{H(X, Y)\}_{X, Y ∈ \cat C} $ which can also be interpreted as an $ A_∞ $-structure on $ \H\Hom_{\cat C} = \{\H\Hom_{\cat C} (X, Y)\}_{X, Y ∈ \cat C} $, since $ H $ and $ \H\Hom_{\cat C} $ are isomorphic as graded vector spaces through the composition $ H \embeds \Ker(μ^1) \twoheadrightarrow \H\Hom_{\cat C} $. Specifically, the $ A_∞ $-structure on $ H $ is defined via trees. We fix terminology as follows:

\begin{definition}
\label{def:kadeishvili-classical-treeshape}
A \emph{Kadeishvili tree shape} $ T $ is a rooted planar tree with $ n ≥ 2 $ leaves whose non-leaf nodes all have at least 2 children. A node in $ T $ is \emph{internal} if it is not a leaf and not the root. The number of internal nodes in $ T $ is denoted $ N_T $. We denote by $ \mathcal{T}_n $ the set of all Kadeishvili tree shapes with $ n $ leaves.

A \emph{Kadeishvili π-tree} $ (T, h_1, …, h_n) $ is a Kadeishvili tree shape $ T ∈ \mathcal{T}_n $ with $ n ≥ 2 $ leaves, together with a sequence $ h_1, …, h_n $ of cohomology elements $ h_i ∈ H(X_i, X_{i+1}) $. Decorate the leaves by $ h_1, …, h_n $ in sequence. Decorate every non-root node with the operation $ hμ $ and the root with the operation $ πμ $. Then the \emph{result} $ \Res(T, h_1, …, h_n) ∈ H(X_1, X_{n+1}) $ of the Kadeishvili π-tree is the result obtained by evaluating the tree from leaves to the root, according to the decorations.
\end{definition}

In other words, to evaluate a π-tree one inserts the inputs at the leaves, applies $ hμ $ at every internal node and $ πμ $ at the root. In every evaluation step, the map $ μ $ is some $ k $-ary product of $ \cat C $ and yields an output in some hom space $ \Hom_{\cat C} (X_i, X_j) $. The subsequent application of $ h $ refers to the codifferential of that hom space $ \Hom_{\cat C} (X_i, X_j) $.

\begin{example}
\autoref{fig:kadeishvili-classical-shapes} depicts all tree shapes with $ n = 2 $ and $ n = 3 $ leaves, as well as a sample tree shape for $ n = 4 $. \autoref{fig:kadeishvili-classical-trees} shows how the sample tree shapes of \autoref{fig:kadeishvili-classical-shapes} together with the input sequences $ h_1, h_2 $ or $ h_1, h_2, h_3 $ or $ h_1, …, h_4 $ get decorated. Explicitly, the results of the first three π-trees read $ π μ^2 (h_2, h_1) $, $ π μ^3 (h_3, h_2, h_1) $, $ π μ^2 (hμ^2 (h_3, h_2), h_1) $.
\end{example}

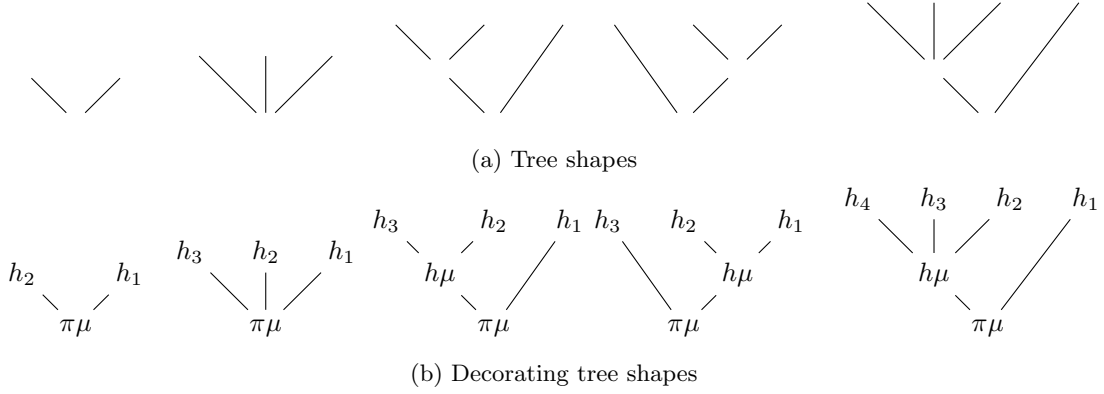
\begin{figure}
\centering
\begin{subfigure}{0.98\linewidth}
\centering
\begin{tikzpicture}
\path node (R1) {}
node[above left of=R1] (A1) {} edge (R1)
node[above right of=R1] (A2) {} edge (R1);
\path (2.5, 0) node (R2) {}
node[above of=R2] (B2) {} edge (R2)
node[left of=B2] {} edge (R2)
node[right of=B2] {} edge (R2);
\path (5.5, 0) node (R3) {}
node[above left of=R3] (C1) {} edge (R3)
node[above left of=C1] (C2) {} edge (C1)
node[above right of=C1] (C3) {} edge (C1)
node[right of=C3] {} edge (R3);
\path (8, 0) node (R4) {}
node[above right of=R4] (D1) {} edge (R4)
node[above right of=D1] (D2) {} edge (D1)
node[above left of=D1] (D3) {} edge (D1)
node[left of=D3] {} edge (R4);
\path (12, 0) node (R5) {}
node[above left of=R5] (E1) {} edge (R5)
node[above of=E1] (E4) {} edge (E1)
node[right of=E4] (E2) {} edge (E1)
node[left of=E4] (E3) {} edge (E1)
node[right of=E2] {} edge (R5);
\end{tikzpicture}
\caption{Tree shapes}
\label{fig:kadeishvili-classical-shapes}
\end{subfigure}
\begin{subfigure}{0.98\linewidth}
\centering
\begin{tikzpicture}
\path node (R1) {$ πμ $}
node[above left of=R1] (A1) {$ h_2 $} edge (R1)
node[above right of=R1] (A2) {$ h_1 $} edge (R1);
\path (2.5, 0) node (R2) {$ πμ $}
node[above of=R2] (B2) {$ h_2 $} edge (R2)
node[left of=B2] {$ h_3 $} edge (R2)
node[right of=B2] {$ h_1 $} edge (R2);
\path (5.5, 0) node (R3) {$ πμ $}
node[above left of=R3] (C1) {$ hμ $} edge (R3)
node[above left of=C1] (C2) {$ h_3 $} edge (C1)
node[above right of=C1] (C3) {$ h_2 $} edge (C1)
node[right of=C3] {$ h_1 $} edge (R3);
\path (8, 0) node (R4) {$ πμ $}
node[above right of=R4] (D1) {$ hμ $} edge (R4)
node[above right of=D1] (D2) {$ h_1 $} edge (D1)
node[above left of=D1] (D3) {$ h_2 $} edge (D1)
node[left of=D3] {$ h_3 $} edge (R4);
\path (12, 0) node (R5) {$ πμ $}
node[above left of=R5] (E1) {$ hμ $} edge (R5)
node[above of=E1] (E4) {$ h_3 $} edge (E1)
node[right of=E4] (E2) {$ h_2 $} edge (E1)
node[left of=E4] (E3) {$ h_4 $} edge (E1)
node[right of=E2] {$ h_1 $} edge (R5);
\end{tikzpicture}
\caption{Decorating tree shapes}
\label{fig:kadeishvili-classical-trees}
\end{subfigure}
\caption{Illustration of Kadeishvili tree shapes and Kadeishvili π-trees}
\end{figure}

\begin{remark}
We have two conventions regarding the order of inputs. Indeed, we sometimes order the elements as $ h_1, …, h_n $ and sometimes as $ h_n, …, h_1 $. The convention is that morphisms $ h_1, …, h_n $ indexed by a set of numbers are always compatible in ascending order: $ h_1 $ is a morphism $ X_1 → X_2 $, while $ h_2 $ is a morphism $ X_2 → X_3 $ etc. In particular, due to our “Polish notation” convention of writing $ A_∞ $-products, the product of the sequence $ h_1, …, h_n $ is written in opposite order as $ μ^n (h_n, …, h_1) $. The way we draw trees, for example in \autoref{fig:kadeishvili-classical-trees}, is also in opposite order. In contrast, wherever we refer to the sequence as a whole and evaluation is not immediate, we write the sequence in natural order. An example of natural order is the expression $ \Res(T, h_1, …, h_n) $.
\end{remark}

The construction of the $ A_∞ $-product $ μ_H $ on $ H = \H\cat C $ can be summarized as follows: Let $ h_1, …, h_n $ be cohomology elements with $ h_i ∈ H(X_i, X_{i+1}) $. Then their higher product is defined as
\begin{equation*}
μ_{\H\cat C} (h_n, …, h_1) = \sum_{T ∈ \mathcal{T}_n} (-1)^{N_T} \Res(T, h_1, …, h_n) ∈ H(X_1, X_{n+1}).
\end{equation*}

\begin{example}
The first higher products of $ \H\cat C $ read as follows:
\begin{align*}
μ^1_{\H\cat C} &= 0, \\
μ^2_{\H\cat C} (h_2, h_1) &= π μ^2 (h_2, h_1), \\
μ^3_{\H\cat C} (h_3, h_2, h_1) &= π μ^3 (h_3, h_2, h_1) - π μ^2 (h μ^2 (h_3, h_2), h_1) - π μ^2 (h_3, h μ^2 (h_2, h_1)).
\end{align*}
\end{example}

The Kadeishvili theorem claims that $ \H\cat C $ together with the product $ μ_{\H\cat C} $ is a minimal model for $ \cat C $. The specific version of the theorem is taken from \cite[Chapter 6, 3.3.2, 3.3.3]{Kontsevich-Soibelman}. This source does not provide any sign rule, but we have verified the correctness of the signs $ (-1)^{N_T} $ in \autoref{th:kadeishvili-auxiliary-Hq}.

\begin{theorem}[Kadeishvili]
Let $ \cat C $ be an $ A_∞ $-category. Then the products $ μ^k_{\H\cat C} $ form a minimal $ A_∞ $-structure on $ \H\cat C $. There is a quasi-isomorphism $ F: \H\cat C → \cat C $ which can be constructed by trees as well. Its 1-ary component $ F^1 $ is the standard inclusion $ \H\cat C = H \embeds \cat C $.
\end{theorem}

%% file: kadeishvili/deformed.tex
\subsection{Deformed differentials}
\label{sec:kadeishvili-deformed}
In this section, we analyze differentials of deformed $ A_∞ $-categories in detail. The starting point is a deformed $ A_∞ $-category $ \cat C_q $ together with a homological splitting of $ \cat C $. The homological splitting of $ \cat C $ is naturally not a homological splitting for $ \cat C_q $. However, one may try find a decomposition of the hom spaces of $ \cat C_q $ with properties that at least resemble those of a homological splitting. The idea is to deform the spaces involved in the homological splitting in order to account for the deformed differential. In the present section, we construct these deformed decompositions and state all properties.

\begin{center}
\begin{tikzpicture}
\path (0, 0) node[align=center] (A) {Homological splitting \\ $ \Hom_{\cat C} = H ⊕ I ⊕ R $} (9, 0) node[align=center] (B) {Deformed decomposition \\ $ \Hom_{\cat C_q} = H_q ⊕ μ^1_q (B \htensor R) ⊕ (B \htensor R) $};
\path[draw, ->] ($ (A.east)!0.15!(B.west) $) -- ($ (A.east)!0.85!(B.west) $) node[midway, above] {upon deformation};
\end{tikzpicture}
\end{center}

The direct approach of \autoref{sec:kadeishvili-classical} fails in the deformed context. In fact, one of the reasons the construction with trees works well in the classical case is that we have $ μ^1 (H) = 0 $. However, it need not be the case that $ μ^1_q (H) = 0 $. This makes the direct description of the minimal model by Kadeishvili trees fail in the deformed context.

The present section provides a workaround. A glance at \cite[Chapter 6, 3.3]{Kontsevich-Soibelman} shows that it suffices to require $ μ^1_q (H) ⊂ H $ instead of $ μ^1_q (H) = 0 $. Even better, we may try in the deformed context to find an infinitesimal deformation $ H_q $ of $ H $ with $ μ^1_q (H_q) ⊂ H_q $. Exploiting this observation is the strategy of our deformed Kadeishvili theorem. The present section is devoted to finding this deformation $ H_q $.

A point of attention is the requirement of a minimal model of $ \cat C_q $ to be a deformation of $ \H\cat C $. This entails that the hom spaces be identified as $ B \htensor \Hom_{\H\cat C}  $ and that the leading term of the products is the product $ μ_{\cat C} $. The present section has been purpose-built to keep track of the identification of $ H_q $ and $ B \htensor H $.

In order to find $ H_q $, we have to analyze the precise shape of the differential $ μ^1_q $. In \autoref{sec:kadeishvili-splitting}, we have seen that a differential on an ordinary $ A_∞ $-category can be written in matrix form. Most matrix entries vanish because of the $ A_∞ $-relations. For deformed $ A_∞ $-categories, we can still write down $ μ^1_q $ in matrix form with respect to $ H ⊕ I ⊕ R $, although no entries vanish by default. A first step is to change $ I $ to $ μ^1_q (B \htensor R) $, which already renders two matrix entries zero:

\begin{lemma}
\label{th:kadeishvili-deformed-badsplitting}
Let $ \cat C $ be an $ A_∞ $-category and $ \cat C_q $ a deformation. Let $ H ⊕ I ⊕ R $ be a homological splitting of $ \cat C $. Then the differential $ μ^1_q $ restricted to $ B \htensor R $ is injective:
\begin{equation}
\label{eq:kadeishvili-deformed-identification}
μ^1_q: B \htensor R ~\isoto~ μ^1_q (B \htensor R),
\end{equation}
and we have a direct sum decomposition of $ B $-modules
\begin{equation}
\label{eq:kadeishvili-deformed-badsplitting}
B \htensor \Hom_{\cat C} = (B \htensor H) ⊕ μ^1_q (B \htensor R) ⊕ (B \htensor R).
\end{equation}
With respect to this decomposition, $ μ^1_q $ takes the shape
\begin{equation}
\label{eq:kadeishvili-deformed-matrixshape}
μ^1_q = \pmat{D & * & 0 \\ μ^1_q E & * & * \\ F & * & 0}
\end{equation}
for some operators
\begin{align*}
D: B \htensor H &→ B \htensor H, \\
E: B \htensor H &→ B \htensor R, \\
F: B \htensor H &→ B \htensor R.
\end{align*}
\end{lemma}

\begin{proof}
First of all, regard the map
\begin{equation*}
μ^1_q: B \htensor R → B \htensor \Hom_{\cat C} .
\end{equation*}
It is $ B $-linear and has leading term the injective map $ μ^1_{\cat C} \restr{R} $. By \autoref{th:prelim-htensor-isocriterion}, it is an embedding. This establishes the first claim.

Second, let us prove the direct sum decomposition. Intuitively, changing the summand $ μ^1 (B \htensor R) $ to $ μ^1_q (B \htensor R) $ constitutes only an infinitesimal change and should leave the decomposition intact. Formally, define the map
\begin{align*}
ψ: \Hom_{\cat C_q}  &→ B \htensor H + μ^1_q (B \htensor R) + B \htensor R \embeds \Hom_{\cat C_q} , \\
(h, μ^1 (r), r') &↦ h + μ^1_q (r) + r', \quad \text{for} \quad h ∈ B \htensor H, \quad r ∈ B \htensor R, \quad r' ∈ B \htensor R.
\end{align*}
The map $ ψ $ has leading term the identity. By \autoref{th:prelim-htensor-autocontinuous}, it is an isomorphism onto $ \Hom_{\cat C_q}  $. In particular, this already establishes that \eqref{eq:kadeishvili-deformed-badsplitting} is a sum decomosition, not necessarily direct. To prove the sum decomposition direct, let $ h ∈ B \htensor H $, $ μ^1_q (r) ∈ μ^1_q (B \htensor R) $ and $ r' ∈ B \htensor R $ with $ h + μ^1_q (r) + r' = 0 $ in $ \Hom_{\cat C_q}  $. This implies $ ψ(h, μ^1_q (r), r') = 0 $ and finally $ h = r = r' = 0 $ since $ ψ $ is an isomorphism and $ μ^1_q \restr{B \htensor R} $ is injective. We conclude that \eqref{eq:kadeishvili-deformed-badsplitting} is a direct sum decomposition.

To obtain the claimed matrix presentation of $ μ^1_q $, simply define $ D $ and $ F $ as $ μ^1_{B \htensor H} $ followed by the projections to $ B \htensor H $ and $ B \htensor R $, respectively. Define $ E $ as $ μ^1_{B \htensor H} $ followed by projection to $ μ^1_q (B \htensor R) $ and the inverse of $ μ^1_q: B \htensor R \isoto μ^1_q (B \htensor R) $. The vanishing of the two indicated matrix entries is immediate, since $ μ^1_q $ sends $ B \htensor R $ to $ μ^1_q (B \htensor R) $ by definition. This settles all claims.
\end{proof}

We are now ready to define $ H_q $.

\begin{lemma}
\label{th:kadeishvili-deformed-splitting}
Let $ \cat C $ be an $ A_∞ $-category and $ \cat C_q $ a deformation. Assume $ H ⊕ I ⊕ R $ is a homological splitting for $ \cat C $. Let $ D, E, F $ denote the operators from \autoref{th:kadeishvili-deformed-badsplitting}. Put
\begin{equation*}
H_q ≔ \{h - Eh \running h ∈ B \htensor H\}.
\end{equation*}
Then we have a direct sum decomposition
\begin{equation}
\label{eq:kadeishvili-deformed-splitting}
B \htensor \Hom_{\cat C} = H_q ⊕ μ^1_q (B \htensor R) ⊕ (B \htensor R).
\end{equation}
It holds that $ μ^1_q (H_q) ⊂ H_q ⊕ B \htensor R $. With respect to this decomposition, $ μ^1_q $ has the shape
\begin{equation*}
μ^1_q = \pmat{* & * & 0 \\ * & * & * \\ 0 & * & 0}.
\end{equation*}
\end{lemma}

\begin{proof}
The decomposition is achieved easily as in the proof of \autoref{th:kadeishvili-deformed-badsplitting}. Namely, $ H_q $ and $ B \htensor H $ only differ by $ R $-terms. To show $ μ^1_q (H_q) ⊂ H_q ⊕ B \htensor R $, we calculate
\begin{equation*}
μ^1_q (h - Eh) = Dh + μ^1_q E (h) + Fh - μ^1_q (Eh) = Dh + Fh ∈ B \htensor H ⊕ B \htensor R = H_q ⊕ B \htensor R.
\end{equation*}
This finishes the proof.
\end{proof}

The decomposition \eqref{eq:kadeishvili-deformed-splitting} plays a crucial role throughout this paper. It is not a homological splitting of $ \cat C_q $ in any sense, since for example $ μ^1_q $ need not vanish on $ H_q $. The decomposition is however an important prerequisite for our deformed Kadeishvili theorem. In particular, whenever computing minimal models of deformed $ A_∞ $-categories, this decomposition needs to be calculated first.

There is a natural identification between $ B \htensor H $ and $ H_q $. The identification associated an element $ h ∈ B \htensor H $ with $ h - Eh ∈ H_q $. Since $ Eh ∈ B \htensor R $, we can recover $ h $ from $ h-Eh $ by stripping off the $ R $ component. This identification plays an important role in this paper. Another important role is played by the map $ μ^1_q: B \htensor R → μ^1_q (B \htensor R) $. We call the inverse of this map the deformed codifferential. Let us fix all important notions in the following definition.

\begin{definition}
\label{def:kadeishvili-deformed-counterpart}
Let $ \cat C $ be an $ A_∞ $-category and $ H ⊕ I ⊕ R $ a homological splitting. Let $ \cat C_q $ be a deformation of $ \cat C $. The \emph{deformed decomposition} of $ \cat C_q $ is the collection of direct sum decompositions \eqref{eq:kadeishvili-deformed-splitting} of all hom spaces in $ \cat C_q $. The \emph{deformed counterpart} of an element $ h ∈ B \htensor H $ is the element $ h - Eh ∈ H_q $. The correspondence between $ H_q $ and $ B \htensor H $ is denoted
\begin{align*}
φ: H_q &\verylongisoto B \htensor H, \\
h - Eh &\verylongmapsto h.
\end{align*}
The \emph{deformed codifferential} of $ \cat C_q $ is the $ R $-linear map
\begin{equation*}
h_q = (μ^1_q \restr{B \htensor R})^{-1}: μ^1_q (B \htensor R) \verylongto B \htensor R.
\end{equation*}
The \emph{deformed projection} of $ \cat C_q $ is the $ R $-linear map
\begin{equation*}
π_q: H_q ⊕ μ^1_q (B \htensor R) ⊕ (B \htensor R) → H_q.
\end{equation*}
\end{definition}

%% file: kadeishvili/optimizing.tex
\subsection{Optimizing curvature}
\label{sec:kadeishvili-optimizing}
In this section, we show how to optimize curvature of an $ A_∞ $-category. In general, it is not possible to remove curvature from an $ A_∞ $-deformation entirely. For the purposes of our Kadeishvili theorem, it is however important to tame the curvature as much as possible. In this section, we show that the curvature of any deformation can be reduced sufficiently for our purpose of constructing a deformed Kadeishvili theorem.

\begin{center}
\begin{tikzpicture}
\path (0, 0) node (A) {Deformation $ \cat C_q $} (8, 0) node[align=center] (B) {Deformation $ \cat C_q^{\bigmax} $ with optimal curvature};
\path[draw, decorate, decoration={snake, amplitude=0.2em, post length=0.5em}, ->] ($ (A.east)!0.2!(B.west) $) to ($ (A.east)!0.8!(B.west) $);
\end{tikzpicture}
\end{center}

Let us start by fixing our terminology:

\begin{definition}
Let $ \cat C $ be an $ A_∞ $-category and $ \cat C_q $ a deformation. Let $ H ⊕ I ⊕ R $ be a homological splitting for $ \cat C $ and $ H_q ⊕ μ^1_q (B \htensor R) ⊕ (B \htensor R) $ be the associated deformed decomposition of $ \cat C_q $. Then $ \cat C_q $ has \emph{optimal curvature} if $ μ^0_q ∈ H_q $.
\end{definition}

In the remainder of the section, we show how to gauge an arbitrary deformed $ A_∞ $-category such that its curvature becomes optimal. We also explain why optimal curvature is the best we can expect. The idea to optimize the curvature is to apply successive gauges. All gauges will be gauge functors $ F $ of the form $ F^1 = \Id $ and $ F^0 = r $ and have no higher components. We may also call such functors “uncurving gauges” because they are strong at reducing curvature. The following definition settles our terminology.

\begin{definition}
\label{def:kadeishvili-optimizing-uncurving}
Let $ \cat C $ be an $ A_∞ $-category and $ \cat C_q $ a deformation. Let $ r = \{r_X\}_{X ∈ \cat C} $ be an element consisting of $ r_X ∈ \mathfrak{m} \End^1_{\cat C} (X) $ for every $ X ∈ \cat C $. Then the \emph{uncurving} of $ \cat C_q $ \emph{by} $ r $ is the category $ \cat C_q' $ obtained from adding $ r_X $ as twisted differential to every $ X ∈ \cat C_q $:
\begin{equation*}
\cat C_q' ≔ \{(X, r_X) \running X ∈ \cat C\} ⊂ \Tw'\cat C_q.
\end{equation*}
The notation $ \Tw'\cat C_q $ is taken from \autoref{th:twisted-deforming-delta}.
\end{definition}

\begin{remark}
Let $ \cat C_q' $ be the uncurving of $ \cat C_q $ by $ r $. Then the curvature of $ \cat C_q' $ is
\begin{equation*}
μ^0_{\cat C_q'} = μ^0_{\cat C_q} + μ^1_{\cat C_q} (r) + μ^2_{\cat C_q} (r, r) + ….
\end{equation*}
The uncurving $ \cat C_q' $ is naturally a deformation of $ \cat C $ and comes with a gauge equivalence
\begin{equation*}
F: \cat C_q' \isoto \cat C_q, \quad \text{given by} \quad F^0 ≔ r, \quad F^1 ≔ \Id, \quad F^{≥2} ≔ 0.
\end{equation*}
The category $ \cat C_q' $ can also be defined by forcing this particular map $ F $ to be a functor of deformations. More on uncurving can be found in \autoref{sec:uncurving-theory}, which focuses on cases where uncurving removes the curvature entirely.
\end{remark}

\begin{remark}
\label{rem:kadeishvili-optimizing-essential}
The name “uncurving” for the gauge in \autoref{def:kadeishvili-optimizing-uncurving} is a slight abuse of terminology: The curvature $ μ^0_{\cat C_q'} $ will not vanish, but has the chance to be less than $ μ^0_{\cat C_q} $. The term uncurving generally refer to any procedure of reducing curvature, while \autoref{def:kadeishvili-optimizing-uncurving} restricts usage of the term to a particular class of functors. According to the explanation in \autoref{sec:uncurving-theory}, this particular class of functors is however the only one that essentially changes curvature, therefore we have adopted the name “uncurving”.
\end{remark}

\begin{remark}
We can now explain the name “optimal curvature”. In fact, any other deformation gauge equivalent to a deformation with optimal curvature will generally have more curvature. To see this, regard a deformation $ \cat C_q $ with optimal curvature. Its curvature already lies in $ H_q $. If we apply uncurving by an element $ r $, the new curvature is $ μ^0_q + μ^1_q (r) + … $. If we choose $ r ∈ B \htensor R $, then $ μ^1_q (r) $ naturally lies in $ μ^1_q (B \htensor R) $ which already downgrades the curvature. If we choose $ r ∈ μ^1_q (B \htensor R) $ or $ r ∈ H_q $, then $ μ^1_q (r) $ typically contains components from $ R $ or $ μ^1_q (B \htensor R) $ as well. The additional summands $ μ^2_q (r, r) $ even worsen the situation. We see that $ μ^0_q ∈ H_q $ is generally the best achievable.
\end{remark}

In the remainder of this section, we prove that any deformed $ A_∞ $-category $ \cat C_q $ has an uncurving with optimal curvature. The idea is to apply repeated uncurving by elements $ s $ which lie in increasingly high order of $ \mathfrak{m} $. We take our clue from inspecting the curvature $ μ_{\cat C_q'}^0 = μ^0_{\cat C_q} + μ^1_{\cat C_q} (s) + … $.
To get $ μ^0_{\cat C_q'} $ as close to zero as possible, write $ μ^0_{\cat C_q} = h + μ^1_{\cat C_q} (r) + r' $ in terms of the deformed decomposition of $ \cat C_q $ and choose $ s = -r $. The curvature of $ \cat C_q' $ then reads
\begin{equation*}
μ^0_{\cat C_q'} = h + μ^1_{\cat C_q} (r) + r' + μ^1_{\cat C_q} (-r) + μ^2_{\cat C_q} (-r, -r) + … = h + r' + \landau(\mathfrak{m}^2).
\end{equation*}
This is very productive strategy, since the new curvature has lost its $ μ^1_q (R) $ component in lowest order. Our idea is to repeat this procedure to eliminate also the higher order terms. A repeated approach is indeed necessary because newly arising curvature terms like $ μ^2_{\cat C_q} (r, r) $ may behave unpredictably.

\begin{definition}
Let $ \cat C $ be an $ A_∞ $-category with deformation $ \cat C_q $. Let $ H ⊕ I ⊕ R $ be a homological splitting for $ \cat C $. The \emph{curvature optimization procedure} is the following inductive procedure, starting with $ i = 0 $ and $ \cat C_q^{(0)} ≔ \cat C_q $.
\begin{enumerate}
\item Form the deformed decomposition $ H_q^{(i)} ⊕ μ^1_{\cat C_q^{(i)}} (B \htensor R) ⊕ B \htensor R $ of $ \cat C_q^{(i)} $.
\item Write the curvature as $ μ^0_{\cat C_q^{(i)}} = h^{(i)} + μ^1_{\cat C_q^{(i)}} (r^{(i)}) + r^{(i)}{}' $ in terms of the decomposition.
\item Define $ \cat C_q^{(i+1)} $ to be the uncurving of $ \cat C_q^{(i)} $ by $ -r^{(i)} $.
\item Repeat.
\end{enumerate}
\end{definition}

\begin{remark}
The definition of $ r^{(i)} $ in terms of $ \cat C_q{}^{(i)} $ can also be written elegantly as $ r^{(i)} = h_q{}^{(i)} (μ^0{}_{\cat C_q{}^{(i)}}) $, where $ h_q{}^{(i)} $ is the deformed codifferential of $ \cat C_q{}^{(i)} $. The letters $ h^{(i)} $, $ r^{(i)} $ and $ r^{(i)}{}' $ are actually families parametrized by objects $ X ∈ \cat C $. In the statement of \autoref{th:kadeishvili-optimizing-th}, we combine this shorthand with the shorthand notation $ \End_{\cat C} = \{\End_{\cat C} (X)\}_{X ∈ \cat C} $. For instance, $ r^{(i)} ∈ \mathfrak{m}^{2^i} \End_{\cat C} $ is to be understood as $ r^{(i)}_X ∈ \mathfrak{m}^{2^i} \End_{\cat C} (X) $ for every $ X ∈ \cat C $.
\end{remark}

After running the curvature optimization procedure, we expect the gauges $ r^{(i)} $ to combine together to one large gauge. We expect the categories $ \cat C_q{}^{(i)} $ to converge to a limit category $ \cat C_q^{\bigmax} $. We also expect the curvature of $ \cat C_q{}^{(i)} $ to converge to the curvature of $ \cat C_q^{\bigmax} $ and the deformed decompositions of $ \cat C_q^{(i)} $ to converge to the deformed decomposition of $ \cat C_q^{\bigmax} $:
\begin{align*}
\lim_{i → ∞} \cat C_q^{(i)}  &=  \cat C_q^{\bigmax}, \\
\lim_{i → ∞} μ^0_{\cat C_q^{(i)}} &=  μ^0_{\cat C_q^{\bigmax}}, \\
\lim_{i → ∞} \big(H_q^{(i)}, μ^1_{\cat C_q^{(i)}} (B \htensor R),  (B \htensor R)\big) &= \big(H_q^{\bigmax},  μ^1_{\cat C_q^{\bigmax}} (B \htensor R), (B \htensor R)\big).
\end{align*}
Ultimately, we hope to find $ μ^0{}_{\cat C_q^{\bigmax}} ∈ H_q^{\bigmax} ⊕ (B \htensor R) $. The next lemma makes this precise, and also shows that we have in fact reached $ μ^0{}_{\cat C_q^{\bigmax}} ∈ H_q^{\bigmax} $ as desired.

\begin{lemma}
\label{th:kadeishvili-optimizing-th}
Let $ \cat C $ be an $ A_∞ $-category and $ \cat C_q $ a deformation. Let $ H ⊕ I ⊕ R $ be a homological splitting for $ \cat C $. Let $ r^{(i)} $ be the sequence obtained from applying the curvature optimization procedure to $ \cat C_q $. Then it holds that $ r^{(i)} ∈ \mathfrak{m}^{2^i} \End_{\cat C} $. Set $ r = \sum_{i ∈ ℕ} r^{(i)} ∈ \mathfrak{m} \End_{\cat C} $ and define $ \cat C_q^{\bigmax} $ as the uncurving of $ \cat C_q $ by $ -r $. Then $ \cat C_q^{\bigmax} $ has optimal curvature and comes with an gauge equivalence
\begin{equation*}
F: \cat C_q^{\bigmax} → \cat C_q, \quad \text{given by} \quad F^0 = -r, \quad F^1 = \Id, \quad F^{≥2} = 0.
\end{equation*}
\end{lemma}

\begin{proof}
We divide the proof into three parts. In the first part of the proof we show $ r^{(i)} ∈ \landau(\mathfrak{m}^{2^i}) $. Denote by $ H_q^{\bigmax} ⊕ μ^1_{\cat C_q^{\bigmax}} (B \htensor R) ⊕ (B \htensor R) $ the deformed decomposition of $ \cat C_q^{\bigmax} $. In the second part of the proof, we show that the curvature $ μ^0_{\cat C_q^{\bigmax}} $ lies in $ H_q^{\bigmax} ⊕ B \htensor R $. In the third part of the proof we conclude that the curvature actually lies in $ H_q^{\bigmax} $.

For the first part, let us show $ r^{(i)} ∈ \mathfrak{m}^{2^i} \End_{\cat C} $ by induction. For $ i = 0 $, the statement holds. Assume it holds for some $ i ∈ ℕ $. Recall that $ \cat C_q^{(i+1)} $ is the uncurving of $ \cat C_q^{(i)} $ by $ -r^{(i)} $. Its curvature is
\begin{align*}
μ^0_{\cat C_q^{(i+1)}} &= μ^0_{\cat C_q^{(i)}} + μ^1_{\cat C_q^{(i)}} (-r^{(i)}) + μ^2_{\cat C_q^{(i)}} (-r^{(i)}, -r^{(i)}) + … \\
&= h^{(i)} + μ^1_{\cat C_q^{(i)}} (r^{(i)}) + r^{(i)}{}' - μ^1_{\cat C_q^{(i)}} (r^{(i)}) + \landau(\mathfrak{m}^{2^{i+1}}) \\
&= h^{(i)} + r^{(i)}{}' + \landau(\mathfrak{m}^{2^{i+1}}).
\end{align*}
To make statements on $ r^{(i+1)} $, write $ h_q^{(i+1)} $ for the deformed codifferential of $ \cat C_q^{(i+1)} $. Then
\begin{align*}
r^{(i+1)} &= h_q^{(i+1)} (μ^0_{\cat C_q^{(i+1)}}) \\
&= h_q^{(i+1)} \big(h^{(i)} + r^{(i)}{}' + \landau(\mathfrak{m}^{2^{i+1}})\big) \\
&= 0 + 0 + \landau(\mathfrak{m}^{2^{i+1}}).
\end{align*}
In the last row, we have used that
\begin{equation*}
h^{(i)}, r^{(i)}{}' ∈ H_q^{(i)} ⊕ (B \htensor R) = (B \htensor H) ⊕ (B \htensor R) = H_q^{(i+1)} ⊕ (B \htensor R).
\end{equation*}
We have also used that $ h_q^{(i+1)}: B \htensor \Hom_{\cat C} → B \htensor R $ preserves the $ \mathfrak{m} $-adic filtration. The reason is that $ h_q^{(i+1)} $ is $ B $-linear and automatically continuous according to \autoref{th:prelim-htensor-autocontinuous}. In total, we arrive at $ r^{(i+1)} ∈ \mathfrak{m}^{2^{i+1}} R $ as claimed. This finishes the induction.

For the second part of the proof, we show that $ μ^0_{\cat C_q^{\bigmax}} $ lies in $ H_q^{\bigmax} ⊕ (B \htensor R) $. For every $ i ∈ ℕ $, regard
\begin{equation*}
μ^0_{\cat C_q^{(i)}} = h^{(i)} + μ^1_{\cat C_q^{(i)}} (r^{(i)}) + r^{(i)}{}'.
\end{equation*}
The left-hand side converges to $ μ^0_{\cat C_q^{\bigmax}} $. The third term on the right-hand side converges to zero. Together this means
\begin{equation*}
(B \htensor H) ⊕ (B \htensor R) \ni h^{(i)} + μ^1_{\cat C_q^{(i)}} (r^{(i)}) \xrightarrow{i → ∞} μ^0_{\cat C_q^{\bigmax}}.
\end{equation*}
We conclude
\begin{equation*}
μ^0_{\cat C_q^{\bigmax}} ∈ (B \htensor H) ⊕ (B \htensor R) = H_q^{\bigmax} ⊕ B \htensor R.
\end{equation*}
This finishes the second part of the proof.

For the third part of the proof, we show that $ μ^0{}_{\cat C_q^{\bigmax}} ∈ H_q^{\bigmax} $. The idea is to show that the $ R $ component of $ μ^0{}_{\cat C_q^{\bigmax}} $ vanishes. In fact, this is an easy a posteriori observation: Write this curvature as $ h + r $ with $ h ∈ H_q^{\bigmax} $ and $ r ∈ B \htensor R $. Then
\begin{equation*}
0 = μ^1_{\cat C_q^{\bigmax}} (μ^0_{\cat C_q^{\bigmax}}) = μ^1_{\cat C_q^{\bigmax}} (h) + μ^1_{\cat C_q^{\bigmax}} (r).
\end{equation*}
On the right hand side, the first summand lies in $ H_q^{\bigmax} ⊕ B \htensor R $ and the second summand lies in $ μ^1{}_{\cat C_q^{\bigmax}} (B \htensor R) $. Correspondingly, both summands vanish. While for $ h $ this is a weak statement, we immediately derive $ r = 0 $ since $ μ^1{}_{\cat C_q^{\bigmax}} $ is injective on $ B \htensor R $. This shows $ μ^0_{\cat C_q^{\bigmax}} ∈ H_q^{\bigmax} $ and finishes the proof.
\end{proof}

%% file: kadeishvili/auxiliary.tex
\subsection{Auxiliary minimal model procedure}
\label{sec:kadeishvili-auxiliary}
In this section, we construct auxiliary minimal models for deformed $ A_∞ $-categories with optimal curvature. The idea is to perform a construction with trees as in the classical case. For a given catgeory $ \cat C_q $ with optimal curvature, the first step in this section is to provide an explicit description of the auxiliary $ A_∞ $-structure on $ H_q $ and a functor $ F_q: H_q → \cat C_q $. We then check that the auxiliary minimal model satisfies the curved $ A_∞ $-axioms and that $ F_q $ satisfies the curved $ A_∞ $-functor axioms.

\begin{center}
\begin{tikzpicture}
\path (0, 0) node (A) {Deformation $ \cat C_q $ with optimal curvature} (8, 0) node[align=center] (B) {Auxiliary deformation $ H_q $ \\ Auxiliary functor $ F_q: H_q → \cat C_q $};
\path[draw, decorate, decoration={snake, amplitude=0.2em, post length=0.5em}, ->] ($ (A.east)!0.2!(B.west) $) to ($ (A.east)!0.8!(B.west) $);
\end{tikzpicture}
\end{center}

\begin{remark}
The material in this section is considered auxiliary because we only construct an $ A_∞ $-structure on $ H_q $ and not on $ B \htensor H $. The $ A_∞ $-structure on $ B \htensor H $ is obtained in \autoref{sec:kadeishvili-general} simply by transfer via $ φ: H_q → B \htensor H $.
\end{remark}

To define the auxiliary $ A_∞ $-structures, we have to set up some context. Let $ \cat C $ be an $ A_∞ $-category and $ H \oplus I \oplus R $ a homological splitting. Let $ \cat C_q $ be a deformation with optimal curvature. Denote by $ H_q \oplus μ^1_{\cat C_q} (B \htensor R) \oplus (B \htensor R) $ the deformed decomposition of $ \cat C_q $. We use the following notation:

\begin{definition}
Consider a sequence $ h_1, …, h_n $ of $ n \geq 2 $ morphisms with $ h_i ∈ H_q (X_i, X_{i+1}) $. Let $ T ∈ \mathcal{T}_n $ be a Kadeishvili tree shape with $ n $ leaves, as in \autoref{sec:kadeishvili-classical}. Define
\begin{equation*}
\Res_q (T, h_1, …, h_n) ∈ H_q (X_1, X_{n+1})
\end{equation*}
to be the evaluation of $ T $ by decorating the leaves with the inputs $ h_1, …, h_n $, the internal nodes by $ h_q μ_q $ and the root by $ π_q μ_q $. Define
\begin{equation*}
\Res_q^h (T, h_1, …, h_n) ∈ B \htensor R(X_1, X_{n+1})
\end{equation*}
to be the evaluation of $ T $ by decorating the leaves with the inputs $ h_1, …, h_n $ and all other nodes by $ h_q μ_q $, including the root.
\end{definition}

\begin{example}
A few sample decorated trees for the definition of $ \Res_q (T, h_1, …, h_n) $ are depicted in \autoref{fig:kadeishvili-auxiliary-trees}. For instance, the first three trees give results $ π_q μ_q^2 (h_2, h_1) $, $ π_q μ_q^3 (h_3, h_2, h_1) $, $ π_q μ_q^2 (h_q μ_q^2 (h_3, h_2), h_1) $.
\end{example}

We temporarily by $ i $ the inclusion map of $ H_q (X_1, X_2) $ into $ \Hom_{\cat C_q} (X_1, X_2) $. With these preparations, we are ready to define auxiliary product structure on the collection of spaces $ H_q = \{H_q (X, Y)\}_{X, Y ∈ \cat C_q} $ and an auxiliary mapping $ F_q: H_q → \cat C_q $:

\begin{definition}
The \emph{auxiliary product structure} on $ H_q $ is defined as follows:
\begin{align*}
μ^0_{H_q} &= μ^0_{\cat C_q}, \\
μ^1_{H_q} &= π_q μ^1_q \restr{H_q}, \\
μ_{H_q}^{n \geq 2} (h_n, …, h_1) &= \sum_{T ∈ \mathcal{T}_n} (-1)^{N_T} \Res_q (T, h_1, …, h_n).
\end{align*}
The candidate functor $ F_q: H_q → \cat C_q $ is defined by
\begin{align*}
F_q^0 &= 0, \\
F_q^1 &= i, \\
F_q^{n \geq 2} (h_n, …, h_1) &= \sum_{T ∈ \mathcal{T}_n} (-1)^{N_T + 1} \Res_q^h (T, h_1, …, h_n).
\end{align*}
\end{definition}

In words, $ μ^0_{H_q} $ is defined as $ μ^0_{\cat C_q} $ which already lies in $ H_q $ since $ \cat C_q $ has optimal curvature. The differential $ μ^1_{H_q} $ is defined by projecting $ μ^1_{\cat C_q} $ down to $ H_q $. All higher products are given by trees. The functor component $ F_q^0 $ is set to zero, the component $ F^1_q $ is the natural embedding of $ H_q (X_1, X_2) $ into the hom space $ \Hom_{\cat C_q} (X_1, X_2) $ and the higher components of $ F_q $ are given by trees.

\begin{figure}
\centering
\begin{tikzpicture}
\path node (A) {$ h_2 $} node[right of=A] (B) {$ h_1 $}
node[below right of=A] {$ π_q μ_q $} edge (A) edge (B);
\path node[right of=B] (C) {$ h_3 $} node[right of=C] (D) {$ h_2 $} node[right of=D] (E) {$ h_1 $}
node[below of=D] {$ π_q μ_q $} edge (C) edge (D) edge (E);
\path node[right of=E] (F) {$ h_3 $} node[right of=F] (G) {$ h_2 $} node[right of=G] (H) {$ h_1 $}
node[below right of=F] (I) {$ h_q μ_q $} edge (F) edge (G)
node[below right of=I] {$ π_q μ_q $} edge (I) edge (H);
\path node[right of=H] (F1) {$ h_3 $} node[right of=F1] (G1) {$ h_2 $} node[right of=G1] (H1) {$ h_1 $}
node[below right of=G1] (I1) {$ h_q μ_q $} edge (G1) edge (H1)
node[below left of=I1] {$ π_q μ_q $} edge (I1) edge (F1);
\path node[right of=H1] (J) {$ h_4 $} node[right of=J] (K) {$ h_3 $} node[right of=K] (L) {$ h_2 $} node[right of=L] (M) {$ h_1 $}
node[below of=K] (N) {$ h_q μ_q $} edge (J) edge (K) edge (L)
node[below right of=N] {$ π_q μ_q $} edge (N) edge (M);
\end{tikzpicture}
\caption{Decorating Kadeishvili π-trees for $ μ_{H_q} $}
\label{fig:kadeishvili-auxiliary-trees}
\end{figure}
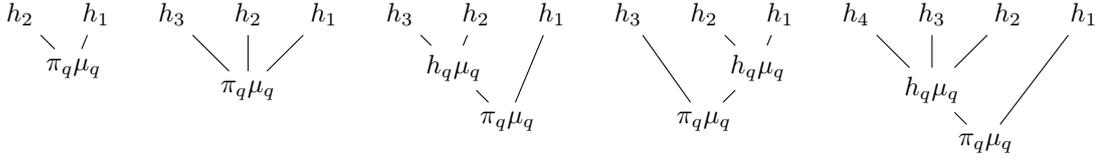

Checking the functor relations for $ F_q $ entails switching around projections $ π_q $ and codifferentials $ h_q $. We need to prepare for this with a simple lemma:

\begin{lemma}
\label{th:kadeishvili-auxiliary-projections}
Let $ \cat C_q $ be a deformed $ A_∞ $-category with optimal curvature. Then the projections to $ B \htensor R $ and $ μ^1_{\cat C_q} (B \htensor R) $ with respect to the deformed decomposition can be written as
\begin{align*}
π_{μ^1_q (B \htensor R)} &= μ^1_q h_q, \\
π_{B \htensor R} &= h_q μ^1_q - h_q μ^1_q μ^1_q h_q.
\end{align*}
\end{lemma}

\begin{proof}
Every hom space in $ \cat C_q $ is the direct sum of the three components $ H_q $, $ μ^1_q (B \htensor R) $ and $ B \htensor R $. Therefore it suffices to check the identities on these three spaces individually.

On $ H_q $, both sides of the first formula evaluate to zero by definition. In the second formula, the right hand side evalualates to zero as well, because $ μ^1_q (H_q) ⊂ H_q ⊕ μ^1_q (B \htensor R) $.

On $ B \htensor R $, both sides of the first formula evaluate to zero by definition. In the second formula, the first term evaluates indeed to the identity and the second term vanishes.

On $ μ^1_q (B \htensor R) $, both sides of the first formula evaluate to the identity. To check the second formula, regard an arbitrary element $ μ^1_q (r) $ with $ r ∈ B \htensor R $. Then
\begin{equation*}
h_q μ^1_q (μ^1_q (r)) - h_q μ^1_q μ^1_q h_q (μ^1_q (r)) = h_q μ^1_q μ^1_q (r) - h_q μ^1_q μ^1_q (r) = 0.
\end{equation*}
We conclude that the claimed identities hold on all three direct summand spaces, finishing the proof.
\end{proof}

In \autoref{th:kadeishvili-auxiliary-Hq}, we prove the desired $ A_∞ $-relations for $ H_q $ and $ F_q: H_q → \cat C_q $. Strictly speaking, $ H_q $ itself is not a deformation of any $ A_∞ $-category. However, it makes perfect sense to check the curved $ A_∞ $-relations for the structure $ μ_{H_q} $ defined on $ H_q $. The product structure $ μ_{H_q} $ is merely an auxiliary tool and will disappear again in \autoref{sec:kadeishvili-general}.

\begin{lemma}
\label{th:kadeishvili-auxiliary-Hq}
Let $ \cat C $ be an $ A_∞ $-category and $ \cat C_q $ a deformation. Assume $ \cat C_q $ has optimal curvature. Then $ μ_{H_q} $ satisfies the curved $ A_∞ $-relations and $ F_q: H_q → \cat C_q $ satisfies the curved $ A_∞ $-functor relations.
\end{lemma}

\begin{proof}
We prove the statements in reverse order: First we show that $ F_q $ satisfies the curved $ A_∞ $-functor relations. Second we conclude that $ μ_{H_q} $ satisfies the curved $ A_∞ $-relations.

For the first part, regard the curved $ A_∞ $-functor relations for $ F_q $:
\begin{equation}
\label{eq:kadeishvili-auxiliary-functor}
\sum (-1)^{‖a_1‖ + … + ‖a_j‖} F_q (a_k, …, μ_{H_q} (a_i, …, a_{j+1}), …, a_1) = \sum μ_{\cat C_q} (F_q (a_k, …), …, F_q (…, a_1)).
\end{equation}
We shall first prove these relations separately for $ k = 0 $ and $ k = 1 $ and then for general $ k ≥ 2 $. For $ k = 0 $, the relation reads $ F_q^1 (μ^0_{H_q}) = μ^0_{H_q} $ since $ F_q^0 = 0 $. This relation holds true by definition of $ F_q^1 $ and $ μ^0{}_{H_q} $. For $ k = 1 $, we calculate
\begin{align*}
F_q^1 (μ_{H_q}^1 (a)) + F_q^2 (a, μ_{H_q}^0) + (-1)^{‖a‖} F_q^2 (μ_{H_q}^0, a) &= π_q (μ_q (a)) - h_q μ_q^2 (a, μ_q^0) - (-1)^{‖a‖} h_q μ_q^2 (μ_q^0, a) \\
&= π_q (μ_q (a)) + h_q (μ_q (μ_q (a))) \\
&= μ_q (a) = μ_{\cat C_q} (F_q^1 (a)).
\end{align*}
In the second equality, we have used the curved $ A_∞ $-relations of $ \cat C_q $. In the third equality, we have used the property of the deformed decomposition that $ μ_q (H_q) ⊂ H_q ⊕ (B \htensor R) $. This settles the cases $ k = 0, 1 $.

Let us now prove \eqref{eq:kadeishvili-auxiliary-functor} in case $ k ≥ 2 $ by projecting both sides onto $ H_q $, $ μ^1_q (B \htensor R) $ and $ B \htensor R $ individually. First, regard the projection on $ H_q $. Since $ F_q^{≥2} $ has image in $ B \htensor R $, we have
\begin{equation*}
π_{H_q} (\text{LHS}) = F_q^1 (μ_{H_q} (a_k, …, a_1)) = μ_{H_q} (a_k, …, a_1) = \sum π_q μ_q (F_q (…), …, F_q (…)) = π_{H_q} (\text{RHS}).
\end{equation*}
We have used nothing but the definition of $ μ_{H_q} $ and $ F_q $. Now regard the projection on $ μ^1_{\cat C_q} (B \htensor R) $. We have
\begin{align*}
π_{μ^1_{\cat C_q} (B \htensor R)} (\text{RHS}) &= μ^1_q h_q μ_q^{≥2} (F_q (…), …, F_q (…)) + μ^1_q h_q μ^1_q F_q (…) \\
&= - μ^1_q F_q (…) + μ^1_q F_q (…) = 0 \\
& = π_{μ^1_{\cat C_q} (B \htensor R)} (\text{LHS}).
\end{align*}
Lastly, regard the projection to $ B \htensor R $. We have
\begin{align*}
π_{B \htensor R} (\text{LHS}) & = (-1)^{‖a_1‖ + … + ‖a_i‖} F_q^{≥2} (…, μ_{H_q}^{≥0} (…), a_i, …) \\
& = (-1)^{‖a_1‖ + … + ‖a_i‖ + 1} h_q μ_q^{≥2} (F_q, …, F_q (…, μ_{H_q}^{≥0}, a_i, …), F_q (a_j, …), …, F_q) \\
& = (-1)^{‖a_1‖ + … + ‖a_j‖ + 1} h_q μ_q^{≥2} (F_q, …, μ_q^{≥0} (F_q, …, F_q), F_q (a_j, …), …, F_q) \\
& = (-1)^{1+1} h_q μ^1_q μ_q^{≥2} (F_q, …, F_q) + (-1)^{1+1} h_q μ^1_q μ^1_q F_q (…) \\
& = + h_q μ^1_q μ_q^{≥2} (F_q, …, F_q) - h_q μ^1_q μ^1_q h_q μ_q^{≥2} (F_q, …, F_q) \\
& = π_{B \htensor R} (μ_q^{≥2} (F_q, …, F_q)) \\
& = π_{B \htensor R} (\text{RHS}).
\end{align*}
In the second equality, we have unraveled the definition of $ F_q $. In the third equality, we have assumed towards induction that \eqref{eq:kadeishvili-auxiliary-functor} already holds for a shorter sequence of inputs. In the fourth equality, we have used the (curved) $ A_∞ $-relation for $ \cat C_q $. In the fifth equality, we have unraveled the definition of $ F_q $ again. In the sixth equality, we have used the expression for the projection according to \autoref{th:kadeishvili-auxiliary-projections}. Finally, we conclude that \eqref{eq:kadeishvili-auxiliary-functor} holds on the entire hom spaces of $ \cat C_q $. In other words, $ F_q $ satisfies the (curved) $ A_∞ $-functor relations.

For the second part of the proof, we show that $ μ_{H_q} $ satisfies the curved $ A_∞ $-relations. The trick is to apply $ F_q^1 $ to the $ A_∞ $-relations for $ μ_{H_q} $ and pull terms from inside to outside using the just proven fact that $ F_q $ satisfies the (curved) $ A_∞ $-functor relations. We calculate
\begin{align*}
& F^1_q (μ_{H_q} (a_k, …, μ_{H_q}^{≥0} (…), …, a_1)) \\
&= F_q^{≥2} (a_k, …, μ_{H_q}^{≥1} (…, μ_{H_q}^{≥0} (…), …), …, a_1) + F_q^{≥2} (a_k, …, μ_{H_q}^{≥0} (…), …, μ_{H_q}^{≥0} (…), …) \\
& \quad + μ_q^{≥1} (F_q (…), …, F_q (…, μ_{H_q}^{≥0} (…), …), …, F_q (…)) \\
&= F_q^{≥2} (a_k, …, 0, …, a_1) + 0 + μ_q^{≥1} (F_q (…), …, μ_q^{≥1} (F_q (…), …, F_q (…)), …, F_q (…)) \\
&= 0 + 0 + 0.
\end{align*}
In the first equality, we have used that $ F_q $ satisfies the (curved) $ A_∞ $-functor relation on the sequence $ a_1, …, μ_{H_q}^{≥0} (…), …, a_k $. The two terms on the second row come from the $ A_∞ $-functor relation, and are distinguished by the choice whether the inner $ μ_{H_q}^{≥0} (…) $ is inserted into the new product $ μ_{H_q}^{≥0} $ or not. The terms of the type $ F_q (μ_{H_q}, μ_{H_q}) $ however appear pairwise and cancel each other. In the second equality, we have used the assumption that $ μ_{H_q} $ already satisfies the $ A_∞ $-relations on shorter sequences.

Finally, we note that $ F^1_q $ is injective on $ H_q $ and therefore $ μ_{H_q} $ satisfies the $ A_∞ $-relations on the sequence $ a_1, …, a_k $. In total, we conclude that $ μ_{H_q} $ satisfies the curved $ A_∞ $-relations.
\end{proof}

%% file: kadeishvili/general.tex
\subsection{The deformed Kadeishvili theorem}
\label{sec:kadeishvili-general}
In this section, we provide our most general Kadeishvili theorem for deformed $ A_∞ $-categories. The starting point is an arbitrary deformed $ A_∞ $-category $ \cat C_q $ and the goal is to find a minimal model for $ \cat C_q $ in the sense of \autoref{def:kadeishvili-existence-def}. The idea is to apply the curvature optimization procedure to $ \cat C_q $, then take the auxiliary minimal model in the sense of \autoref{sec:kadeishvili-auxiliary} and to pull back the structure in order to form a deformation of $ \H\cat C $.

\begin{center}
\begin{tikzpicture}
\path (0, 0) node (A) {Deformed $ A_∞ $-category $ \cat C_q $} (8, 0) node[align=center] (B) {Minimal model $ \H\cat C_q $};
\path[draw, decorate, decoration={snake, amplitude=0.2em, post length=0.5em}, ->] ($ (A.east)!0.2!(B.west) $) to ($ (A.east)!0.8!(B.west) $);
\end{tikzpicture}
\end{center}

\begin{definition}
\label{def:kadeishvili-general-def}
Let $ \cat C $ be an $ A_∞ $-category and $ \cat C_q $ a deformation. Let $ H \oplus I \oplus R $ be a homological splitting for $ \cat C $. Apply the curvature optimization procedure to $ \cat C_q $. Let $ \cat C_q^{\bigmax} $ be the result and $ r $ be the gauge used. Denote by $ H_q^{\bigmax} \oplus μ^1{}_{\cat C_q^{\bigmax}} (B \htensor R) \oplus (B \htensor R) $ the deformed decomposition of $ \cat C_q^{\bigmax} $ and by $ φ: H_q^{\bigmax} → B \htensor H $ the associated isomorphism. Apply the auxiliary minimal model procedure to $ \cat C_q^{\bigmax} $. Let $ μ_{H_q^{\bigmax}} $ be the resulting auxiliary $ A_∞ $-structure and $ F_q^{\bigmax}: H_q^{\bigmax} → \cat C_q^{\bigmax} $ be the auxiliary functor.

Then we define the $ A_∞ $-structure $ μ_{\H\cat C_q} $ on $ B \htensor H $ and the functor $ F_q: \H\cat C_q → \cat C_q $ by
\begin{equation}
\label{eq:kadeishvili-general-def}
\begin{aligned}
μ_{\H\cat C_q} &= φ ∘ μ_{H_q^{\bigmax}} ∘ φ^{-1}, \\
F_q &= (\Id - r) ∘ F_q^{\bigmax} ∘ φ^{-1}.
\end{aligned}
\end{equation}
\end{definition}

\begin{remark}
In \eqref{eq:kadeishvili-general-def}, the circle symbol denotes composition of curved $ A_∞ $-functors. By abuse of notation, we have denoted the gauge functor from the curvature optimization procedure by $ \Id - r $, standing for the functor with 0-ary component $ -r $ and 1-ary component $ \Id $ and vanishing higher components. Furthermore, we have interpreted $ φ $ as an $ A_∞ $-functor $ H_q^{\bigmax} → B \htensor H $ with only a 1-ary component. More explicitly, the definition for $ μ_{\H\cat C_q} $ reads
\begin{equation*}
μ_{\H\cat C_q}^{n \geq 0} (h_n, …, h_1) ≔ φ μ^n_{H_q^{\bigmax}} (φ^{-1} (h_n), …, φ^{-1} (h_1)).
\end{equation*}
\end{remark}

\begin{theorem}
\label{th:kadeishvili-general-th}
Let $ \cat C $ be an $ A_∞ $-category. Let $ H \oplus I \oplus R $ be a homological splitting for $ \cat C $ and let $ \H\cat C $ be the minimal model obtained from this splitting. Then $ \H\cat C_q $ is an $ A_∞ $-deformation of $ \H\cat C $ and $ F_q: \H\cat C_q → \cat C_q $ is a quasi-isomorphism of deformed $ A_∞ $-categories. In particular, $ \H\cat C_q $ is a minimal model for $ \cat C_q $.
\end{theorem}

\begin{proof}
It is our task to unwrap all definitions and to apply \autoref{th:kadeishvili-auxiliary-Hq}. The application of the curvature optimization procedure has made $ \cat C_q^{\bigmax} $ a category related to $ \cat C_q $ by the gauge equivalence $ \Id - r: \cat C_q^{\bigmax} → \cat C_q $. Subsequent application of the auxiliary minimal model procedure has given $ H_q^{\bigmax} $ an $ A_∞ $-structure with a functor $ F_q^{\bigmax}: H_q^{\bigmax} → \cat C_q^{\bigmax} $. Pulling back has given a product structure on $ \H\cat C_q $.

A first observation is that $ \H\cat C_q $ satisfies the curved $ A_∞ $-axioms. Indeed, it was merely pulled back form $ H_q^{\bigmax} $ and the product structure on $ H_q^{\bigmax} $ in turn satisfies the $ A_∞ $-axioms due to \autoref{th:kadeishvili-auxiliary-Hq}. The leading term of $ μ_{\H\cat C_q} $ is easily seen to be $ μ_{\H\cat C} $ and hence $ \H\cat C_q $ is a deformation of $ \H\cat C $.

A second observation is that with respect to the three curved $ A_∞ $-structures on $ \H\cat C_q $, $ H_q^{\bigmax} $, $ \cat C_q^{\bigmax} $, $ \cat C_q $, the following three mappings mappings define curved $ A_∞ $-functors:
\begin{equation*}
φ: H_q^{\bigmax} → \H\cat C_q, \qquad F_q^{\bigmax}: H_q^{\bigmax} → \cat C_q^{\bigmax}, \qquad \Id - r: \cat C_q^{\bigmax} → \cat C_q.
\end{equation*}
For $ φ $ and $ \Id - r $, this is the case by definition of pullback/uncurving. For $ F_q^{\bigmax} $, this is the statement of \autoref{th:kadeishvili-auxiliary-Hq}. In summary, $ F_q $ is merely a composition of these three functors:
\begin{equation*}
F_q: \H\cat C_q \xrightarrow{φ^{-1}} H_q^{\bigmax} \xrightarrow{F_q^{\bigmax}} \cat C_q^{\bigmax} \xrightarrow{\Id - r} \cat C_q.
\end{equation*}
We conclude that $ F_q $ itself is a curved $ A_∞ $-functor. Its leading term is the functor $ F: \H\cat C → \cat C $ obtained from the Kadeishvili construction for the non-deformed category $ \cat C $. Since $ F $ is a quasi-isomorphism, we conclude that $ F_q $ is a quasi-isomorphism in the sense of \autoref{def:curved-functors-quasiequivalence}. This settles all claims and proves that $ \H\cat C_q $ is a minimal model for $ \cat C_q $ in the sense of \autoref{def:kadeishvili-existence-def}.
\end{proof}

%% file: kadeishvili/caseD.tex
\subsection{The $ D = 0 $ case}
\label{sec:kadeishvili-caseD}
In this section, we examine the deformed Kadeishvili construction in a special case. The starting point is a curvature-free deformed $ A_∞ $-category $ \cat C_q $ where the deformed differential satisfies $ μ^1_q (H) ⊂ μ^1_q (B \htensor R) $. It turns out that in this case, the differential $ μ^1_{\H\cat C_q} $ on the minimal model vanishes. We re-interpret this case as an instance of the Kadeishvili theorem over base rings.

As we explain in \autoref{th:kadeishvili-caseD-crit}, there are multiple ways of saying $ μ^1_{\cat C_q} (H) ⊂ μ^1_{\cat C_q} (B \htensor R) $. One of them is requiring the operator $ D $ appearing in the description \autoref{th:kadeishvili-deformed-badsplitting} to vanish. We may therefore also call the present assumption the “$ D = 0 $ case”.

\begin{lemma}
\label{th:kadeishvili-caseD-crit}
Let $ \cat C $ be an $ A_∞ $-category and $ \cat C_q $ a deformation. Let $ H ⊕ I ⊕ R $ be a homological splitting of $ \cat C $. Denote by $ D, E, F $ the operators from \autoref{th:kadeishvili-deformed-badsplitting}. If $ \cat C_q $ is curvature-free, then we have $ D^2 = 0 $ and $ F = -ED $ and the following statements are equivalent:
\begin{enumerate}
\item For every $ h ∈ B \htensor H $ there exists an $ ε ∈ B \htensor R $ such that $ μ^1_q (h) = μ^1_q (ε) $.
\item We have $ μ^1_q (H) ⊂ μ^1_q (B \htensor R) $.
\item We have $ D = F = 0 $.
\item We have $ D = 0 $.
\end{enumerate}
In the first statement, the element $ ε $ is necessarily infinitesimal: $ ε ∈ \mathfrak{m} R $. Similarly, the right hand side of the inclusion in $ μ^1_q (H) ⊂ μ^1_q (B \htensor R) $ can be replaced by $ μ^1_q (\mathfrak{m} R) $.
\end{lemma}

\begin{proof}
Thanks to curvature-freeness, the differential $ μ^1_q $ squares to zero. The identities $ D^2 = 0 $ and $ F = -ED $ now follow from evaluating $ (μ^1_q)^2 = 0 $ with respect to the matrix presentation \eqref{eq:kadeishvili-deformed-matrixshape}. The four enumerated statements are all different ways of stating the condition $ D = F = 0 $. The only nontrivial observation is that $ D = 0 $ already implies $ F = 0 $ since $ F = -ED $. For the final infinitesimality observations, note that $ μ^1_q (h) $ is necessarily infinitesimal, since $ μ^1 (H) = 0 $ and $ μ^1_q $ is only an infinitesimal deformation of $ μ^1 $.
\end{proof}

The deformed decomposition of $ \cat C_q $ has very favorable properties if $ \cat C_q $ is curvature-free and satisfies $ D = 0 $:

\begin{lemma}
\label{th:kadeishvili-caseD-matrix}
Let $ \cat C $ be an $ A_∞ $-category and $ \cat C_q $ a deformation. Let $ H ⊕ I ⊕ R $ be a homological splitting of $ \cat C $. Assume $ \cat C_q $ is curvature-free and $ D = 0 $. With respect to the deformed decomposition of $ \cat C_q $, the differential $ μ^1_q $ takes the shape
\begin{equation*}
μ^1_q = \pmat{0 & 0 & 0 \\ 0 & 0 & * \\ 0 & 0 & 0}.
\end{equation*}
\end{lemma}

\begin{proof}
For the third column, note that $ μ^1_q $ by definition sends $ B \htensor R $ to $ μ^1_q (B \htensor R) $. For the second column, note that $ μ^1_q $ squares to zero. For the first column, pick an element $ h - Eh ∈ H_q $ with $ h ∈ B \htensor H $. Then by the first column of \eqref{eq:kadeishvili-deformed-matrixshape} we have $ μ^1_q (h) = μ^1_q Eh $ and hence $ μ^1_q (h - Eh) = 0 $. This finishes the proof.
\end{proof}

According to \autoref{th:kadeishvili-caseD-matrix}, the deformed decomposition in case $ D = 0 $ has many properties which we expect from a homological splitting. For this reason, we establish the following terminology alias:

\begin{definition}
Let $ \cat C $ be an $ A_∞ $-category and $ \cat C_q $ a deformation. Let $ H ⊕ I ⊕ R $ be a homological splitting of $ \cat C $. Assume $ \cat C_q $ is curvature-free and $ D = 0 $. Then the \emph{deformed homological splitting} of $ \cat C_q $ is the deformed decomposition $ H_q ⊕ μ^1_q (B \htensor R) ⊕ (B \htensor R) $.
\end{definition}

The minimal model $ \H\cat C_q $ has favorable properties in case $ \cat C_q $ is curvature-free and $ D = 0 $. In fact, both curvature $ μ^0_{\H\cat C_q} $ and differential $ μ^1_{\H\cat C_q} $ vanish by construction. The higher products are computed by Kadeishvili trees, putting $ φ^{-1} $ on every leaf, $ h_q μ^{\cat C_q} $ on every internal node and $ φ π_q μ^{\cat C_q} $ on the root. The entire procedure can be summarized as follows:

\begin{corollary}
\label{th:kadeishvili-caseD-th}
Let $ \cat C $ be an $ A_∞ $-category and $ H \oplus I \oplus R $ a homological splitting. Let $ \cat C_q $ be a curvature-free deformation of $ \cat C $ with $ D = 0 $. Then a minimal model $ \H\cat C_q $ is determined by the following procedure:
\begin{enumerate}
\item For every $ h ∈ H $, let $ ε_h ∈ B \htensor R $ such that $ μ^1_{\cat C_q} (h) = μ^1_{\cat C_q} (ε_h) $.
\item Define $ H_q = \{h - ε_h \running h ∈ H\} $.
\item Define $ φ: H_q → B \htensor H $ by $ h - ε_h ↦ h $.
\item Calculate the deformed codifferential $ h_q : H_q \oplus μ^1_{\cat C_q} (B \htensor R) \oplus (B \htensor R) → B \htensor R $.
\item Calculate the deformed projection $ π_q : H_q \oplus μ^1_{\cat C_q} (B \htensor R) \oplus (B \htensor R) → H_q $.
\item Set $ μ^1_{\H\cat C_q} = μ^0_{\H\cat C_q} = 0 $.
\item Regard arbitrary Kadeishvili tree shapes $ T $.
\item Decorate $ T $ with $ φ^{-1} $ at the leaves, $ h_q μ_{\cat C_q} $ at the internal nodes and $ π_q μ_{\cat C_q} $ at the root.
\item Define $ μ^{\geq 2}_{\H\cat C_q} (h_k, …, h_1) $ as sum over the result of these trees, with sign $ (-1)^{N_T} $.
\end{enumerate}
\end{corollary}

We would like to provide an aftermath to this corollary. More precisely, we will offer an independent explanation of the condition $ D = 0 $. There is namely a classical Kadeishvili theorem that works for $ A_∞ $-categories defined over rings: Let $ S $ be a ring and $ \cat C $ an $ S $-linear $ A_∞ $-algebra. If the cohomology $ \H(A) $ is a projective $ S $-module, then the projection $ \Ker(μ^1_A) \projects \H(A) $ has a lift $ \H(A) → A $ which is an $ S $-linear quasi-isomorphism of complexes. The original construction of Kadeishvili builds noncanonically a minimal $ A_∞ $-structure on $ \H(A) $ together with an $ A_∞ $-quasi-isomorphism $ \H(A) → A $. This version of the Kadeishvili theorem can be found for instance in \cite{Petersen}.

It is natural to apply this Kadeishvili theorem to curvature-free deformed $ A_∞ $-categories. In fact, if $ \cat C_q $ is a curvature-free deformed $ A_∞ $-category, then the classical Kadeishvili theorem gives a minimal model under the condition that $ \H\Hom_{\cat C_q} (X, Y) $ are projective $ B $-modules for every $ X, Y ∈ \cat C $:

\begin{center}
\begin{tikzpicture}
\path (0, 0) node[align=center] {$ μ^0_{\cat C_q} = 0 $ and \\ $ \H\Hom_{\cat C_q} (X, Y) $ projective $ B $-modules} (4, 0) node {\Large $ \Longrightarrow $} (8, 0) node[align=center] {Classical Kadeishvili \\ applies to $ \cat C_q $};
\end{tikzpicture}
\end{center}

In \autoref{th:kadeishvili-deformed-coh} we show that curvature-freeness together with $ D = 0 $ implies the projectivity condition. In particular, we recover \autoref{th:kadeishvili-caseD-th} as a consequence of the classical Kadeishvili theorem, under the technical assumption that $ \H\Hom_{\cat C} (X, Y) $ is finite-dimensional for all $ X, Y ∈ \cat C $. We have opted for hiding the quantification by $ X, Y ∈ \cat C $ in some cases and making it explicit in other cases.

\begin{lemma}
\label{th:kadeishvili-deformed-coh}
Let $ \cat C $ be an $ A_∞ $-category and $ \cat C_q $ a curvature-free deformation. Choose a homological splitting $ H ⊕ I ⊕ R $. Denote by $ D, E, F $ the operators from \autoref{th:kadeishvili-deformed-badsplitting}. Then we have a natural quasi-isomorphism of $ B $-modules
\begin{equation}
\label{eq:kadeishvili-deformed-coh}
\H\big(\Hom_{\cat C_q}, μ^1_q\big) ~\cong~ \H\big(B \htensor H, D\big).
\end{equation}
In particular if $ D = 0 $ and $ H(X, Y) $ is finite-dimensional, then $ \H(\Hom_{\cat C_q} (X, Y), μ^1_{\cat C_q}) $ is a projective $ B $-module.
\end{lemma}

\begin{proof}
To equate the two cohomology modules, we provide explicit morphisms $ φ, ψ $ of chain complexes in both directions. Next, we check that the maps actually commute with the differential. We finally show that in cohomology, both compositions $ φψ $ and $ ψφ $ descend to the identity.

Our first step is to give explicit morphisms of chain complexes. In terms of the decomposition \eqref{eq:kadeishvili-deformed-badsplitting} of $ \Hom_{\cat C_q} $ into $ B \htensor H ⊕ μ^1_q (B \htensor R) ⊕ B \htensor R $, put
\begin{align*}
φ: (\Hom_{\cat C_q} , μ^1_q) &\longrightarrow (B \htensor H, D), \\
(h, μ^1_q (r'), r) &\longmapsto h, \qquad \text{for} \quad h ∈ B \htensor H, \quad r, r' ∈ B \htensor R, \\
ψ: (B \htensor H, D) &\longrightarrow (\Hom_{\cat C_q} , μ^1_q), \\
h &\longmapsto (h, 0, -Eh), \qquad \text{for} \quad h ∈ B \htensor H.
\end{align*}
We are now ready to check that both $ φ $ and $ ψ $ are chain maps. Indeed, we have
\begin{equation*}
φ(μ^1_q (h, μ^1_q (r'), r)) = φ(Dh, μ^1_q (Eh) + μ^1_q (r), Fh) = Dh = D φ(h, μ^1_q (r'), r)
\end{equation*}
and
\begin{equation*}
ψ(Dh) = (Dh, 0, -EDh) = (Dh, 0, Fh) = μ^1_q (h, 0, -Eh) = μ^1_q (ψ(h)).
\end{equation*}
In the above calculations, we have written elements of the direct sum $ B \htensor H ⊕ μ^1_q (B \htensor R) ⊕ B \htensor R $ as tuples instead of sums.

The next step is to calculate $ φψ $ and $ ψφ $ and verify that they descend to the identity on cohomology. For $ φψ $, this is trivial since $ φψ = \id $. For the other composition $ ψφ $, we pick an element $ x ∈ \Ker(μ^1_q) $ and check whether $ ψ(φ(x)) - x $ lies in the image of $ μ^1_q $. First of all, write $ x = (h, μ^1_q (r'), r) $ and note that $ x ∈ \Ker(μ^1_q) $ implies $ Dh = 0 $, $ r + Eh = 0 $ and $ Fh = 0 $. We obtain
\begin{equation*}
ψ(φ(x)) - x = ψ(h, μ^1_q (r'), r) - (h, μ^1_q (r'), r) = (h, 0, -Eh) - (h, μ^1_q (r'), r) = (0, μ^1_q (r'), 0).
\end{equation*}
The expression on the right is simply the image $ μ^1_q (0, 0, r') $. We conclude that $ φ $ and $ ψ $ are quasi-inverse to each other. This establishes the desired quasi-isomorphism \eqref{eq:kadeishvili-deformed-coh}. In case $ D = 0 $ and $ H(X, Y) $ is finite-dimensional, the cohomology is simply $ B \tensor H(X, Y) $ which is projective. This finishes the proof.
\end{proof}


With the help of \autoref{th:kadeishvili-deformed-coh}, we can also reformulate the condition $ D = 0 $ to a more intuitive statement. Let us distinguish between \emph{true} and \emph{actual cohomology}. By true cohomology of $ \cat C_q $, we mean the flat tensor products $ B \htensor \H\Hom_{\cat C} (X, Y) $. It only depends on the non-deformed category itself. By actual cohomology, we mean the directly observed cohomology $ \H(\Hom_{\cat C_q}, μ^1_{\cat C_q}) $. The two cohomologies typically differ:

\begin{center}
\begin{tabular}{@{}ccc@{}}
\morearraystretch
& \textbf{Original} & \textbf{Deformed} \\\hline
\textbf{Differential} & $ ℂ \overset{0}{→} ℂ $ & $ ℂ⟦q⟧ \overset{q}{→} ℂ⟦q⟧ $ \\
\textbf{Cohomology} & $ ℂ[0] ⊕ ℂ[1] $ & \parbox[m]{0.3\linewidth}{\begin{center} (actual) $ 0[0] ⊕ \frac{ℂ⟦q⟧}{(q)}[1] $ \\[0.5em] (true) $ ℂ⟦q⟧[0] ⊕ ℂ⟦q⟧[1] $ \end{center}}
\end{tabular}
\end{center} 

\autoref{th:kadeishvili-deformed-coh} quantifies the difference between true and actual cohomology. True and actual cohomology of $ \cat C_q $ are equal if $ D = 0 $ and fail to be canonically equal if $ D \neq 0 $. Among curvature-free deformations, we can summarize our observations without any claim to rigor very roughly as follows:

\begin{center}
\begin{tikzpicture}
\path (0, 0) node[align=center] {Classical Kadeishvili applies to $ \cat C_q $} (4, 0) node {\Large $ \Longleftrightarrow $} (8, 0) node[align=center] {$ \H\Hom_{\cat C_q} (X, Y) $ projective $ B $-modules};
\begin{scope}[shift={(0, -0.75)}]
\path (-4, 0) node {\Large $ \Longleftrightarrow $} (0, 0) node {true cohomology = actual cohomology} (4, 0) node {\Large $ \Longleftrightarrow $} (6, 0) node {$ D = 0 $} (8, 0) node {\Large $ \Longleftrightarrow $} (10, 0) node {$ μ^1_{\H\cat C_q} = 0 $.};
\end{scope}
\end{tikzpicture}
\end{center}